\documentclass[preprint]{elsarticle}

\def\e{{\epsilon}}

\def\norm#1{\|#1\|} 

\usepackage{algorithm}
\usepackage{algorithmic}
\usepackage{amsmath}
\usepackage{amsthm}
\usepackage{amssymb}
\usepackage{caption}
\usepackage{comment}
\usepackage{graphicx}
\usepackage{float}
\usepackage[latin1]{inputenc}
\usepackage[pdfborder={0 0 0}, colorlinks=true, linkcolor=blue, citecolor=red]{hyperref}\hypersetup{pdfborder=0 0 0}
\usepackage{mathrsfs}
\usepackage{multicol}
\usepackage{multirow}
\usepackage{natbib}
\setcitestyle{numbers}
\usepackage{wrapfig}
\usepackage{caption}
\usepackage{subcaption}
\usepackage{xcolor}
\usepackage{lineno,hyperref}
\modulolinenumbers[5]
\usepackage[textsize=footnotesize]{todonotes}
\usepackage{siunitx} 
\usepackage{xargs}
\usepackage{euscript}

\graphicspath{{./figures/}}

\definecolor{cred}{RGB}{139,37,0}
\definecolor{cblue}{rgb}{0.0, 0.45, 0.73}
\definecolor{cgreen}{rgb}{0.2, 0.8, 0.2}
\definecolor{cyellow}{rgb}{1.0, 0.97, 0.0}
\definecolor{corange}{rgb}{1.0, 0.8, 0.0}

\def\genbox#1#2#3#4#5#6{
    \leavevmode\raise#4bp\hbox to#5bp{\vrule height#5bp depth0bp width0bp
    \pdfliteral{q .5 w \csname #2COLOR\endcsname\space RG
                       \csname #3PDF\endcsname{#5}{#6} S Q
             \ifx1#1 q \csname #2COLOR\endcsname\space rg 
                       \csname #3PDF\endcsname{#5}{#6} f Q\fi}\hss}}

\def\circbox    #1#2{\genbox{#1}{#2}  {circ}     {0}   {5}    {2.5}}

\newcommand{\beq} {\begin{equation}}
\newcommand{\eeq} {\end{equation}}
\newcommand{\bdm} {\begin{displaymath}}
\newcommand{\edm} {\end{displaymath}}
\newcommand{\bit}{\begin{itemize}}
\newcommand{\eit}{\end{itemize}}
\newcommand{\bde}{\begin{description}}
\newcommand{\ede}{\end{description}}
\newcommand{\bce}{\begin{center}}
\newcommand{\ece}{\end{center}}
\newcommand{\ben} {\begin{enumerate}}
\newcommand{\een} {\end{enumerate}}
\newcommand{\bea} {\begin{eqnarray}}
\newcommand{\eea} {\end{eqnarray}}
\newcommand{\barr} {\begin{array}}
\newcommand{\earr} {\end{array}}
\newcommand{\bean} {\begin{eqnarray*}}
\newcommand{\eean} {\end{eqnarray*}}
\newcommand{\edoc} {\end{document}}

\newcommand{\dd}[2] {\ensuremath{\frac{\partial {#1}}{\partial {#2}}}}
\newcommand{\DD}[2] {\ensuremath{\frac{d {#1}}{d {#2}}}}

\newcommand{\eval}[2][\right]{\relax
  \ifx#1\right\relax \left.\fi#2#1\rvert}

\providecommand{\norm}[1]{\left\lVert#1\right\rVert}

\newcommand{\half} {\ensuremath{\frac{1}{2}}}

\newcommand{\mc}[1]{\mathcal{#1}}

\newcommand{\mb}[1]{\mathbf{#1}}

\newcommand{\nort}[1]{{\left\vert\kern-0.25ex\left\vert\kern-0.25ex\left\vert #1 
    \right\vert\kern-0.25ex\right\vert\kern-0.25ex\right\vert}}

\newcommand{\LRp}[1]{\left( #1 \right)}
\newcommand{\LRs}[1]{\left[ #1 \right]}
\newcommand{\LRa}[1]{\left< #1 \right>}
\newcommand{\LRc}[1]{\left\{ #1 \right\}}

\newcommand{\jump}[1] {\ensuremath{\LRs{\!\!\LRs{{#1}}\!\!}}}

\newcommand{\average}[1] {\ensuremath{\LRc{\!\!\{#1\}\!\!}}}

\newcounter{tablecell}

\newcommand{\figlab}[1]{\label{fig:#1}}
\newcommand{\eqnlab}[1]{\label{eq:#1}}
\newcommand{\theolab}[1]{\label{theo:#1}}

\newcommand{\tablab}[1]{\label{tab:#1}}

\newcommand{\figref}[1]{\ref{fig:#1}}

\newcommand{\eqnref}[1]{\eqref{eq:#1}}
\newcommand{\alglab}[1]{\label{alg:#1}}
\newcommand{\algref}[1]{\ref{alg:#1}}
\newcommand{\seclab}[1]{\label{sect:#1}}
\newcommand{\secref}[1]{\ref{sect:#1}}
\newcommand{\tabref}[1]{\ref{tab:#1}}

\renewcommand{\v}{v}

\newcommand{\sign}[1]{\text{ sgn}\!\LRp{#1}}

\newcommand{\R}{{\mathbb{R}}}

\newcommand{\f}{f}

\newcommand{\pOmega}{\partial \Omega}

\newcommand{\K}{K}
\newcommand{\Kp}{K^+}
\newcommand{\Km}{K^-}
\newcommand{\pK}{{\partial \K}}

\newcommand{\Poly}{{\mc{P}}}

\newcommand{\V}{{V}}

\newcommand{\Vh}{{\V_h}}

\newcommand{\VhK}{{\V_h\LRp{\K}}}

\newcommand{\n}{\mb{n}}
\newcommand{\nm}{\n^-}
\newcommand{\np}{\n^+}

\newcommand{\interp}[2]{{\mathbb{I}^{#1} \left( #2 \right)}}
\newcommand{\interpN}[1]{{\interp{N}{#1}}}

\newcommand{\fhat}{{\widehat{\f}}}

\newcommand{\env}{{\varphi}}

\newcommand{\Nbh}{ {\EuScript{N} } }

\newcommand{\averageM}[1]{\ensuremath{\LRc{\hspace{-2pt}\LRc{#1}\hspace{-2pt}}}}

\newcommand{\pOmegah}{{\pOmega_h}}
\newcommand{\Omegah}{{\Omega_h}}

\newcommand{\tila}{\tilde{a}}
\newcommand{\tilb}{\tilde{b}}
\newcommand{\tilc}{\tilde{c}}

\newcommand{\nb}{{\bf n}}

\newcommand{\Scal}{\eta}
\newcommand{\Fcal}{\mathcal{F}}

\newcommand{\dtt}{\triangle t}

\newcommand{\dt}{{\triangle t}}

\newtheorem{theorem}{Theorem}[section]
\newtheorem{corollary}{Corollary}[theorem]
\newtheorem{proposition}[theorem]{Proposition}

\newtheorem{remark}{Remark}

\newcommandx{\question}[2][1=]{\todo[linecolor=red,backgroundcolor=red!25,bordercolor=red,#1]{#2}}
\newcommandx{\change}[2][1=]{\todo[linecolor=blue,backgroundcolor=blue!25,bordercolor=blue,#1]{#2}}
\newcommandx{\add}[2][1=]{\todo[linecolor=OliveGreen,backgroundcolor=OliveGreen!25,bordercolor=OliveGreen,#1]{#2}}
\newcommandx{\improve}[2][1=]{\todo[linecolor=orange,backgroundcolor=orange!25,bordercolor=orange,#1]{#2}}
\newcommandx{\remove}[2][1=]{\todo[linecolor=yelllow,backgroundcolor=yellow!10,bordercolor=red,#1]{#2}}
\begin{document}

\begin{frontmatter}

\title{Entropy-Preserving and Entropy-Stable Relaxation IMEX and Multirate Time-Stepping Methods}
\author[Argonne]{Shinhoo Kang\corref{mycorrespondingauthor}}
\cortext[mycorrespondingauthor]{Corresponding author}
\ead{shinhoo.kang@anl.gov}
\author[Argonne]{Emil M. Constantinescu}
\ead{emconsta@anl.gov}

\address[Argonne]{Mathematics and Computer Science Division, Argonne National Laboratory, Lemont, IL, USA}

\begin{abstract}

We propose entropy-preserving and entropy-stable partitioned Runge--Kutta (RK) methods.
In particular, we extend the explicit relaxation Runge--Kutta methods 
to IMEX--RK methods and a class of explicit second-order multirate methods for stiff problems 
arising from scale-separable or grid-induced stiffness in a system.
The proposed approaches not only mitigate system stiffness
 but also fully support entropy-preserving and entropy-stability properties at a discrete level.
The key idea of the relaxation approach is to adjust the step completion with 
a relaxation parameter so that the time-adjusted solution satisfies the entropy condition at a discrete level. 
The relaxation parameter is computed by solving a scalar nonlinear equation at each timestep in general;
however, as for a quadratic entropy function, 
we theoretically derive the explicit form of the relaxation parameter and numerically confirm that the relaxation parameter works the Burgers equation. 
Several numerical results for ordinary differential equations and the Burgers equation 
are presented to demonstrate the entropy-conserving/stable behavior of these methods. 
We also compare the relaxation approach and the incremental direction technique for the Burgers equation with and without a limiter in the presence of shocks.
  
\end{abstract}

\begin{keyword}
entropy conservation/stability \sep discontinuous Galerkin \sep implicit-explicit \sep multirate integrator \sep Burgers equation 
\end{keyword}

\end{frontmatter}


\section{Introduction}


High-order methods for solving partial differential equations are popular because of their high-order accuracy and low numerical dissipation and dispersion errors, compared with low-order schemes \cite{ainsworth2004dispersive}.  
In terms of numerical robustness, however,
the low-order schemes are still an attractive choice for computational fluid dynamics 
because they are less prone to numerical instability than are high-order methods~\cite{chan2018discretely}. 
To this end, 
further stabilization techniques 
such as artificial viscosity, slope limiting, or filtering are needed in the vicinity of shocks or underresolved features. 

The entropy-conserving and entropy-stable methods are an alternative way to improve robustness
by satisfying the entropy condition at a discrete level. 
Tadmor~\cite{tadmor1987numerical} proposed entropy-conservative/stable finite-volume schemes,
which are extended to high-order methods \cite{jiang1994cell,nordstrom2005well,fjordholm2012arbitrarily,fisher2013high,gassner2013skew,carpenter2014entropy,fernandez2014generalized} with  
two important tools: the summation by parts (SBP) operator and flux-differencing techniques.
\footnote{
  The former mimics the integration by parts at a discrete level, 
   and the latter unveil the mechanism underlying the skew-symmetric formulation.
    The split forms consist of both conservative and nonconservative forms of equations such that 
    the aliasing errors caused by the volume integral terms become minimized.
}
In particular, 
several entropy-stable discontinuous Galerkin (DG) methods have been developed
with collocated points  on quadrilateral and hexagonal meshes~\cite{gassner2016split,wintermeyer2017entropy},  
on triangular meshes~\cite{chen2017entropy}, 
and with general points~\cite{chan2018discretely}
by a hybridized SBP operator. 

From a time discretization perspective, 
Nordstr{\"o}m and Lundquist in \cite{nordstrom2013summation} proposed SBP-based implicit time integrators to have fully discrete entropy-stable schemes.
The work in \cite{boom2015high,ranocha2021new} incorporated SBP in implicit Runge--Kutta (RK) methods. 
Friedrich et al.~\cite{friedrich2019entropy} proposed entropy-stable space-time methods. 
For entropy-stable explicit time integrators, 
Ketcheson~\cite{ketcheson2019relaxation} modified the step completion in 
standard Runge--Kutta methods to guarantee 
the square entropy conservation or stability, namely, $L^2$ stability, which are referred to as relaxation methods.
\footnote{
  Classical explicit RK or linear multistep methods cannot preserve general quadratic invariants \cite{ketcheson2019relaxation}.
}

The relaxation idea stems from the earlier works of Sanz-Serna and Manoranjan~\cite{sanz1982explicit,sanz1983method}, which modified the time step size of the Leapfrog scheme for the {\it Korteweg--de Vries} equation
and nonlinear {\it Schr\"{o}dinger} equations such that the quadratic invariant is conserved at a fully discrete level.  
Ketcheson \cite{ketcheson2019relaxation} revisited this relaxation idea and developed 
relaxation Runge--Kutta methods that guarantee conservation or stability for any inner-product norm. 
Relaxation methods have been further extended to 
the multistep methods \cite{ranocha2020general} and  deferred correction methods \cite{abgrall2021relaxation}  
and studied for Hamiltonian problems \cite{ranocha2020relaxation}, 
compressible Euler, and Navier--Stokes equations \cite{ranocha2020relaxationEuler}.
    
Inspired by the work in \cite{ketcheson2019relaxation}, 
we propose the relaxation methods for partitioned RK methods to tackle stiff problems.
Specifically, we extend the relaxation RK methods to IMEX Runge--Kutta (IMEX RK) 
and the second--order multirate Runge--Kutta (MRK2) method \cite{constantinescu2007multirate}.
Chemical kinetics \cite{stone2013techniques}, 
biochemical reactions \cite{komori2014stochastic},
electrical circuits \cite{bartel2002multirate}, and 
fluid mechanics \cite{kang2019imex} are all examples of  
stiff problems in many engineering and scientific applications.
Partitioned Runge--Kutta (RK) methods define a class of integrators that use different time-stepping algorithms for different problem components. The aim of these methods is to avoid a monolithic algorithm when the problem at hand has components with different dynamical properties, which may require suitable treatment for computational efficiency. Two of the most popular partitioned RK methods are implicit-explicit (IMEX)~\cite{ascher1997implicit} and multirate~\cite{constantinescu2007multirate}.  

IMEX schemes are widely used 
in multiscale problems including atmospheric~\cite{giraldo2010semi,gardner2018implicit}, ocean~\cite{newman2016communication}, 
sea-ice~\cite{lemieux2014second}, shallow-water~\cite{kang2019imex}, and wind turbine models~\cite{streiner2008coupled} and in plasma simulations~\cite{miller2019imex}. 
By treating the fastest waves implicitly, 
IMEX methods overcome the stringent time step size of explicit methods and simplify the fully implicit system solves by using an explicit integrator for the nonstiff components. 
IMEX methods can also handle geometric-induced stiffness arising from mesh refinement 
by treating the fine-grid solution implicitly~\cite{kanevsky2007application}. 
Similarly, 
multirate time integrators are a good candidate to tackle the stiffness issues.
In multirate methods, an original problem is split into several subproblems, 
 allowing different time step sizes on each subproblem. 
\footnote{
  In IMEX methods, the same time step size is used for both sitff and nonstiff parts. 
}
Multirate methods are used 
in various applications such as atmospheric~\cite{SkamarockKlemp08,seny2013multirate} and 
air pollution models~\cite{schlegel2012implementation},
the Burgers equation \cite{constantinescu2007multirate},
Euler equations \cite{wensch2009multirate}, and 
compressible Navier--Stokes equations \cite{mikida2019multi}. 

Our proposed approaches not only alleviate the stiffness in a system 
but also provide entropy-preserving and entropy-stability properties at a fully discrete level. 
While the relaxation method is a straightforward step correction procedure,
users will benefit from having a formula for partitioned Runge--Kutta methods.
The presented methods could be a viable option to improve 
the robustness of stiff simulations.
In particular, our contributions in this paper are as follows.
\begin{enumerate}
  \item We derive the entropy-conserving/stable conditions for relaxation IMEX methods and provide an explicit relaxation expression for a class of IMEX methods. 
  \item We provide a similar result as above for partitioned multirate Runge--Kutta methods. 
  \item We demonstrate entropy stability and inner-product-based conservation on several numerical examples that employ IMEX methods for problems with stiff components and explicit multirate for problems with variable dynamical scales.
\end{enumerate}


This paper is organized as follows.
In Section \secref{model-problems} we describe the model problems and the one-dimensional entropy-conserving/stable discontinuous Galerkin spectral element method \cite{gassner2013skew}.
In Section \secref{ESTimeSplitting} we introduce the entropy-conserving/stable IMEX and MRK2 methods 
and provide a novel analysis for the relaxation parameters of the IMEX-RK and MRK2 methods. 
In Section \secref{NumericalResults} we demonstrate the total mass conservation and  
the entropy conservation/stability of the proposed methods through numerical examples. 
Specifically, for the Burgers equation, 
we compare the relaxation approach and the incremental direction technique 
with and without a limiter in the vicinity of a shock. 
In Section \secref{Conclusion} we present our conclusions. 
 

\subsection{Problem Statement}

Underresolved solutions cause aliasing errors, which can trigger numerical instability. This often happens when sharp gradient solutions are developed with insufficient spatial and temporal resolutions. 
 One idea to maintain stability is to conserve or bound a quantity called entropy at a discrete level, which is a convex functional of the solution. Moreover, some applications require quadratic invariants preservation. This is not possible by directly using methods with explicit partitions. Relaxation methods have been proposed for monolithic, that is, single-partitioned (explicit), methods to overcome this limitation. This study extends the relaxation concept to two different classes of partitioned Runge--Kutta methods.



\subsection{Model Problems and Spatial Discretization Methods}
\seclab{model-problems}

We introduce notation and model problems, along with a choice for the spatial discretization, making the presentation of the new time-stepping algorithms easier to follow.




\subsubsection{Ordinary Differential Equation: Conserved Exponential Entropy }

We consider 
the ordinary differential equation (ODE) example introduced in \cite{ranocha2020relaxationEuler}:
\begin{align}
  \eqnlab{ode-expo-entropy}
  \DD{}{t}
  \begin{pmatrix}
    q_1\\
    q_2
  \end{pmatrix}
  = 
  \begin{pmatrix}
    - \exp (q_2)\\
    \exp (q_1)
  \end{pmatrix}. 
\end{align}
This system preserves the exponential entropy of form
$$ \eta (q) = \exp (q_1) + \exp (q_2).$$

\subsubsection{Ordinary Differential Equation: Nonlinear Pendulum}

We also consider the nonlinear pendulum 
described by the first-order ODE system 
\begin{align}
  \eqnlab{ode-pendulum-entropy}
  \DD{}{t}
  \begin{pmatrix}
    q_1\\
    q_2
  \end{pmatrix}
  = 
  \begin{pmatrix}
    - \sin (q_2)\\
    q_1
  \end{pmatrix}\,,
\end{align} 
with initial condition $q=(1.5,0)^T$ and entropy function 
$\eta(q) =  0.5 q_1^2 -\cos (q_2)$.

\subsubsection{Partial Differential Equation: The Burgers Equation}

We consider the inviscid Burgers equation on the time and space interval $(t,x) \in [0,T]\times\Omega$:
\begin{align}
\eqnlab{pde-burgers-eq}
 \dd{q}{t} + \half\dd{q^2}{x} 
 = 0 \text{ in } [0,T]\times\Omega,
\end{align}
where $q$ is a scalar quantity and $\Omega\subset \R$ is the one-dimensional domain. 
When considering implicit-explicit methods, we will split the spatial operator in two  
by  defining a linearized flux $F_L$ of $F :=\half q^2$ by 
\begin{align*}
  F_L:= \tilde{q} q\,,
\end{align*}
which will be treated implicitly, and the remaining nonlinear flux $\Fcal_{\mc{N}}$ 
\[
  F_{\mc{N}} := \frac{q^2}{2} - \tilde{q}q\,,
\]
with a reference state $\tilde{q}$ (for example, $\tilde{q} = q_n$: the numerical solution at $t_n$), which will be treated explicitly. 
We can now write \eqnref{pde-burgers-eq} as the partitioned problem
\begin{align}
\eqnlab{pde-burgers-eq-split}
 \dd{q}{t} 
 + 
  \dd{}{x}\underbrace{ \LRp{ \tilde{q} q }
  }_{F_L}
   + 
   \dd{}{x}\underbrace{ \LRp{ \frac{q^2}{2} - \tilde{q}q } 
   }_{F_{\mc{N}}} = 0
  \text{ in } \Omega\,.
\end{align}

\subsubsection{Discontinuous Galerkin  Spatial Discretization}

We denote by $\Omegah := \cup_{\ell=1}^{N_E} K_{\ell}$ the mesh containing a finite
collection of non-overlapping elements, $K_{\ell}$, that partition $\Omega$,  
where $N_E$ is the total number of elements.
Let $\pOmega_h := \LRc{\pK:\K\in \Omegah}$ be the collection of the boundaries of all elements.
For two neighboring elements $\Kp$ and $\Km$ that share an interior
interface $\e = \Kp \cap \Km$, we denote by $q^\pm$ the trace of the
solutions on $\e$ from $\K^\pm$.  
We define $\nm$ as the unit outward normal vector on
the boundary $\pK^-$ of element $\Km$ and $\np = -\nm$ as the unit outward
normal of a neighboring element $\Kp$. 
On the interior interfaces $\e$
, we define the mean/average operator $\averageM{v}$,
where $v $ is 
a scalar quantity, by
$\averageM{v}:=\half\LRp{v^- + v^+}$,
and the jump operator 
$\jump{v}:= v^+ \nb^+ + v^- \nb^-$. 

Let ${\Poly^N} \LRp{D}$ denote the space of polynomials of degree at
most $N$ on a domain $D$. Next, we introduce the following discontinuous piecewise polynomial space as
\begin{align*}
\Vh\LRp{\Omega_h} &:= \LRc{v \in L^2\LRp{\Omega_h}:
  \eval{v}_{\K} \in \Poly^N\LRp{\K}, \forall \K \in \Omega_h},
\end{align*}
and similar spaces $\VhK$ by 
replacing $\Omega_h$ with $\K$.
We define $\LRp{\cdot,\cdot}_\K$ as the $L^2$-inner product on an
element $\K$, 
and $\LRa{\cdot,\cdot}_{\pK}$ as the
$L^2$-inner product on the element boundary $\pK$.
We also define the 
inner products as
$\LRp{\cdot,\cdot}_{\Omega_h} :=
\sum_{\K\in \Omega_h}\LRp{\cdot,\cdot}_\K$ and
$\LRa{\cdot,\cdot}_{\pOmega_h} :=
\sum_{\pK\in \pOmega_h}\LRa{\cdot,\cdot}_\pK$.
We define associated norms as 
$\norm{\cdot}:= \norm{\cdot}_{\Omegah}:= \LRp{ \sum_{K\in \Omegah} \norm{\cdot}_{K}^2 }^\half$ ,
where $\norm{\cdot}_{\K}=\LRp{\cdot,\cdot}_{\K}^\half$. 
 
The entropy-conserving/stable DG skew-symmetric formulation \cite{gassner2013skew} of \eqnref{pde-burgers-eq} is as follows:
Seek $q_h \in \Vh\LRp{\Omegah}$ such that 
  \begin{align}
     \eqnlab{gov-burgers1d-u}
   \LRp{\dd{q_h}{t},v}_\Omegah 
     &:= \LRa{\mc{R}(q_h),v},
  \end{align}
  where
  \begin{align*}
    \LRa{\mc{R}(q_h),v} &:= 
    \frac{2}{3} \LRp{\interpN{\frac{q_h^2}{2}},\dd{\v}{x}}_\Omegah \\
 &- \frac{1}{6}  \LRp{\LRp{\interpN{q_h\dd{q_h}{x}},\v}_\Omegah  -  \LRp{q_h,\dd{\interpN{q_h\v}}{x}}_\Omegah }
  - \LRa{\nb \widehat{ \frac{q_h^2}{2}}, \v}_\pOmegah
  \end{align*}
for all $v \in \Vh\LRp{\Omega_h}$. 
Here, $q_h$ is a polynomial approximation to $q$ on each element $K$;
that is, for $x \in K$, 
$
q \approx q_h := \sum_{j=0}^N q_j (t) \ell_j (x)
$
with nodal values of $u_j=u(x_j)$ and Lagrange basis function 
$\ell_j=\Pi_{k=0,k\ne j}^N \LRp{\frac{x - x_j}{x_k - x_j} }$
satisfying $\ell_j(x_i) = \delta_{ij}$ (for $j=0,\cdots, N$); 
$\interpN{}$ is the interpolation operator such that 
$ 
\interpN{f}:= \sum_{j=0}^N f(x_j) \ell_j(x)
$; and $ \fhat=\widehat{ \frac{q_h^2}{2}}$ is a numerical flux. 

For the semi-discrete entropy-conserving formulation, we take the entropy-conserving flux, 
$$\fhat_{EC} := \frac{1}{6}\LRp{(q_h^+)^2 + (q_h^-)^2 + q_h^- q_h^+};$$ and 
for the semi-discrete entropy-stable formulation, we use the Lax--Friedrichs flux, 
$$\fhat_{ES} := \averageM{\frac{q_h^2}{2} } + \frac{\tau}{2} \jump{q_h} $$ 
with $\tau:=\max (|q_h^+|,|q_h^-|)$. 
The Lax--Friedrichs flux with the skew-symmetric formulation yields the energy-stable DG method \cite{gassner2013skew}.

The split form of the energy-conserving/stable DG weak formulation of \eqnref{pde-burgers-eq-split} is as follows: seek $q_h \in \Vh\LRp{\Omegah}$ such that 
\begin{align}
 \eqnlab{gov-burgers1d-split}
 \LRp{\dd{q_h}{t},v}_\Omegah 
   &:= \LRa{Lq_h,v} + \LRa{\mc{N}(q_h),v},
\end{align}
where
\begin{align*}
  \LRa{Lq_h,v} &:= - \half\LRp{\interpN{\dd{\tilde{q}_h q_h }{x}},\v}_\Omegah
  + \half\LRp{\interpN{\tilde{q}_h q_h},\dd{\v}{x}}_\Omegah \\
  &- \LRa{\nb \LRp{\widehat{\tilde{q}_h q_h }-\half \tilde{q}_h q_h}, \v}_\pOmegah, \\
  \LRa{\mc{N}(q_h),v} &:=\LRa{\mc{R}(q_h),v}  - \LRa{Lq_h,v},
\end{align*}
for all $v \in \Vh\LRp{\Omega_h}$. We take 
$ \widehat{\tilde{q}_h q_h} := \average{\tilde{q}_h q_h}$ for the entropy-conserving flux and
$ \widehat{\tilde{q}_h q_h} := \average{\tilde{q}_h q_h} + \half \max(|\tilde{q}_h^\pm|)\jump{
  q_h}$ for the Lax--Friedrichs flux. 
  The reference state $\tilde{q}_h$ is taken as the elementwise mean value of $q_h$ at $t_n$, 
    so that $\tilde{q}_h$ becomes a constant on each element. 
%

\subsubsection{Flux Limiters}


Entropy-conserving/stable schemes are provably stable, 
but it is still not enough to eliminate high-frequency oscillations near a shock region.  
To control the Gibbs phenomenon, we employ 
the limiter proposed by \cite{chen2017entropy}. 
The idea is to construct a linear function based on 
the two modified 
left and right values, $\check{q}^{l}$ and $\check{q}^{r}$ for each element,  
\begin{subequations}
  \eqnlab{shock-limiter-burgers-cs17}
  \begin{align} 
    \check{q}^{l} & = \overline{q_h}_K - m\LRp{\overline{q_h}_K - q^{l},\overline{q_h}_{K+1} - \overline{q_h}_{K}, \overline{q_h}_{K} - \overline{q_h}_{K-1} },\\
    \check{q}^{r} & = \overline{q_h}_K + m\LRp{q^{r} - \overline{q_h}_K,\overline{q_h}_{K+1} - \overline{q_h}_{K}, \overline{q_h}_{K} - \overline{q_h}_{K-1} },\\
    \check{q} &= \overline{q_h}_K + (q_K-\overline{q_h}) \LRp{ \frac{\check{q}^{l} + \check{q}^{r} - 2 \overline{q_h}_K }{ q^{l} + q^{r}  - 2 \overline{q_h}_K  }  }, 
  \end{align}  
\end{subequations}
where $q^{l}:=q_h(x_0)$ and $q^{r}:=q_h(x_N)$ 
are the leftmost and the rightmost values on the $K$th element,  respectively; $\overline{q_h}_K$ is the mean value on the $K$th element; 
and 
$m$ is the minmod function defined by  
\begin{align*}
m(a,b,c) = \begin{cases} s \min (|a|,|b|,|c|), ~\textnormal{if}~ s=\sign{a}=\sign{b}=\sign{c}\\ 0, ~\textnormal{otherwise}~ \end{cases}.
\end{align*}
Once a solution is integrated by one time step, 
we apply the limiter to the updated solution as a postprocessing task.

\section{Entropy-Stable Time-Splitting Methods}
\seclab{ESTimeSplitting}
 
In this section we propose entropy-conserving/stable IMEX and multirate methods by using relaxation methods.

Given a scalar hyperbolic equation, 
\begin{align} 
\eqnlab{scala-hyperboliceq}
\dd{q}{t} + \dd{F}{x}= 0, 
\end{align}
where $x \in \Omega$ with $\Omega$ convex, we define a convex function $\Scal:\Omega \rightarrow \R $ called an entropy function if there exists the entropy flux $\Fcal$ satisfying $\dd{\Scal}{q} \dd{F}{x} = \dd{\Fcal}{x}$
and $\Fcal=\env F - \psi$. Here, $\env:=\dd{\Scal}{q}$ and $\psi$ are the entropy variable and potential flux, respectively.
We multiply the entropy variable to \eqnref{scala-hyperboliceq} and integrate it over the domain, arriving at 
 the tendency of the entropy function, 
 \begin{align*}
  \LRp{\dd{\Scal}{t},1}_\Omegah = - \LRa{\nb \Fcal,1}_\pOmegah.
\end{align*}
\footnote{Here, we have used a chain rule, 
$\dd{\Scal}{t} = \dd{\Scal}{q} \dd{q}{t} = \env \dd{q}{t}$. }
For a dissipative system, the entropy tendency should decrease: 
\begin{align}
\eqnlab{entropy-theory}
  \LRp{\dd{\Scal}{t},1}_\Omegah \le - \LRa{\nb \Fcal,1}_\pOmegah.
\end{align}
With periodic or compactly supported boundary conditions, the term on the right-hand side vanishes; hence, 
the semi-discrete entropy stability is guaranteed. 
For a fully discretized system, we expect 
$$
\LRp{\Scal (q_{h,n+1}),1}_\Omegah \le \LRp{\Scal(q_{h,n}),1}_\Omegah
$$
at a discrete level for a dissipative system; however, in practice 
the entropy condition is not guaranteed for all times. 
Here, $q_{h,n}$ and $q_{h,n+1}$ are approximations to $q$ at $t=t_n$ and $t=t_{n+1}$, respectively.
We will not include the subscript $\Omega_h$ in the inner product and $h$ in $q$ unless it is required explicitly.

\begin{remark}
  For the Burgers equation, 
  with the entropy function $\Scal=\frac{q^2}{2}$ and entropy flux $\Fcal=\frac{q^3}{3}$\cite{carpenter2013high}, 
the semi-discrete form in \eqnref{entropy-theory} yields 
$$ \half \DD{}{t}\norm{q}^2 \le \LRa{q, \mc{R}(q)}, $$
where  $\mc{R}(q)$ is in \eqnref{gov-burgers1d-u}. 
The entropy stability implies 
 $L^2$ stability. 
\end{remark}

\subsection{Relaxation Runge--Kutta Method}
The standard explicit RK methods are 
\begin{align*}
    Q_i &= q_n + \dt \sum_{j=1}^{i-1} a_{ij} R_j,\quad i=1,2,\cdots,s,
    \\ 
    q_{n+1} &= q_n + \dt \sum_{i=1}^s b_{i} R_i,
\end{align*}
where $R_i:=R(Q_i)$ and $a_{ij}$ and $b_i$ are scalar coefficients for $s$-stage RK methods. 
The basic idea of the relaxation Runge--Kutta method \cite{ketcheson2019relaxation,ranocha2020relaxationEuler} is to adjust the step completion with the relaxation parameter $\gamma$, effectively taking a modified step size such that the entropy stability is ensured.
The time-adjusted solution at $t_{n+\gamma}(=t_n+\gamma \dt)$ is
\begin{align*}
    q(t_n + \gamma \dt) \approx q_{n+\gamma} = q_n + \gamma \dt \sum_{i=1}^s \LRp{ b_{i} R_i } = \gamma q_{n+1} + (1-\gamma) q_n.
\end{align*}
\footnote{ 
  This is simply 
  a weighted sum of the current and the next step solutions. 
}
The change in the entropy from $t_n$ to $t_{n+\gamma}$ can be expressed as  
\begin{multline*} 
  \eta (q_{n+\gamma}) - \eta (q_n)
  = 
  \underbrace{ 
    \eta (q_{n+\gamma}) - \eta (q_n) 
    - \gamma \dt \sum_{i=1}^s b_i \LRp{R_i,\env_i}
  }_{:=\theta (\gamma)} \\
  + \gamma \dt \sum_{i=1}^s b_i \LRp{R_i,\env_i}
\end{multline*} 
with $\env_i:=\env (Q_i)$.
The last term on the right-hand side is smaller than or equal to zero, provided by $\gamma>0$ and $b_i\ge 0$. 
With a root $\gamma$ of $\theta(\gamma)=0$,
the total entropy is bounded, 
 $\eta (q_{n+\gamma}) \le \eta (q_n). $ 
Here the nonlinear scalar equation $\theta(\gamma)=0$ can be solved, for example, by 
Brent's method, the Levenberg--Marquard algorithm, or 
Newton's method \cite{ranocha2020relaxationEuler, abgrall2021relaxation}.

\begin{remark}
  The scheme using the $q(t_n + \dt) \approx q_{n+\gamma}$ approximation is referred to as an {\it incremental direction technique} (IDT) method \cite{calvo2006preservation}.
  The work in \cite[Theorem 2.7]{ketcheson2019relaxation} shows that the
  IDT method is one order less accurate than the relaxation approach. 
\end{remark}

\subsection{Relaxation IMEX Methods}
  
Consider a semi-discretized system, 
$$
\dd{q}{t} = R(q) = f(q) + g(q).
$$

Recall $s$-stage IMEX-RK methods \cite{ascher1997implicit,pareschi2005implicit,Kennedy2003additive, giraldo2013implicit},  
\begin{subequations}
   \eqnlab{IMEXRK}
  \begin{align}
   \eqnlab{IMEXRK-qi}
     Q_{n,i} &= q_n
       + \dtt\sum_{j=1}^{i-1}a_{ij} f_j
       + \dtt\sum_{j=1}^{i}\tilde{a}_{ij} g_j,\quad i=1,\hdots,s,\\
   \eqnlab{IMEXRK-qn}
     q_{n+1} &= q_n
       + \dtt\sum_{i=1}^{s}b_{i} f_i
       + \dtt\sum_{i=1}^{s}\tilde{b}_{i} g_i,
  \end{align}
\end{subequations}
where 
$f_i = f\LRp{t_n+c_i\dtt, Q_{n,i}}$, 
$g_i = g\LRp{t_n+\tilde{c}_i\dtt,Q_{n,i}}$,
$q_n = q(t_n)$;
$Q_{n,i}$ is the $i$th intermediate state;
and $\dtt$ is the time step size.
The scalar coefficients $a_{ij}$, $\tila_{ij}$, $b_i$, $\tilb_i$,
$c_i$, and $\tilc_i$ determine all the properties of a given IMEX-RK
scheme. 
For each stage, the intermediate state $Q_{n,i}$ is obtained in general by a nonlinear solve, 
\begin{align*}
    Q_{n,i} - \tila_{ii} \dt g_j = q_n + \dt \sum_{j=1}^{i-1} \LRp{a_{ij} f_j + \tila_{ij} g_j }. 
\end{align*}
\begin{remark}
A practical way to avoid the nonlinear solve is to linearize the flux. 
To that end, we define a linear operator $L(\tilde{q}):=\dd{R}{q}\big\vert_{\tilde{q}}$ 
and choose $f(q):=R(q)-Lq$ and $g(q):=Lq$, where $\tilde{q}$ can be $q_n$ or $Q_{n,i}$, $i=1,\dots,s$. 
Then at each stage the intermediate state $Q_{n,i}$ requires only a linear solve:
\begin{align}
    Q_{n,i} - \tila_{ii} \dt L_i Q_{n,i} = q_n + \dt \sum_{j=1}^{i-1} \LRp{a_{ij} N_j  + \tila_{ij} L_j Q_{n,j} }, 
    \eqnlab{ARK-Qni}
\end{align}
where 
$R_i=R(t_n+c_i\dt,Q_{n,i})$, $L_i=L(t_n+c_i\dt,\tilde{q})Q_{n,i}$, and 
$N_i=R_i - L_i$.
\end{remark}

The relaxation IMEX-RK methods adjust the final time step size by
\begin{align}
    \eqnlab{IMEX-qn-relaxation}
    q(t_n + \gamma \dt) \approx q_{n+\gamma} = q_n + \gamma \dt \sum_{i=1}^s \LRp{ b_{i} f_i + \tilde{b}_i g_i } = \gamma q_{n+1} + (1-\gamma) q_n.
\end{align}

Now, the change in the entropy from $t_n$ to $t_{n+\gamma}$ becomes
\begin{multline} 
\eqnlab{entropy-gamma-computation-imexrk}
  \eta (q_{n+\gamma}) - \eta (q_n)
  = 
  \underbrace{ 
    \eta (q_{n+\gamma}) - \eta (q_n) 
    - \gamma \dt \sum_{i=1}^s \LRp{b_i f_i + \tilde{b}_i g_i,\env_i}
  }_{:=\theta (\gamma)} \\
  + \gamma \dt \sum_{i=1}^s\LRp{ b_i f_i + \tilde{b}_i g_i,\env_i}.
\end{multline}

\begin{proposition}
\theolab{relax-imex-energy-conservation}
The relaxation IMEX-RK methods in \eqnref{IMEXRK-qi}, \eqnref{IMEXRK-qn},
and \eqnref{IMEX-qn-relaxation}
are entropy-conserving/stable 
with an entropy-conserving/stable spatial discretization of $f(q)$ and $g(q)$ and the relaxation parameter such that 
\begin{align}
    \eqnlab{imex-gamma}
\eta (q_{n+\gamma}) - \eta (q_n) 
    - \gamma \dt \sum_{i=1}^s \LRp{b_i f_i + \tilde{b}_i g_i,\env_i} = 0. 
\end{align}
In particular, for energy entropy $\eta(q)=\half \norm{q}^2$, 
the relaxation parameter is explicitly determined by 
\begin{align*}
\gamma = 2\norm{ q_{n+1} - q_n }^{-2} \dt \sum_{i=1}^s  \LRp{ b_i f_i + \tilde{b}_i g_i,Q_{n,i} - q_n  }. 
\end{align*}
\end{proposition}
 
\begin{proof}
By solving  \eqnref{imex-gamma} for $\gamma$, 
the first,  second, and  third terms in \eqnref{entropy-gamma-computation-imexrk} vanish. 
    With an entropy-conserving/stable spatial discretization of $f$ and $g$, 
    the last term in \eqnref{entropy-gamma-computation-imexrk} 
    becomes nonpositive, that is, 
$\LRp{f_i,\env_{i}} \le 0 $ and 
$\LRp{g_i,\env_{i}} \le 0$ for $i=1,\cdots,s$, and hence 
$\eta (q_{n+\gamma}) \le \eta (q_n)$. 
  
  By substituting $\eta$ with 
  the inner-product norm $\half\norm{q}^2$ and by using \eqnref{IMEXRK-qn}, 
\eqnref{imex-gamma} 
  can be written as 
\begin{multline*}    
    \norm{q_{n+\gamma}}^2 - \norm{q_{n}}^2 
    - 2\gamma \dt \sum_{i=1}^s \LRp{b_i f_i + \tilde{b}_i g_i,Q_{n,i}} \\
    = \norm{ \gamma\LRp{q_{n+1} - q_n} + q_n  }^2 - \norm{q_{n}}^2 
    - 2\gamma \dt \sum_{i=1}^s \LRp{b_i f_i + \tilde{b}_i g_i,Q_{n,i}} \\
    = \gamma^2 \norm{ q_{n+1} - q_n }^{2} + 2\gamma \LRp{q_{n+1}-q_n,q_n} 
    - 2\gamma \dt \sum_{i=1}^s \LRp{b_i f_i + \tilde{b}_i g_i,Q_{n,i}} \\
    = \gamma^2 \norm{ q_{n+1} - q_n }^{2}  
    - 2\gamma \dt \sum_{i=1}^s \LRp{b_i f_i + \tilde{b}_i g_i,Q_{n,i}-q_n} = 0.
\end{multline*}
Rearranging a nonzero $\gamma$ leads to the desired result.
\end{proof}
 

\begin{corollary}
\theolab{relax-ark-entropy-conservation}
The relaxation IMEX-RK methods with $b_i = \tilde{b}_i$ in \eqnref{IMEXRK-qi}, \eqnref{IMEXRK-qn},
and \eqnref{IMEX-qn-relaxation}
are entropy conserving/stable with an entropy-conserving/stable spatial discretization $R(q)=f(q)+g(q)$ and the relaxation parameter 
\begin{align}
    \eqnlab{ark-gamma-general}
\eta (q_{n+\gamma}) - \eta (q_n) 
    - \gamma \dt \sum_{i=1}^s b_i\LRp{R_i,\env_i} = 0. 
\end{align}
In particular, for the energy entropy $\eta(q)=\half \norm{q}^2$ and nonstationary solution, 
the relaxation parameter is explicitly determined by 
\begin{align*}
\gamma = 2\norm{ q_{n+1} - q_n }^{-2} \dt \sum_{i=1}^s  b_i\LRp{R_i,Q_{n,i} - q_n  }. 
\end{align*}
\end{corollary}

\subsection{Relaxation Multirate Runge--Kutta Method}

We apply the relaxation approach to the second-order multirate Runge--Kutta  method \cite{constantinescu2007multirate}.
The MRK2 method is based on a partitioned Runge--Kutta method 
where the second-order strong-stability-preserving Runge--Kutta \cite{shu1988efficient} serves as the base method; further details are given in \cite{constantinescu2007multirate}. 

Multirate methods can be applied in different contexts. To simplify the exposition and without the loss of generality, however, we focus here on geometric-induced stiffness. 
We consider that some parts of a domain are spatially refined with a fixed 2:1 balancing ratio; that is, the ratio of an element size to its adjacent element size is at most 2. 
In the following, we first consider a two-level decomposition and then  generalize the idea to an arbitrary-level domain decomposition.

\subsubsection{Two-Level Decomposition}

A domain is decomposed into two subdomains: coarse and fine regions with the ratio of a 2:1 grid size.
Depending on the grid size and the location, the fast, the buffer, and the slow zones are identified as shown in Figure \figref{mrk2-lev2-illustration}.  
The fine region is considered the fast zone. 
The coarse regions are composed of the buffer zone and the slow zone. 
The buffer next to the fast zone is the fast buffer,
 while the buffer next to the slow zone is the slow buffer.

\begin{figure} 
    \centering
    \includegraphics[trim=0.0cm 0.0cm 0.0cm 0.0cm, width=0.95\columnwidth]{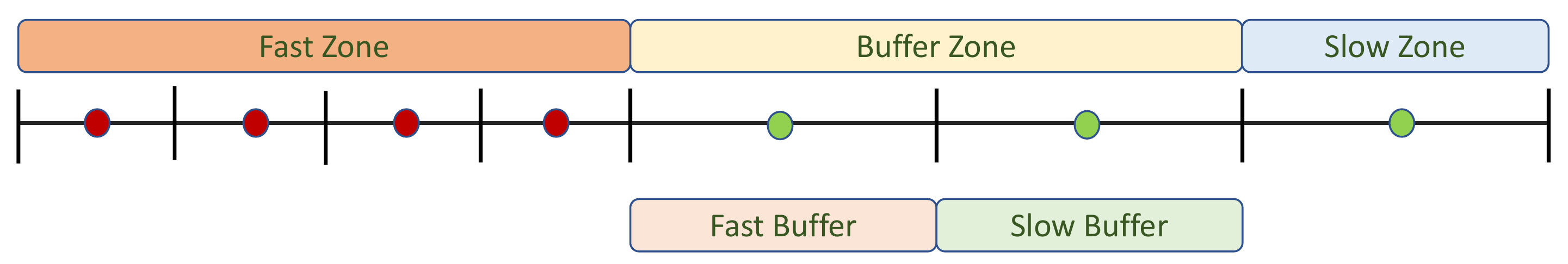}
    \caption{Illustration of MRK2 with a two-level decomposition: 
    a domain is decomposed into  fine (\circbox1{cred}) and  coarse regions (\circbox1{cgreen}). 
    Depending on the grid size and the location, the fast, the buffer, and the slow zones are identified. 
    The fine region is considered the fast zone. 
    The coarse regions are composed of the buffer zone and the slow zone. 
    The buffer next to the fast zone is the fast buffer,
    while the buffer next to the slow zone is the slow buffer.
    }
    \figlab{mrk2-lev2-illustration}
\end{figure}

Table \tabref{mrk2-butcher} shows the Butcher tableau for MRK2 with a two-level decomposition.
There are four global stages in all ($s=4$).
The solution on each element is updated depending on 
what region the element belongs to: the fast zone, the fast buffer, the slow buffer, and the slow zone. 
We assign zone number 1 for the fast zone, 2 for the fast buffer, 3 for the slow buffer, and 4 for the slow zone.
The intermediate states and the next step solution for each zone number $z$ are 
\begin{subequations}
    \begin{align}     
    \eqnlab{mrk2-qi-twolevel}
    Q_{n,i}^{\LRc{z}} &= q_n^{\LRc{z}} + \dt \sum_{j=1}^{i-1} a_{ij}^{\LRc{z}} R_j^{\LRc{z}}, \quad i=1,2,\cdots, s,\\
    \eqnlab{mrk2-qn-twolevel}
    q_{n+1}^{\LRc{z}} &= q_n^{\LRc{z}} + \dt \sum_{i=1}^s b_i^{\LRc{z}} R_i^{\LRc{z}} 
    \end{align} 
\end{subequations}
for $z=1,\cdots, 4$. 
  
\begin{table}[ht]
    \caption{Butcher tableau for MRK2 with a two-level decomposition. }
    \tablab{mrk2-butcher} 
    \begin{center}
    \begin{subtable}[c]{0.36\textwidth}    
        \begin{equation*}
            \begin{array}{c|cccc}
                  0 &     &     &     & \\
                1/2 & 1/2 &     &     & \\
                1/2 & 1/4 & 1/4 &     & \\
                  1 & 1/4 & 1/4 & 1/2 & \\
                \hline
                    & 1/4 & 1/4 & 1/4 & 1/4
            \end{array}    
        \end{equation*}
        \caption{Fast zone/buffer}
    \end{subtable}
    \begin{subtable}[c]{0.34\textwidth}    
        \begin{equation*}
            \begin{array}{c|cccc}
                0 &   &   &   & \\
                1 & 1 &   &   & \\
                0 & 0 & 0 &   & \\
                1 & 0 & 0 & 1 & \\
                \hline
                & 1/4 & 1/4 & 1/4 & 1/4
            \end{array}    
        \end{equation*}
        \caption{Slow buffer}
    \end{subtable}
    \begin{subtable}[c]{0.26\textwidth}    
        \begin{equation*}
            \begin{array}{c|cccc}
                0 &   &   &   & \\
                1 & 1 &   &   & \\
                0 & 0 & 0 &   & \\
                1 & 1 & 0 & 0 & \\
                \hline
                & 1/2 & 0 & 0 & 1/2
            \end{array}    
        \end{equation*}
        \caption{Slow zone}
    \end{subtable}
    \end{center}
\end{table}

The relaxation MRK2 for a two-level decomposition is 
\begin{align}
    \eqnlab{mrk2-qn-twolevel-relaxation}
    q^{\LRc{z}}(t_n + \gamma \dt) \approx q^{\LRc{z}}_{n+\gamma} = q^{\LRc{z}}_n + \gamma \dt \sum_{i=1}^s b_{i} R_i^{\LRc{z}} = \gamma q^{\LRc{z}}_{n+1} + (1-\gamma) q^{\LRc{z}}_n 
\end{align}
for $z=1,\cdots,4$.

\begin{proposition}
    \theolab{relax-mrk2-twolevel-energy-conservation}
    The relaxation MRK2 method for a two-level decomposition in \eqnref{mrk2-qi-twolevel}, \eqnref{mrk2-qn-twolevel},
    and \eqnref{mrk2-qn-twolevel-relaxation}
    are entropy conserving/stable with an entropy-conserving/stable spatial discretization $R$ and the relaxation parameter satisfying 
    \begin{align}
        \eqnlab{mrk2-twolevel-gamma}
    \eta (q_{n+\gamma}) - \eta (q_n) 
        - \gamma \dt \sum_{z=1}^4 \sum_{i=1}^s b_{i}^{\LRc{z}} \LRp{ R_i^{\LRc{z}},\env_{n,i}^{\LRc{z}}} = 0. 
    \end{align}
    For the quadratic invariant $\half\norm{q}^2$ and a dynamic solution ($q_{n+1}^{\LRc{z}} \ne q_n^{\LRc{z}}$), the relaxation parameter is explicitly determined: \begin{align*} 
    \gamma = 2 \LRp{\sum_{z=1}^4\norm{ q_{n+1}^{\LRc{z}} - q_n^{\LRc{z}} }^{2}}^{-1} 
    \LRp{(\dt)^2 \sum_{z=1}^4 \sum_{i=1}^s b_i^{\LRc{z}}  \LRp{
    R_i^{\LRc{z}},\frac{\LRp{Q_{n,i}^{\LRc{z}} - q_n^{\LRc{z}} } }{\dt}  } }. 
    \end{align*}
\end{proposition}

\begin{proof}
    The change in the entropy from $t_n$ to $t_{n+\gamma}$ becomes
    \begin{multline} 
    \eqnlab{entropy-gamma-computation-mrk2}
      \eta (q_{n+\gamma}) - \eta (q_n)
      = 
      \underbrace{ 
        \eta (q_{n+\gamma}) - \eta (q_n) 
        - \gamma \dt \sum_{z=1}^4 \sum_{i=1}^s b_{i}^{\LRc{z}} \LRp{ R_i^{\LRc{z}},\env_{n,i}^{\LRc{z}}}
      }_{:=\theta (\gamma)} \\
      + \gamma \dt \sum_{z=1}^4 \sum_{i=1}^s b_{i}^{\LRc{z}} \LRp{ R_i^{\LRc{z}},\env_{n,i}^{\LRc{z}}}.
    \end{multline} 

By solving \eqnref{mrk2-twolevel-gamma} for $\gamma$, 
the first,  second, and third terms in \eqnref{entropy-gamma-computation-mrk2} vanish just as in the IMEX case. 
    With an entropy-conserving/stable spatial discretization of $R$, 
    the last term in \eqnref{entropy-gamma-computation-mrk2} 
    becomes again nonpositive.
  
  By substituting $\eta$ with 
  the inner-product norm $\half\norm{q}^2$ and by using \eqnref{mrk2-qn-twolevel}, 
  \eqnref{mrk2-twolevel-gamma} becomes
  \begin{multline*}    
        \norm{q_{n+\gamma}}^2 - \norm{q_{n}}^2 
        - 2\gamma \dt \sum_{z=1}^4 \sum_{i=1}^s b_{i}^{\LRc{z}} \LRp{ R_i^{\LRc{z}},Q_{n,i}^{\LRc{z}}} \\
        = \gamma^2 \norm{q_{n+1} - q_n}^2 
        + 2\gamma \sum_{z=1}^4 \LRp{ q_{n+1}^{\LRc{z}} - q_{n}^{\LRc{z}} ,q_{n}^{\LRc{z}}} 
        - 2\gamma \dt \sum_{z=1}^4 \sum_{i=1}^s b_{i}^{\LRc{z}} \LRp{ R_i^{\LRc{z}},Q_{n,i}^{\LRc{z}}} \\
        = \gamma^2 \norm{q_{n+1} - q_n}^2 
        - 2\gamma \dt \sum_{z=1}^4 \sum_{i=1}^s b_{i}^{\LRc{z}} \LRp{ R_i^{\LRc{z}},Q_{n,i}^{\LRc{z}} - q_{n}^{\LRc{z}} } = 0.
    \end{multline*}
    Rearranging the last equation yields an explicit form for $\gamma$. 
\end{proof}

At each stage, communication occurs 
between the fast zone and the fast buffer and 
between the fast buffer and the slow buffer. 
However, communication happens only at the first and the last stages between the slow buffer and the slow zone. 
After exchanging the interface data at the fourth stage, 
the right-hand side of the slow buffer at the second stage is evaluated.
Based on this observation, 
we group the fast zone, the fast buffer, and the slow buffer by a level block 
that has four stages, which we call a cycle.
We will use the level block notation for multilevel decomposition in the next section.

\subsubsection{Beyond Two-Level Decomposition}

We start by defining a level block. 
A level block ($B$) is formed by consecutive elements with the same multirate level ($\ell$), 
each of which is assigned to a zone number ($z$). 
That means a level block ($B$) consists of a fast zone ($fz$), fast buffer ($fb$), and slow buffer ($sb$). 
\footnote{
    We view the fast zone of  level $\ell$ as the slow zone with respect to  level $\ell+1$. 
    For instance, the level blocks with level 0 and level 1 in Figure \figref{mrk2-lev3-illustration} 
    correspond to the fast zone, the buffer zone, and the slow zone in Figure \figref{mrk2-lev2-illustration}.
}
A level block ($B$) can have a neighbor level block ($\Nbh(B)$) that has $\ell\pm1$ multirate level. 
We let $L_{\max}$ be the maximum level and $0$ be the minimum (root) level. 
We let $s_G:=2^{L_{\max}+1}$ be the total number of global stages 
and let $s^{\LRc{0}}=2$ and $s^{\LRc{\ell(B)}}=4$ (for $\ell = 1,2,\cdots$) be the total number of local stages of a level block ($B$).
We also let $m(B):=2^{\LRc{\ell(B)}-1}$ 
be the number of substeps and $\dt^{\LRc{\ell(B)}}:= \LRp{\half}^{\ell(B)-1}\dt$ be the local time step size of a level block ($B$)
so that $m(B) \dt^{\LRc{\ell(B)}} = \dt$ if $\ell(B)>0$.  
When $\ell(B)=0$, we take $m(B)=1$. 
We assume that each fast buffer ($fz$) and slow buffer ($sb$) consist of one element. 
 
The intermediate states and the next step solution of a level block ($B$) and a zone number $(z)$ 
are written as
\begin{subequations}
    \eqnlab{mrk2-multilevel}
    \begin{align}     
    \eqnlab{mrk2-qi-multilevel}
    Q_{n+\frac{k}{m(B)},i }^{\LRc{B,z}} 
        &=  q_{n+\frac{k}{m(B)}}^{\LRc{B,z}} 
         + \dt^{\LRc{\ell(B)}} \sum_{j=1}^{i-1} a_{ij}^{\LRc{z}} R_{n+\frac{k}{m(B)},j}^{\LRc{B,z}}, \\
    \eqnlab{mrk2-qn-multilevel}   
    q_{n+1 }^{\LRc{B,z}} 
    &= q_{n}^{\LRc{B,z}}  
        + \dt^{\LRc{\ell(B)}} 
        \sum_{r=1}^{m(B)}\sum_{i=1}^{s^{\LRc{\ell(B)}}} b_{i}^{\LRc{z}} R_{n+\frac{r-1}{m(B)},i}^{\LRc{B,z}}, 
    \end{align} 
\end{subequations}
where 
\begin{align*}
q_{n+\frac{k}{m(B)}}^{\LRc{B,z}}  
        &=q_{n}^{\LRc{B,z}} + \dt^{\LRc{\ell(B)}} \sum_{r=1}^{k}\sum_{i=1}^{s^{\LRc{\ell(B)}}} b_{i}^{\LRc{z}} R_{n+\frac{r-1}{m(B)},i}^{\LRc{B,z}},    
\end{align*}
for $k=0,1,\cdots,m(B)-1$ and $i=1,\cdots, s^{\LRc{\ell(B)}}$. 

At the first global stage, all level blocks are activated, 
which means that the intermediate states of all the level blocks are updated 
and exchanged between adjacent active level blocks. 
At the second and the third global stages, the level blocks that have the maximum level $L_{\max}$ are activated. 
At the fourth global stage, the level blocks that have $L_{\max}$ and $L_{\max-1}$ levels are activated. This implies that after one cycle, these level blocks are synchronized. This process is repeated until all the level blocks are synchronized at the last global stage, $s_G$. We construct the activation table in Algorithm \algref{MRK2-acvtable} to control the synchronization.
That is, according to the activation table, certain level blocks are activated at a given global stage. 
 
\begin{algorithm}[h!t!b!]
  \begin{algorithmic}[1]
    \ENSURE Given the maximum level ($L_{\max}$), construct the activation table ($actvTable$) of the size $s_G \times N_B$. Here, $N_B$ is the total number of level blocks. 
    
    \STATE{$ actvTable[:,:]=0$ }
    \FOR{$B$ in $\LRc{1:N_B}$}
      \STATE $nActv \gets 2^{\ell(B)}$
      \STATE $d \gets 2^{(L_{\max}+1-\ell(B))}$
      \FOR {$i=1:nActv$}
          \STATE {$actvTable[~1   + d(i-1),B] \gets 1$}
          \STATE {$actvTable[s_G - d(i-1),B] \gets 1$}
      \ENDFOR
    \ENDFOR
  \end{algorithmic}
  \caption{Activation Table for Level Blocks}
  \alglab{MRK2-acvtable}
\end{algorithm}


\begin{figure} 
    \centering
    \includegraphics[trim=0.0cm 0.0cm 0.0cm 0.0cm, width=0.95\columnwidth]{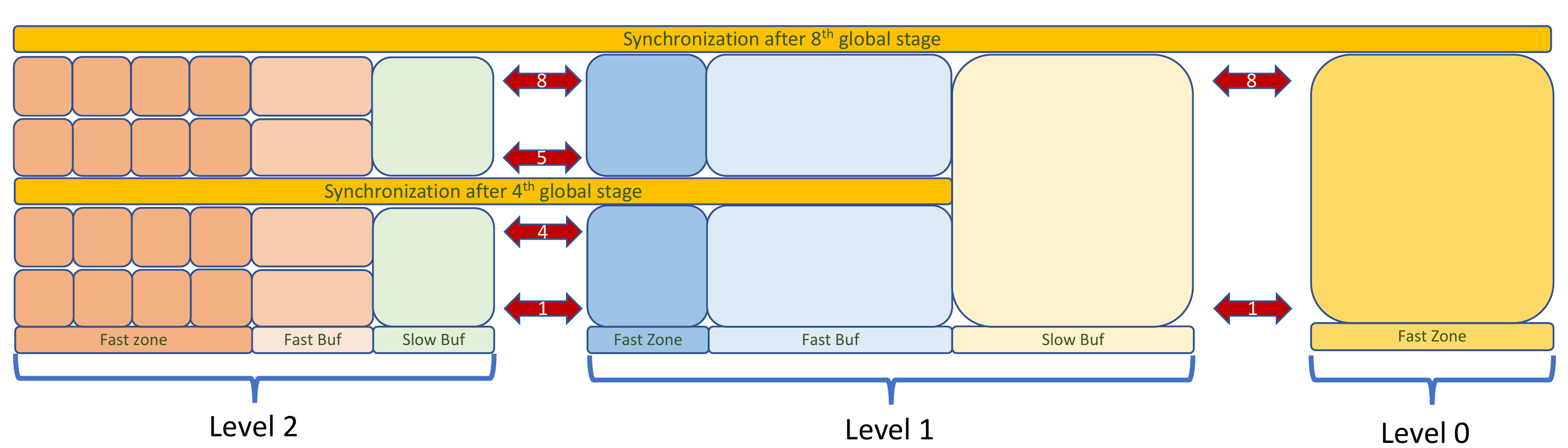}
    \caption{Illustration of MRK2 with a three-level decomposition: 
    three level blocks ($B_1$, $B_2$, and $B_3$) have $0$, $1$, and $2$ multirate levels, respectively. 
    Each level block is composed of the fast zone, the fast buffer, and the slow buffer. 
    The maximum level is two, $L_{\max} = 2$. 
    The number of global stages (the depth of MRK2) is $8(=2^{L_{\max}+1})$, 
    and the number of local stage of each level is $4$ (except level 0). 
    The $B_3$ level block communicates with $B_2$ level block at four stages (i.e., 1, 4, 5, and 8 global stages), 
    whereas the $B_2$ level block exchanges interface data with the $B_1$ level block at two stages (i.e., 1 and 8 global stages).
    }
    \figlab{mrk2-lev3-illustration}
\end{figure}

We give an example with a three-level decomposition in Figure \figref{mrk2-lev3-illustration}, 
where three level blocks ($B_1$, $B_2$, and $B_3$) have $0$, $1$, and $2$ multirate levels, respectively. 
The maximum level is two, $L_{\max} = 2$; thus the total number of global stage becomes $s_G=8$. 
$B_1$ and $B_2$ have one subcycle ($m(B_1)=m(B_2)=1$), and 
$B_3$ has two subcycles ($m(B_3)=2$). 
At every subcycle, a level block ($B$) needs to be synchronized with its neighbors ($\mc{N}(B)$). 
The $B_3$ level block communicates with the $B_2$ level block at four stages (i.e., 1, 4, 5, and 8 global stages), 
whereas the $B_2$ level block exchanges the interface data with the $B_1$ level block at two stages (i.e., 1 and 8 global stages).

The intermediate states and the next step solution corresponding to Figure \figref{mrk2-lev3-illustration} yield
  \begin{align*}     
    Q_{n,i }^{\LRc{B_1,z}}             &= q_{n}^{\LRc{B_1,z}} + 2\dt \sum_{j=1}^{2} a_{ij}^{\LRc{z}} R_{n,j}^{\LRc{B_1,z}} \quad i=1,2, \\
    Q_{n+\frac{0}{1},i }^{\LRc{B_2,z}} &= q_n^{\LRc{B_2,z}}   + \dt \sum_{j=1}^{4} a_{ij}^{\LRc{z}} R_{n,j}^{\LRc{B_2,z}} \quad i=1,\cdots,4,\\
    Q_{n+\frac{0}{2},i }^{\LRc{B_3,z}} &= q_n^{\LRc{B_3,z}}   + \frac{\dt}{2} \sum_{j=1}^{4} a_{ij}^{\LRc{z}} R_{n,j}^{\LRc{B_3,z}} \quad i=1,\cdots,4,\\
    Q_{n+\frac{1}{2},i }^{\LRc{B_3,z}} &= q_{n+\half}^{\LRc{B_3,z}} + \frac{\dt}{2} \sum_{j=1}^{4} a_{ij}^{\LRc{z}} R_{n+\half,j}^{\LRc{B_3,z}} \quad i=1,\cdots,4, 
  \end{align*} 
and 
    \begin{align*}     
    q_{n+1}^{\LRc{B_1,z}} &= q_{n}^{\LRc{B_1,z}} + 2\dt \sum_{i=1}^{2} b_{i}^{\LRc{z}} R_{n,i}^{\LRc{B_1,z}},\\
    q_{n+1}^{\LRc{B_2,z}} &= q_{n}^{\LRc{B_2,z}} +  \dt \sum_{i=1}^{4} b_{i}^{\LRc{z}} R_{n,i}^{\LRc{B_2,z}},\\
    q_{n+1}^{\LRc{B_3,z}} &= q_{n}^{\LRc{B_3,z}} +  \frac{\dt}{2} \sum_{r=1}^2\sum_{i=1}^{4} b_{i}^{\LRc{z}} R_{n+\frac{r-1}{2},i}^{\LRc{B_3,z}},
    \end{align*} 
where 
$$
    q_{n+\half}^{\LRc{B_3,z}} = q_{n}^{\LRc{B_3,z}} + \frac{\dt}{2} \sum_{i=1}^{4} b_{i}^{\LRc{z}} R_{n,i}^{\LRc{B_3,z}}, 
$$  
for $z=1,2,3$.


The relaxation MRK2 for multilevel decomposition is 
\begin{align}     
  \eqnlab{mrk2-qn-multilevel-relaxation}
   q_{n+\gamma }^{\LRc{B,z}} =  q_{n}^{\LRc{B,z}}  
    + \gamma \dt^{\LRc{\ell(B)}} \sum_{r=1}^{m(B)}\sum_{i=1}^{s^{\LRc{\ell(B)}}} b_{i}^{\LRc{z}} R_{n+\frac{r-1}{m(B)},i}^{\LRc{B,z}}  
\end{align} 
for a level block $B$ and a zone number $z$. 

\begin{proposition}
    \theolab{relax-mrk2-multilevel-entropy-conservation}
    Let $N_B$ be the number of level blocks. 
    The relaxation MRK2 method for multilevel decomposition in 
    \eqnref{mrk2-qi-multilevel}, 
    \eqnref{mrk2-qn-multilevel},
    and \eqnref{mrk2-qn-multilevel-relaxation}
    are entropy conserving/stable 
    with an entropy-conserving/stable spatial discretization $R(q)$ and the relaxation parameter
    \begin{multline}
        \eqnlab{mrk2-multilevel-gamma}
    \eta (q_{n+\gamma}) - \eta (q_n)  \\
        - \gamma \sum_{B=1}^{N_B} \dt^{\LRc{\ell(B)}} \sum_{z=1}^3 
        \sum_{r=1}^{m(B)}\sum_{i=1}^{s^{\LRc{\ell(B)}}} 
        b_i^{\LRc{z}} 
        \LRp{ 
          R_{n+\frac{r-1}{m(B)},i}^{\LRc{B,z}},
          \env_{n+\frac{r-1}{m(B)},i}^{\LRc{B,z}} 
        } = 0. 
    \end{multline}
    For the quadratic invariant $\half\norm{q}^2$, the relaxation parameter is explicitly determined: 
    \begin{multline*}
    %
    %
    \gamma = 2 \LRp{\sum_{B=1}^{N_B}\sum_{z=1}^3 \norm{ q_{n+1}^{\LRc{B,z}} - q_n^{\LRc{B,z}} }^{2}}^{-1} \\
    \LRp{\sum_{B=1}^{N_B} \dt^{\LRc{\ell(B)}} \sum_{z=1}^3 \sum_{r=1}^{m(B)}\sum_{i=1}^{s^{\LRc{\ell(B)}}} b_i^{\LRc{z}}  \LRp{ R_{n+\frac{r-1}{m(B)},i}^{\LRc{B,z}},
      Q_{n+\frac{r-1}{m(B)},i}^{\LRc{B,z}} - q_n^{\LRc{B,z}}  } }. 
    \end{multline*}
\end{proposition}

\begin{proof}
    
     The change in the entropy from $t_n$ to $t_{n+\gamma}$ becomes
        \begin{multline} 
        \eqnlab{entropy-gamma-computation-mrk2-multilevel}
          \eta (q_{n+\gamma}) - \eta (q_n)
          =  \\
          \underbrace{ 
            \eta (q_{n+\gamma}) - \eta (q_n)
            - \gamma \sum_{B=1}^{N_B} \dt^{\LRc{\ell(B)}} \sum_{z=1}^3 
        \sum_{r=1}^{m(B)}\sum_{i=1}^{s^{\LRc{\ell(B)}}} 
        b_i^{\LRc{z}} 
        \LRp{ 
          R_{n+\frac{r-1}{m(B)},i}^{\LRc{B,z}},
          \env_{n+\frac{r-1}{m(B)},i}^{\LRc{B,z}} 
        }
          }_{:=\theta (\gamma)} \\
          + \gamma \sum_{B=1}^{N_B} \dt^{\LRc{\ell(B)}} \sum_{z=1}^3 
        \sum_{r=1}^{m(B)}\sum_{i=1}^{s^{\LRc{\ell(B)}}} 
        b_i^{\LRc{z}} 
        \LRp{ 
          R_{n+\frac{r-1}{m(B)},i}^{\LRc{B,z}},
          \env_{n+\frac{r-1}{m(B)},i}^{\LRc{B,z}} 
        }.
        \end{multline} 
    
    By solving \eqnref{mrk2-multilevel-gamma} for $\gamma$, 
    the first,  second, and  third terms in \eqnref{entropy-gamma-computation-mrk2-multilevel} vanish as before. 
        With an entropy-conserving/stable spatial discretization of $R$, 
        the last term in \eqnref{entropy-gamma-computation-mrk2-multilevel} 
        becomes nonpositive as expected.
  
  By substituting $\eta$ with 
  the inner-product norm $\half\norm{q}^2$ and using \eqnref{mrk2-qn-multilevel}, 
  \eqnref{mrk2-multilevel-gamma} becomes
    \begin{multline*}    
        \norm{q_{n+\gamma}}^2 - \norm{q_{n}}^2 
        - 2\gamma \sum_{B=1}^{N_B} \dt^{\LRc{\ell(B)}} \sum_{z=1}^3 
        \sum_{r=1}^{m(B)}\sum_{i=1}^{s^{\LRc{\ell(B)}}} 
        b_i^{\LRc{z}} 
        \LRp{ 
          R_{n+\frac{r-1}{m(B)},i}^{\LRc{B,z}},
          Q_{n+\frac{r-1}{m(B)},i}^{\LRc{B,z}} 
        } \\
        =
        \gamma^2 \norm{q_{n+1} - q_n }^2 
        + 2 \gamma \sum_{B=1}^{N_B} \sum_{z=1}^3 \LRp{ q_{n+1}^{\LRc{B,z} } - q_{n}^{\LRc{B,z} }, q_n^{\LRc{B,z} }} \\
        - 2\gamma \sum_{B=1}^{N_B} \dt^{\LRc{\ell(B)}} \sum_{z=1}^3 
        \sum_{r=1}^{m(B)}\sum_{i=1}^{s^{\LRc{\ell(B)}}} 
        b_i^{\LRc{z}} 
        \LRp{ 
          R_{n+\frac{r-1}{m(B)},i}^{\LRc{B,z}},
          Q_{n+\frac{r-1}{m(B)},i}^{\LRc{B,z}} 
        } \\
        =
        \gamma^2 \norm{q_{n+1} - q_n }^2 \\
        - 2\gamma \sum_{B=1}^{N_B} \dt^{\LRc{\ell(B)}} \sum_{z=1}^3 
        \sum_{r=1}^{m(B)}\sum_{i=1}^{s^{\LRc{\ell(B)}}} 
        b_i^{\LRc{z}} 
        \LRp{ 
          R_{n+\frac{r-1}{m(B)},i}^{\LRc{B,z}},
          Q_{n+\frac{r-1}{m(B)},i}^{\LRc{B,z}}  - q_{n}^{\LRc{B,z}}
        } =0.  
    \end{multline*}
   Rearranging the last equation yields the explicit $\gamma$. 
\end{proof}

 
For implementation, 
first we balance the multirate level of each element 
so that all the level blocks have a 2:1 local time step size ratio to their adjacent level blocks according to Algorithm \algref{MRK-balancing}. 
Next  we construct the activation table in Algorithm \algref{MRK2-acvtable}. 
Then we compute the entropy-conserving/stable solutions 
according to \eqnref{mrk2-multilevel}. 
 
 \begin{algorithm}[h!t!b!]
    \begin{algorithmic}[1]
      \ENSURE Let $K$ be the total number of elements. Let $LE \in \R$ be the initial multirate levels of all elements. 
      Let $BLE \in \R$ be the 2:1 balanced multirate levels of all elements.
      That is, a level jump between adjacent elements is at most one. 
      Given $LE$,  $BLE$ is created. 
      Let $bs$ be the buffer size of two. 
      We assume that $bs+1$ left/right boundary elements have the same level, 
      $LE_{(1:bs+1)}= \ell_{left}$ and $LE_{((K-bs):K)}= \ell_{right}$. 
       
      \STATE{Compute level difference, $\delta LE_{(1:K-1)}= LE_{(2:K)} - LE_{(1:K-1)} $} 
      \STATE{$idx \leftarrow Find (\delta LE \ne 0)$ }
      \STATE{$cond=\textbf{false}$; $BLE \gets LE$}
      \WHILE{$!cond$}
        \STATE{$cond=\textbf{true}$}
        \FOR{$i$ in $\LRc{1:length(idx)}$}
          \STATE{$iK_1 \leftarrow idx[i]$}
          \STATE{$iK_2 \leftarrow idx[i+1]$}
          \IF{$(\delta LE_{(iK_1)} == -1)$ and $(\delta LE_{(iK_2)} == -1)$ } 
            \IF{$(iK_2 - iK_1) < (bs+1) $}
              \STATE{$cond=\textbf{false}$}
              \STATE{$nc = (bs + 1) - (iK_2-iK_1)$}
              \STATE{$BLE_{(iK_2+1:iK_2+nc)} = \max(BLE_{(iK_2+1:iK_2+nc)}, LE_{(iK_2)})$ } 
            \ENDIF
          \ELSIF{$(\delta LE_{(iK_1)} == -1)$ and $(\delta LE_{(iK_2)} == 1)$} 
            \IF{$(iK_2 - iK_1) < (bs+1) $}
              \STATE{$cond=\textbf{false}$}
              \STATE{$BLE_{(iK_1+1:iK_2)} = \max(BLE_{(iK_2)},LE_{(iK_1)})$}
            \ENDIF
          \ELSIF{$(\delta LE_{(iK_1)} == 1)$ and $(\delta LE_{(iK_2)} == 1)$} 
            \STATE{$cond=\textbf{false}$}
            \IF{$(iK_2 - iK_1) < (bs+1) $}
              \STATE{$cond=\textbf{false}$}
              \STATE{$nc = (bs + 1) - (iK_2-iK_1)$}
              \STATE{$BLE_{(iK_1+1-nc:iK_1)} = \max(BLE_{(iK_1+1-nc:iK_1)}, LE_{(iK_2)})$ } 
            \ENDIF
          \ENDIF
        \ENDFOR
        \STATE{Update level of elements, $LE \leftarrow BLE$}
        \STATE{Compute level difference, $\delta LE_{(1:K-1)}:= LE_{(2:K)} - LE_{(1:K-1)} \in \R^{K-1}$ } 
        \STATE{$idx \leftarrow Find (\delta LE \ne 0)$ }
      \ENDWHILE
    \end{algorithmic}
    \caption{Balancing Multirate Level of Elements}
    \alglab{MRK-balancing}
  \end{algorithm}

\section{Numerical Results}
\seclab{NumericalResults}

In this section we present several numerical experiments to demonstrate 
the entropy-conserving/stable properties of the proposed IMEX methods and the multirate methods. 
We compare standard methods, relaxation approaches, and incremental direction techniques for both IMEX and multirate methods. 
For IMEX methods, we use additive Runge--Kutta (ARK) methods \cite{Kennedy2003additive}, 
and call them Relaxation-ARK and IDT-ARK for their relaxation 
and incremental direction techniques, respectively. 
For multirate methods, 
we employ the second-order partitioned multirate Runge--Kutta (MRK2) methods 
\cite{constantinescu2007multirate}, 
which we refer to as Relaxation-MRK2 and IDT-MRK2 for their relaxation and incremental direction techniques, respectively.
We use the IMEX methods for handling scale-separable stiffness on a uniform mesh and the multirate method for dealing with geometric-induced stiffness on nonuniform meshes.
We measure the $L^2$ error of $q$ by $\norm {q - q_r}$, 
 where $q_r$ is either an exact solution or a reference solution. 
 The total entropy difference and the total mass difference are denoted by 
 $|\eta(t) - \eta(0)|$ and $|\textnormal{mass}(t) - \textnormal{mass}(0)|$ at time $t$, 
where $\textnormal{mass}(t):= \LRp{q,1}$.

\subsection{Entropy-Preserving IMEX for ODEs}
\subsubsection*{Conserved Exponential Entropy}
We take the initial condition of $q=(1,0.5)^T$ for \eqnref{ode-expo-entropy} 
and run the simulations for $t\in[0,5]$ with $\dt=0.1$. 
We plot the time series of the exponential entropy in Figure \figref{ode-exponential-entropy}.
We observe that the total entropy differences for both the Relaxation-ARK and IDT-ARK are below $\mathcal{O}(10^{-13})$, 
whereas the standard ARK counterpart shows a difference of orders of magnitude, such as $\mathcal{O}(10^{-3})$, as expected.

\begin{figure} 
  \centering
      \includegraphics[trim=1.2cm 0.2cm 0.2cm 5.2cm,
      width=0.95\columnwidth]{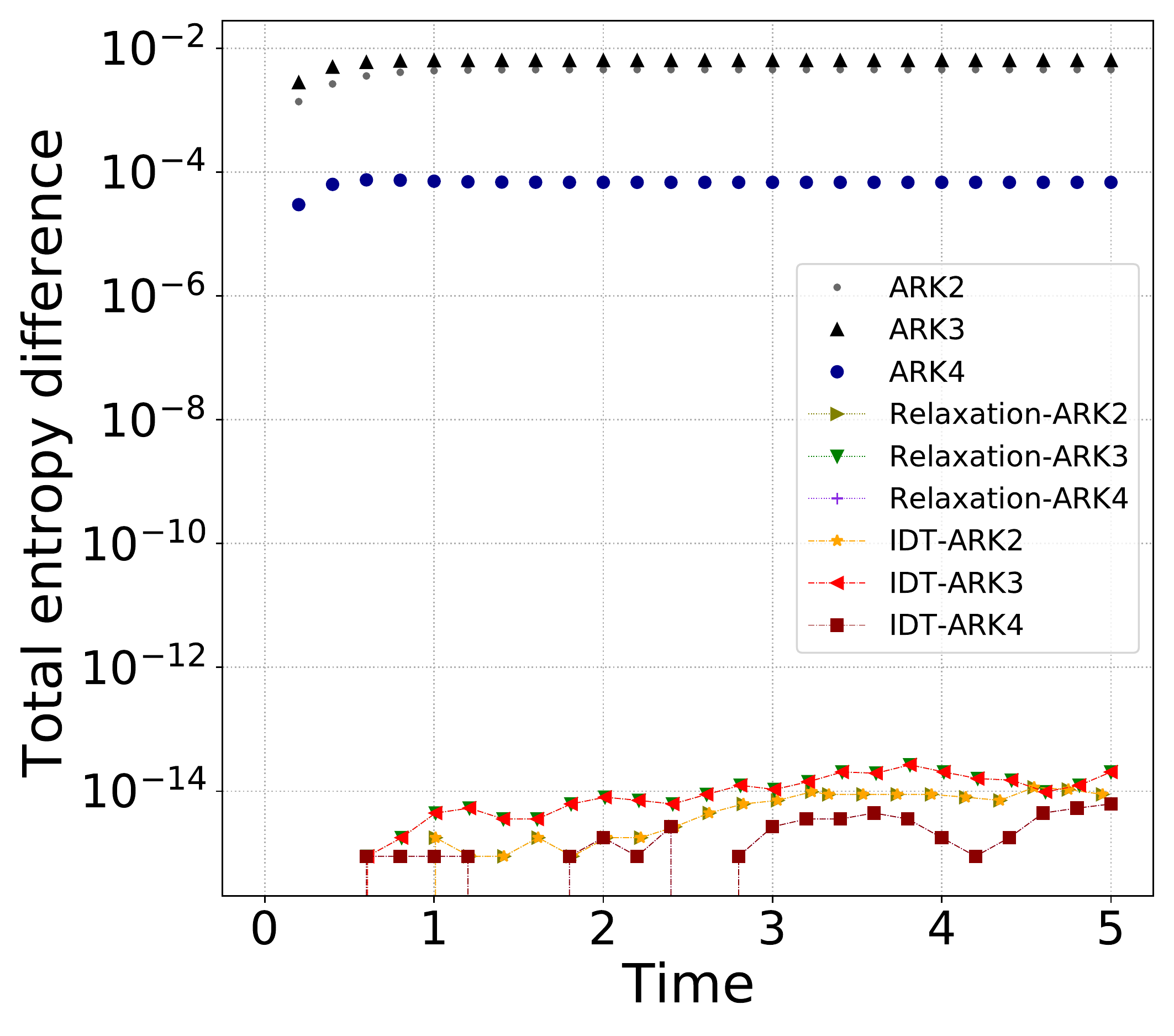}
  \caption{ODE: conserved exponential entropy. 
   Total entropy differences for both the Relaxation-ARK and IDT-ARK are bounded within $\mathcal{O}(10^{-13})$, 
   whereas the ARK counterpart shows a difference of orders of magnitude, $\mathcal{O}(10^{-3})$. }
   \figlab{ode-exponential-entropy}
\end{figure}

\subsubsection*{Nonlinear Pendulum}
 
For \eqnref{ode-pendulum-entropy} we examine the entropy behavior and the solution trajectory over time in Figure \figref{ode-nonlinear-pendulum}.
We take $\dt=0.9$ and run the simulations for $t\in[0,1000]$. 
Both the Relaxation-ARK and IDT-ARK keep the pendulum in a track, 
but standard ARK methods cannot hold the pendulum in the path. 
The total entropy difference for both the Relaxation-ARK and IDT-ARK are bounded within $\mathcal{O}(10^{-13})$;
however, as expected, standard ARK methods have $\mathcal{O}(1)$ entropy difference during the simulation.

\begin{figure} 
  \centering
  \begin{subfigure}[Trajectory]{0.5\linewidth}
      \includegraphics[trim=1.5cm 0.0cm 2.2cm 5.5cm,
      width=0.95\columnwidth]{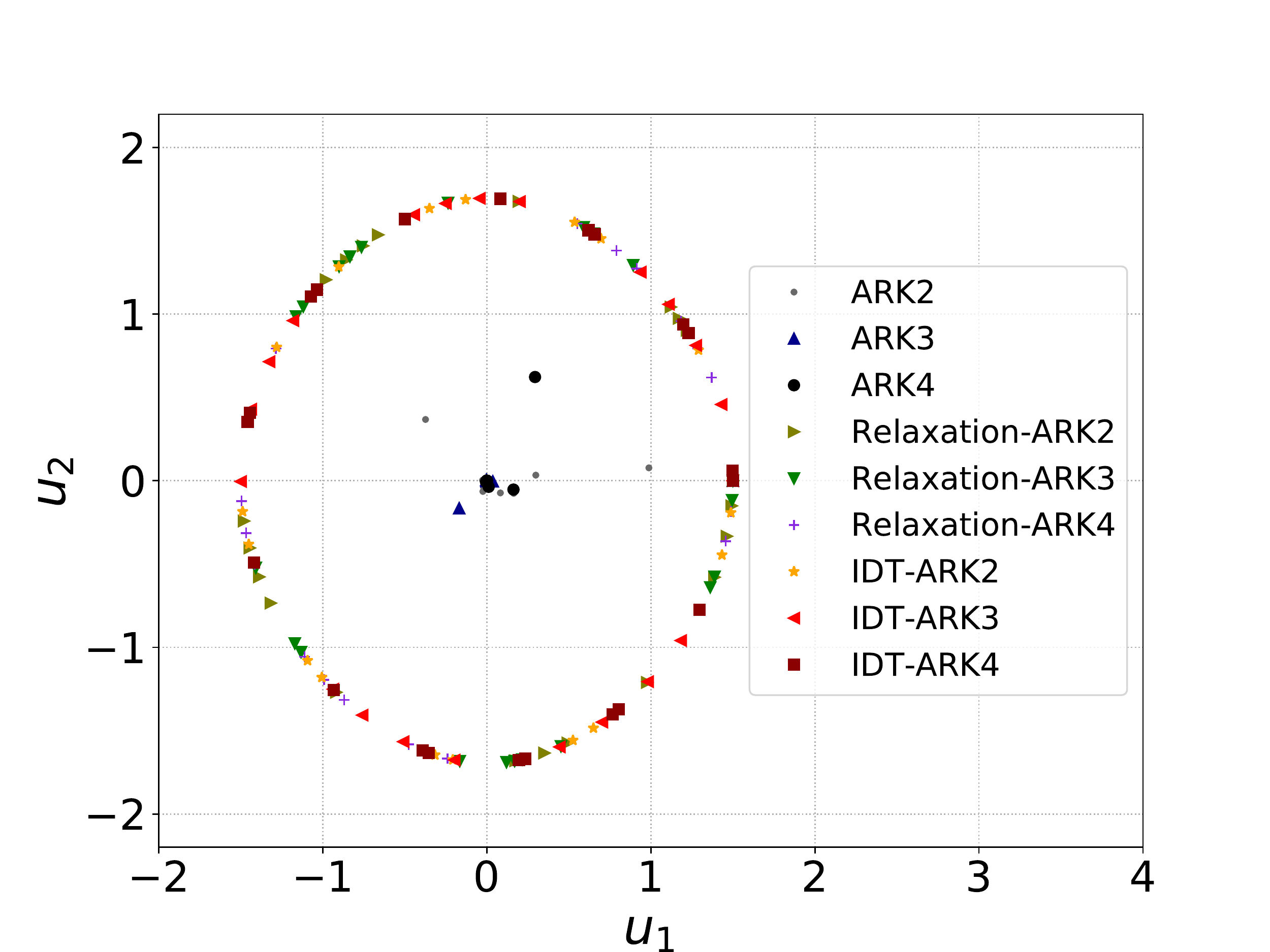}
      \caption{Trajectory}
      \figlab{ode-nonlinear-pendulum-trajectory}
  \end{subfigure} %
  \begin{subfigure}[Energy difference]{0.45\linewidth}    
      \includegraphics[trim=1.5cm 0.5cm 2.2cm 5.5cm,
      width=0.9\columnwidth]{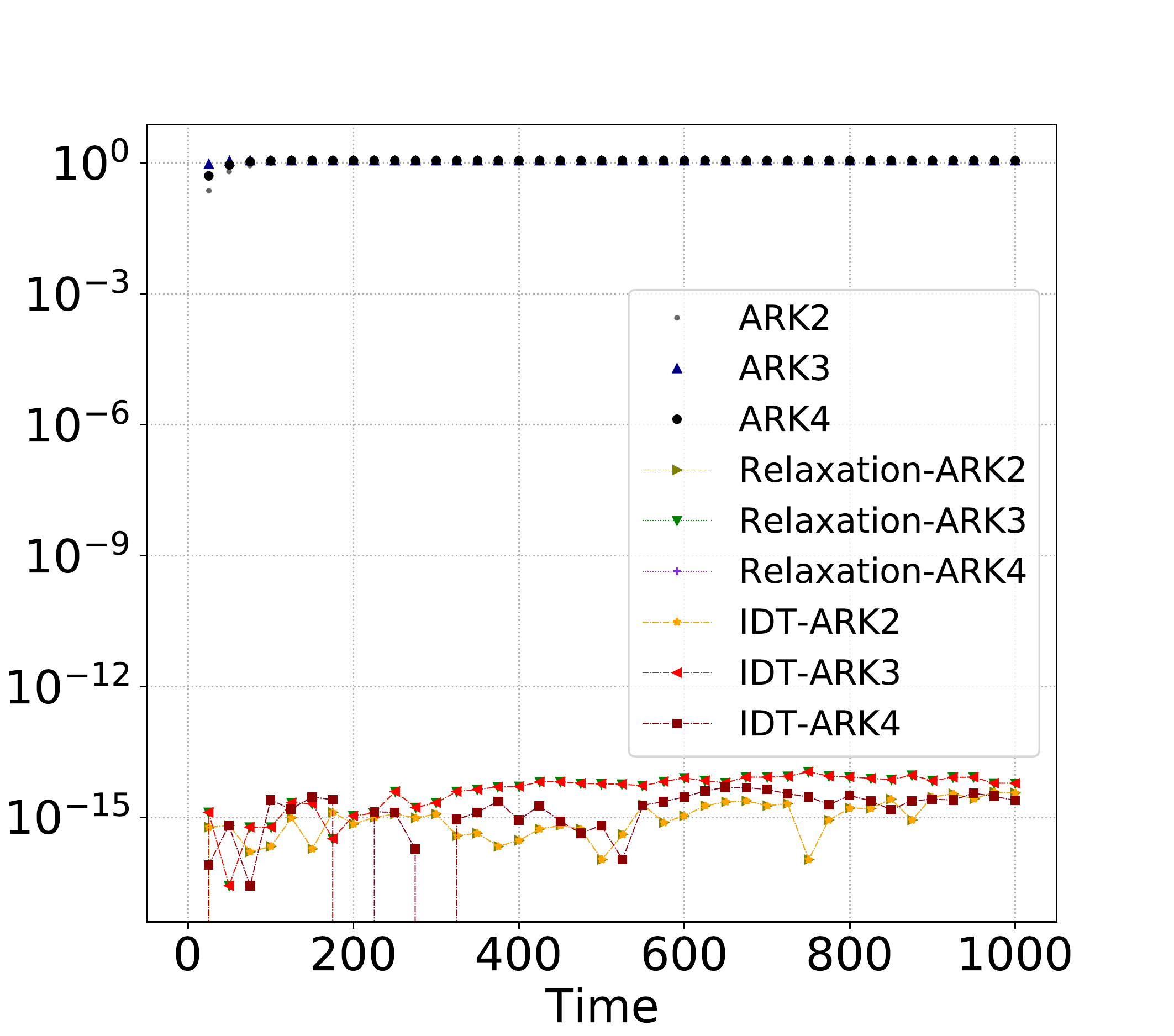}
      \caption{Entropy difference}
      \figlab{ode-nonlinear-pendulum-entropyloss}
  \end{subfigure} 
  \caption{ODE: nonlinear pendulum. 
  Both the Relaxation-ARK and IDT-ARK keep the pendulum in a track, 
  but standard ARK methods cannot hold the pendulum on the exact path. 
  The total entropy difference for both the Relaxation-ARK and IDT-ARK is bounded below $\mathcal{O}(10^{-13})$; 
  however, standard ARK methods have $\mathcal{O}(1)$ entropy difference during the simulation.
  }
   \figlab{ode-nonlinear-pendulum}
\end{figure}

\subsection{Entropy-Stable IMEX for the Burgers Equation on a Uniform Mesh} 
 
We consider a Gaussian initial profile, which develops a shock as time passes for the Burgers equation.
The initial condition is given as 
\begin{alignat*}{2}
  q(t=0) &= \exp(-10 x^2)
\end{alignat*}
on $x\in[-1,1]$. 
A  periodic boundary condition is applied. 

We first perform temporal convergence studies with entropy-conserving (EC) and entropy-stable (ES) fluxes 
for the ARK, Relaxation-ARK, and IDT-ARK methods.
In particular, we use the IMEX methods based on the linearized flux in \eqnref{gov-burgers1d-split}.
We take the RK4 solution (with $\dt = 5\times 10^{-6}$, $N=3$, and $N_E=100$) as the ``ground truth" solution 
and measure the relative errors at $t = 0.2$ (before forming a shock) in Table \tabref{burgers-tconv-gaussian-ark2} and Table \tabref{burgers-tconv-gaussian-ark3}.

In Table \tabref{burgers-tconv-gaussian-ark2}
we observe the second-order rate of convergence for both ARK2 and Relaxation-ARK2 with EC and ES fluxes. IDT-ARK2, however, shows the first-order rate of convergence.
This is a consequence of the time discretization error of the IDT approach. 
Similarly, in Table \tabref{burgers-tconv-gaussian-ark3},
 IDT-ARK3 shows a second-order rate of convergence, 
which is one degree less accurate than that of its ARK3 and Relaxation-ARK3 counterparts. 
As shown in both Table \tabref{burgers-tconv-gaussian-ark2} and Table \tabref{burgers-tconv-gaussian-ark3}, the relative error of Relaxation-ARK methods is slightly lower than that of naive ARK methods.
  
 \begin{table}[t] 
  \caption{Gaussian example: temporal convergence study of ARK2 methods conducted 
  on a uniform mesh of $N=3$ and $K=100$.  
  Time step sizes are chosen as $\dt=0.00125\LRc{1, 1/2, 1/4, 1/8, 1/16}$ with EC flux; 
  and $\dt=0.005\LRc{1, 1/2, 1/4, 1/8, 1/16}$ with ES flux. 
  By taking the RK4 solution with $\dt=5.0\times 10^{-6}$ as the  ``ground truth" solution, 
  we measure the relative errors of the ARK2, Relaxation-ARK2, and IDT-ARK2 methods at $t=0.2$. } 
  \tablab{burgers-tconv-gaussian-ark2} 
  \begin{center} 
    \begin{tabular}{*{1}{c}|*{1}{c}|*{2}{c}|*{2}{c}|*{2}{c}} 
      \hline 
      \multirow{2}{*}{$flux$}
      & \multirow{2}{*}{$dt$}
      & \multicolumn{2}{c}{ARK2} 
      & \multicolumn{2}{c}{Relaxation-ARK2 } 
      & \multicolumn{2}{c}{IDT-ARK2} \tabularnewline 
      & & Error & Order &Error & Order &Error & Order \tabularnewline 
    \hline\hline 
    &1.250e-03&       1.60E-05 &      $-$&      1.48E-05 &     $-$&       1.30E-04 & $-$\tabularnewline
    &6.250e-04&       4.00E-06 &    2.00&       3.71E-06 &    2.00&       6.65E-05 &    0.97\tabularnewline
 EC &3.125e-04&       1.00E-06 &    2.00&       9.29E-07 &    2.00&       3.37E-05 &    0.98\tabularnewline
    &1.563e-04&       2.51E-07 &    2.00&       2.32E-07 &    2.00&       1.69E-05 &    0.99\tabularnewline
    &7.813e-05&       6.27E-08 &    2.00&       5.81E-08 &    2.00&       8.49E-06 &    1.00\tabularnewline
 \tabularnewline 
    &5.000e-03&       2.51E-04 &      $-$&      2.32E-04 &     $-$&       4.80E-04 & $-$\tabularnewline
    &2.500e-03&       6.36E-05 &    1.98&       5.88E-05 &    1.98&       2.50E-04 &    0.94\tabularnewline
 ES &1.250e-03&       1.60E-05 &    1.99&       1.48E-05 &    1.99&       1.30E-04 &    0.94\tabularnewline
    &6.250e-04&       4.00E-06 &    2.00&       3.71E-06 &    2.00&       6.65E-05 &    0.97\tabularnewline
    &3.125e-04&       1.00E-06 &    2.00&       9.29E-07 &    2.00&       3.36E-05 &    0.98\tabularnewline
 \hline\hline 
     \end{tabular} 
  \end{center}     
\end{table} 
 
\begin{table}[t] 
  \caption{Same as \tabref{burgers-tconv-gaussian-ark2}, except the third-order accurate methods.} 
  \tablab{burgers-tconv-gaussian-ark3} 
  \begin{center} 
    \begin{tabular}{*{1}{c}|*{1}{c}|*{2}{c}|*{2}{c}|*{2}{c}} 
      \hline 
      \multirow{2}{*}{$flux$}
      & \multirow{2}{*}{$\dt$}
      & \multicolumn{2}{c}{ARK3} 
      & \multicolumn{2}{c}{Relaxation-ARK3 } 
      & \multicolumn{2}{c}{IDT-ARK3} \tabularnewline 
      & & Error & Order &Error & Order &Error & Order \tabularnewline 
    \hline\hline 
    &1.250e-03&       4.76E-07 &     $-$&       4.53E-07 &     $-$&       1.02E-05 & $-$\tabularnewline
    &6.250e-04&       6.03E-08 &    2.98&       5.76E-08 &    2.98&       2.54E-06 &    2.00\tabularnewline
 EC &3.125e-04&       7.61E-09 &    2.99&       7.29E-09 &    2.98&       6.34E-07 &    2.00\tabularnewline
    &1.563e-04&       9.58E-10 &    2.99&       9.18E-10 &    2.99&       1.58E-07 &    2.00\tabularnewline
    &7.813e-05&       1.20E-10 &    2.99&       1.15E-10 &    2.99&       3.96E-08 &    2.00\tabularnewline
 \tabularnewline 
    &5.000e-03&       2.76E-05 &     $-$&       2.58E-05 &     $-$&       1.65E-04 & $-$\tabularnewline
    &2.500e-03&       3.61E-06 &    2.93&       3.41E-06 &    2.92&       4.09E-05 &    2.01\tabularnewline
 ES &1.250e-03&       4.59E-07 &    2.97&       4.36E-07 &    2.97&       1.02E-05 &    2.01\tabularnewline
    &6.250e-04&       5.79E-08 &    2.99&       5.51E-08 &    2.99&       2.54E-06 &    2.00\tabularnewline
    &3.125e-04&       7.26E-09 &    3.00&       6.92E-09 &    2.99&       6.34E-07 &    2.00\tabularnewline
 \hline\hline     
      \end{tabular} 
  \end{center}     
\end{table}

To investigate the entropy-conserving properties of ARK methods, 
we conduct the numerical experiments for $t\in \LRs{0,2}$ with a uniform mesh of $N=3$ and $N_E=800$. 
The time step size of RK2 is taken as $\dt_{RK}=3.125\times 10^{-5}$, 
whereas the time  step sizes of ARK, Relaxation-ARK, and IDT-ARK have $5\times \dt_{RK}$.
\footnote{
  RK2 with $2\times \dt_{RK}$ leads to blow up its numerical solution.
}
Figure \figref{pde-burgers-ec-totalenergyhistory-ark} shows the time series of the total energy and its difference 
for the RK2,
ARK2, Relaxation-ARK2, IDT-ARK2, 
ARK3, Relaxation-ARK3, and IDT-ARK3 methods. 
The second- and the third-order Relaxation-ARK and IDT-ARK methods conserve their total energies within $\mc{O}(10^{-13})$ differences,
whereas ARK2 and ARK3 show a slightly decreasing trend of total energy. 
This is because IMEX methods act as a high-frequency filter by treating the fast-varying dynamics implicitly \cite{kang2019imex,GiraldoKellyConstantinescu13}.
As a result, energy-stable behavior is observed for the standard ARK methods. 
RK2, however, does not have any filter functionality, so its total energy shows an increasing trend.
 
\begin{figure} 
  \centering
  \begin{subfigure}[Total energy history (EC)]{0.45\linewidth}
      \includegraphics[trim=0.2cm 0.2cm 0.2cm 0.2cm, width=0.95\columnwidth]{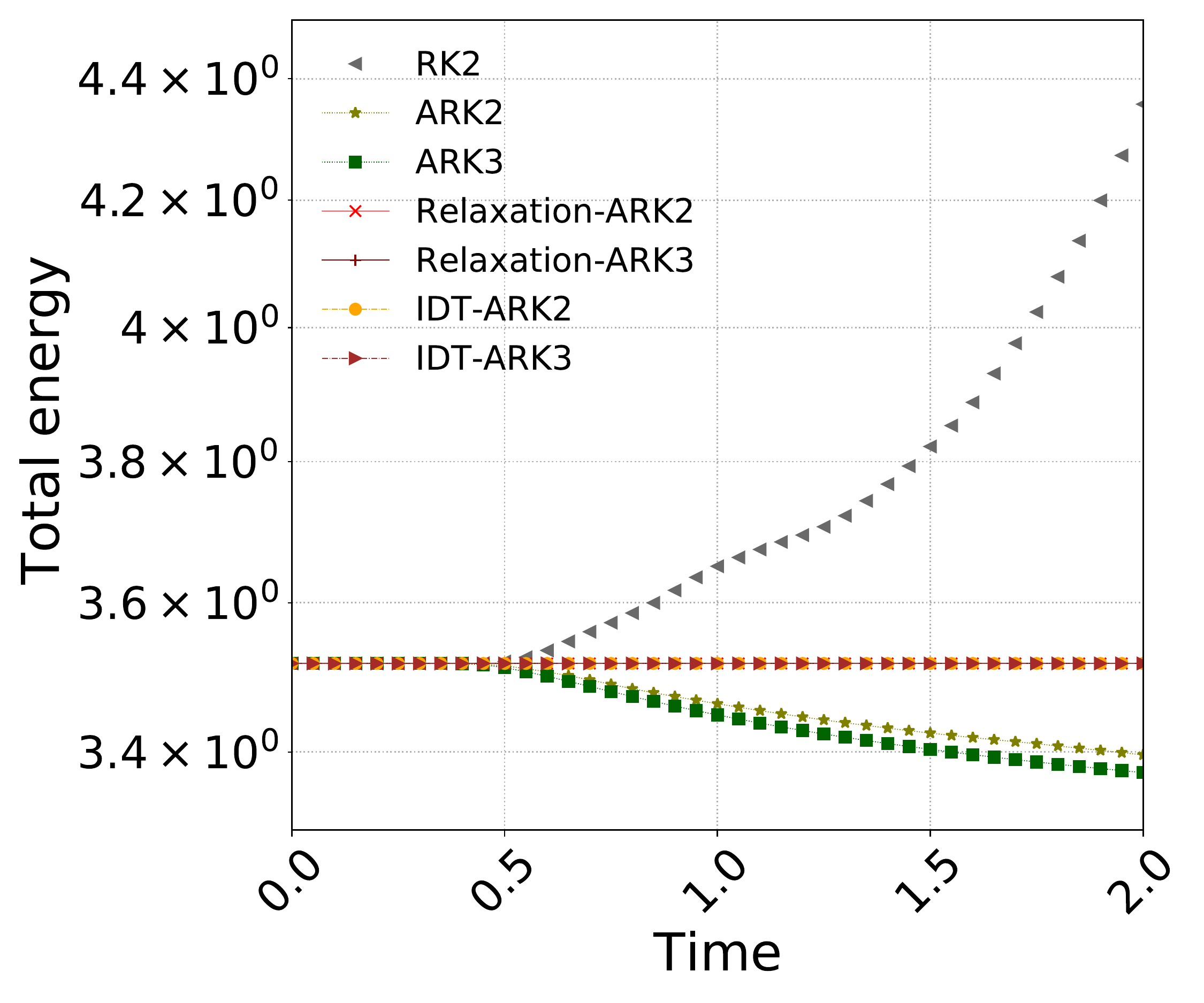}
      \caption{Total energy history (EC)}
      \figlab{pde-burgers-ec-energyloss-ark}
  \end{subfigure} %
  \begin{subfigure}[Total energy difference (EC)]{0.45\linewidth}    
      \includegraphics[trim=0.2cm 0.2cm 0.2cm 0.2cm, width=0.95\columnwidth]{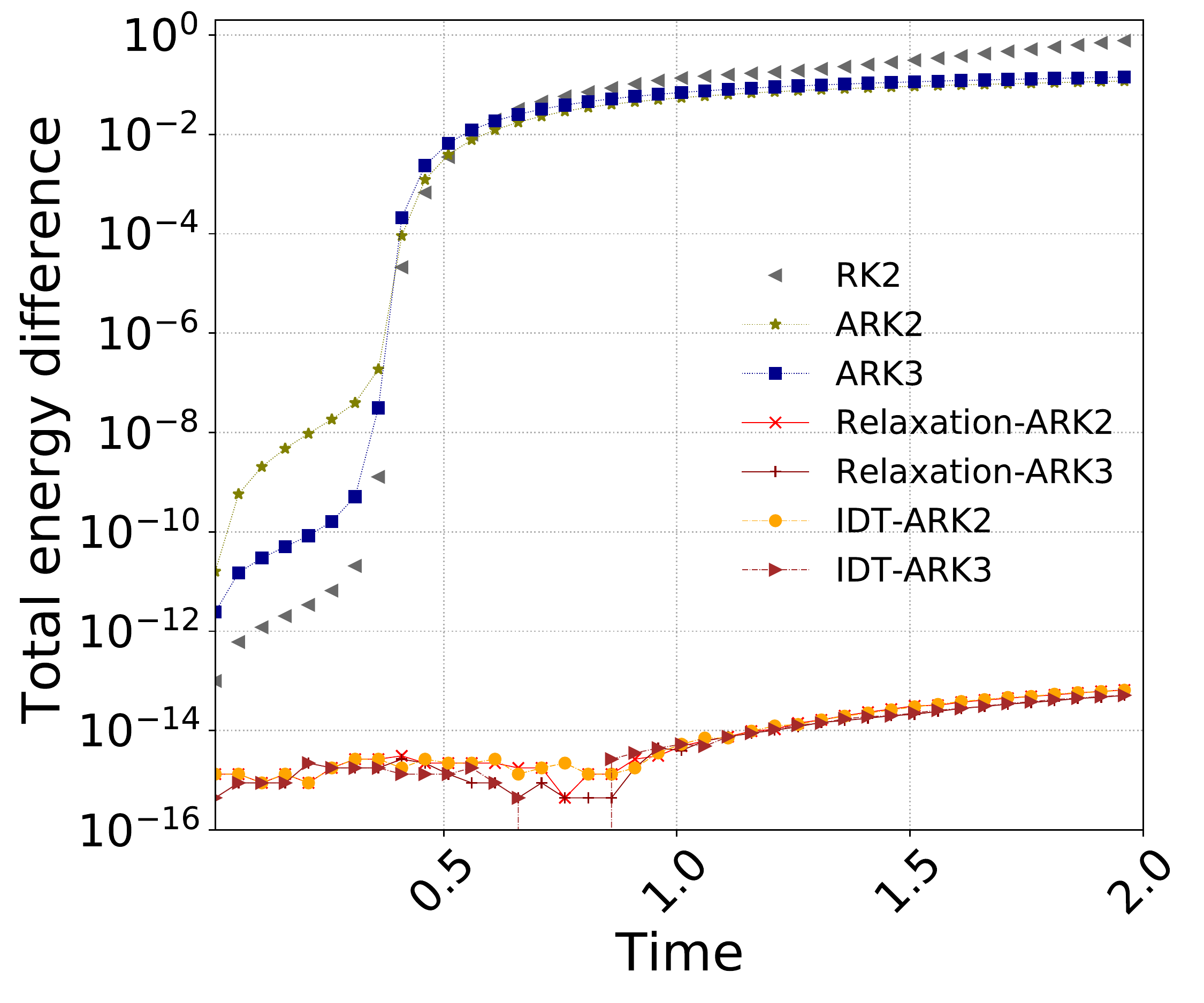}
      \caption{Total energy difference (EC)}
      \figlab{pde-burgers-ec-energyhistory-ark}
  \end{subfigure} 
  \caption{Histories of total entropy and its difference of Gaussian example for the Burgers equation with EC flux: 
      the relaxation methods with EC flux conserve the total energy within $\mc{O}(10^{-13})$.  }
   \figlab{pde-burgers-ec-totalenergyhistory-ark}
\end{figure}

We show snapshots at $t=1$ in Figure \figref{pde-burgers-ec-ss-ark-comparison}. 
All numerical solutions suffer from high-frequency noise arising from the Gibbs phenomenon 
in the presence of a shock. 
However, the numerical solutions do not blow up thanks to the skew-symmetric formulation \cite{gassner2013skew}.
Compared with RK2, ARK2 dramatically eliminates the high-frequency oscillation. 
Relaxation-ARK2 and IDT-ARK2 also reduce the high-frequency oscillation but not as significantly as ARK2. 

\begin{figure} 
  \centering  
  \begin{subfigure}[snapshot]{0.45\linewidth}
      \includegraphics[trim=1.5cm 1.0cm 1.5cm 1.0cm, width=0.9\columnwidth]{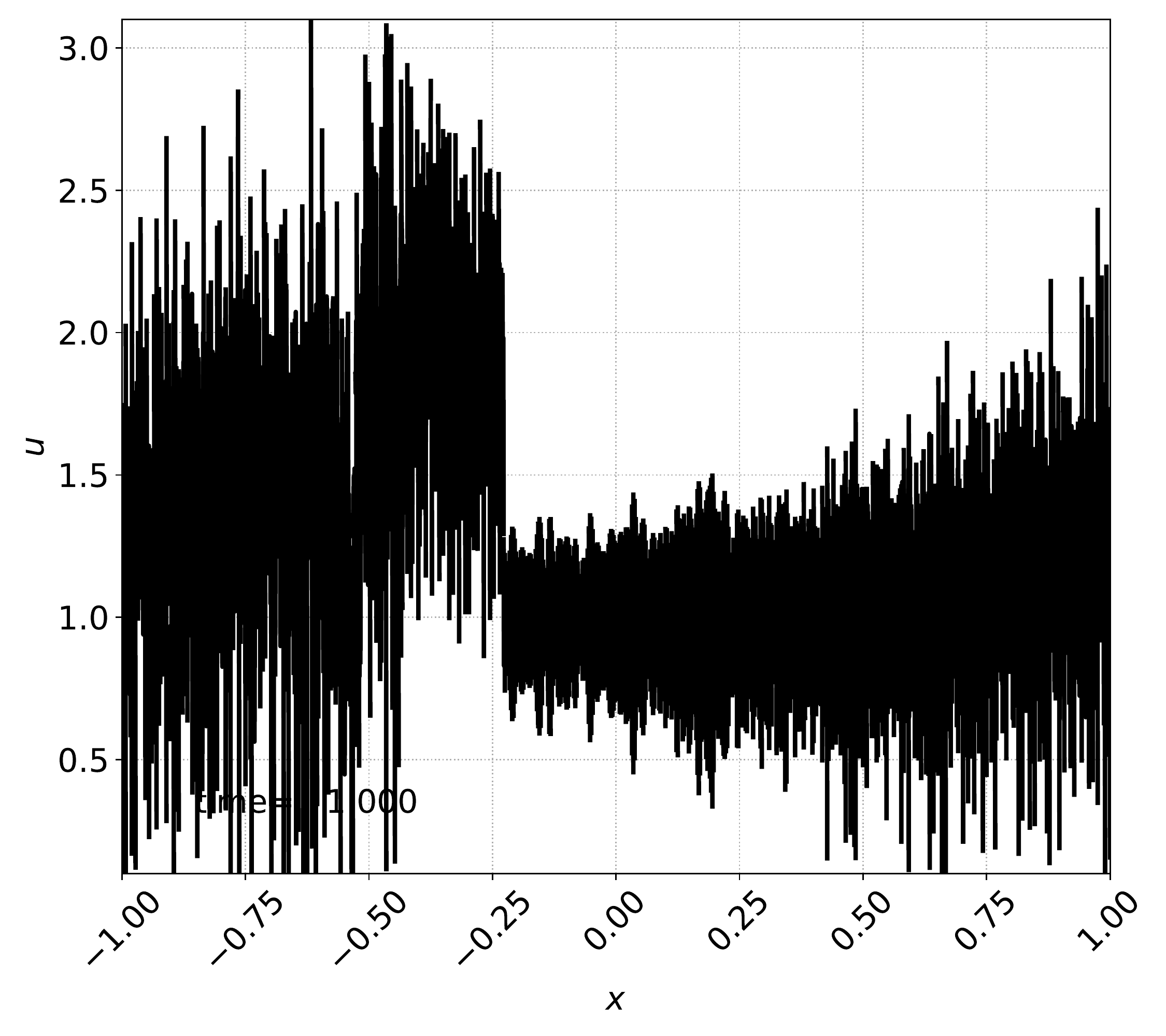}
      \caption{RK2 }
      \figlab{pde-burgers-ss-ec-ark-t1-rk2}
  \end{subfigure} %
  \begin{subfigure}[snapshot]{0.45\linewidth}
    \includegraphics[trim=1.5cm 1.0cm 1.5cm 1.0cm, width=0.9\columnwidth]{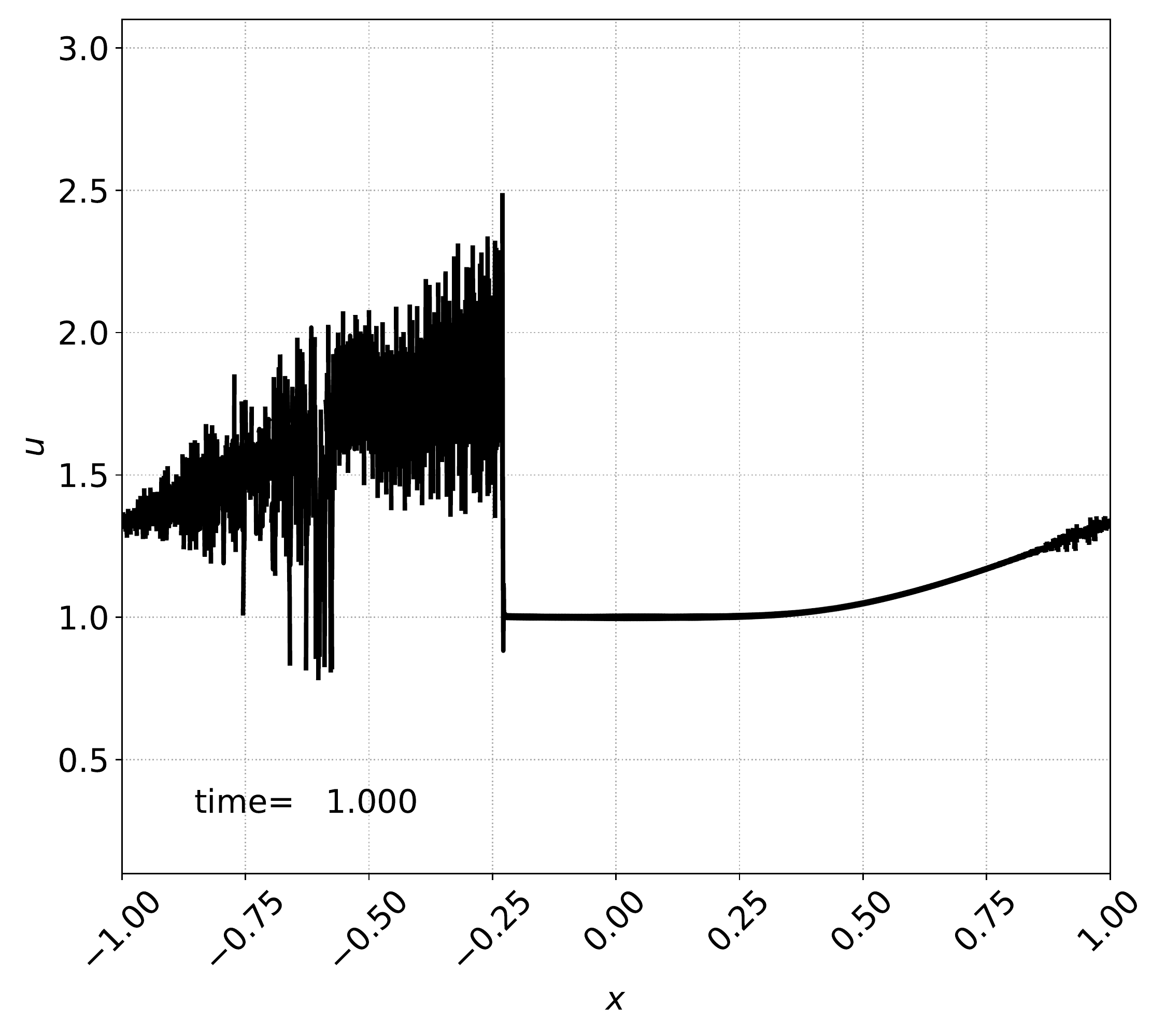}
    \caption{ARK2 }
    \figlab{pde-burgers-ss-ec-ark-t1-ark2}
  \end{subfigure} %
  \begin{subfigure}[snapshot]{0.45\linewidth}
    \includegraphics[trim=1.5cm 1.0cm 1.5cm 0.0cm, width=0.9\columnwidth]{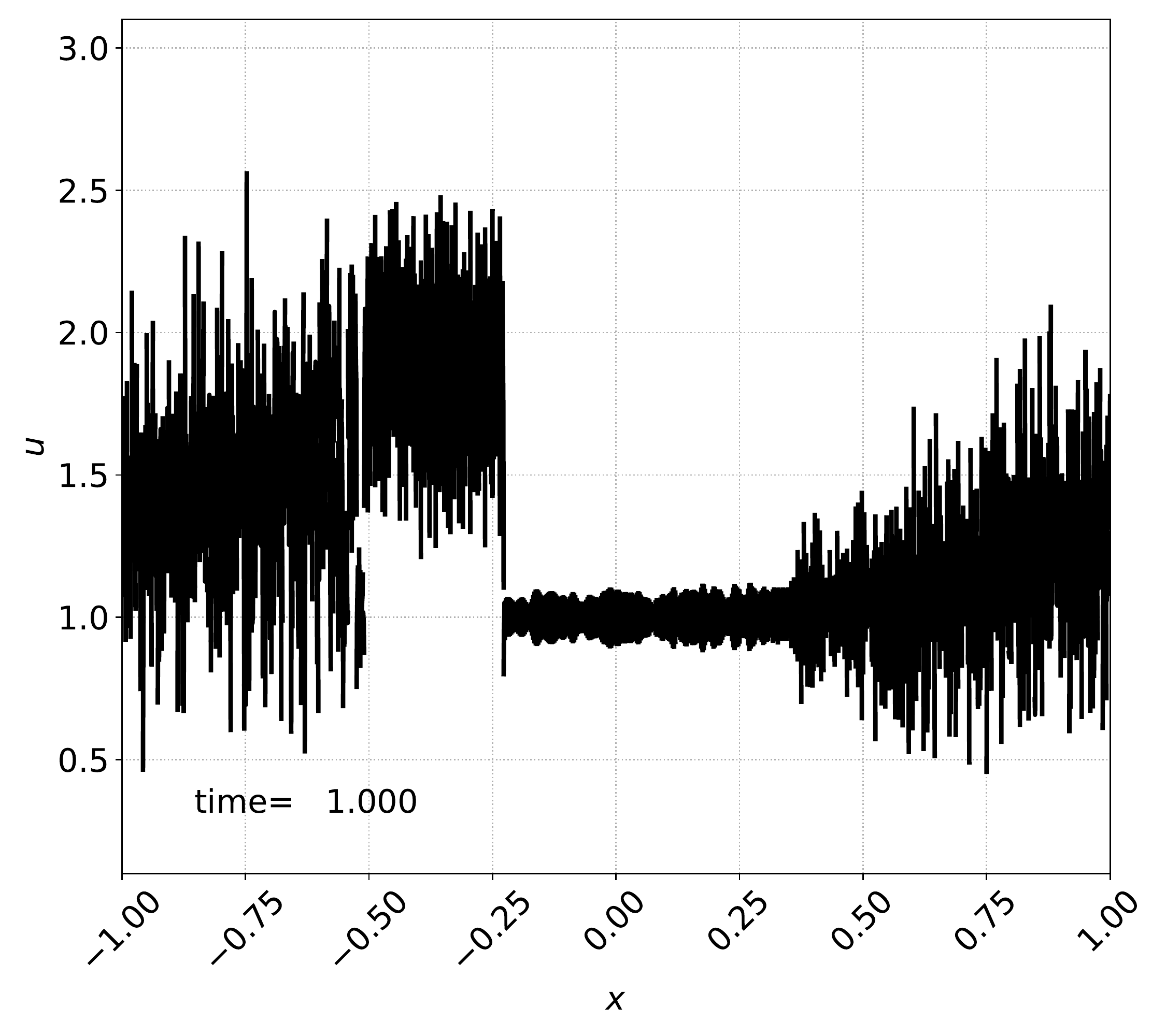}
    \caption{Relaxation-ARK2}
    \figlab{pde-burgers-ss-ec-ark-t1-relaxark2}
  \end{subfigure} 
  \begin{subfigure}[snapshot]{0.45\linewidth}
    \includegraphics[trim=1.5cm 1.0cm 1.5cm 0.0cm, width=0.9\columnwidth]{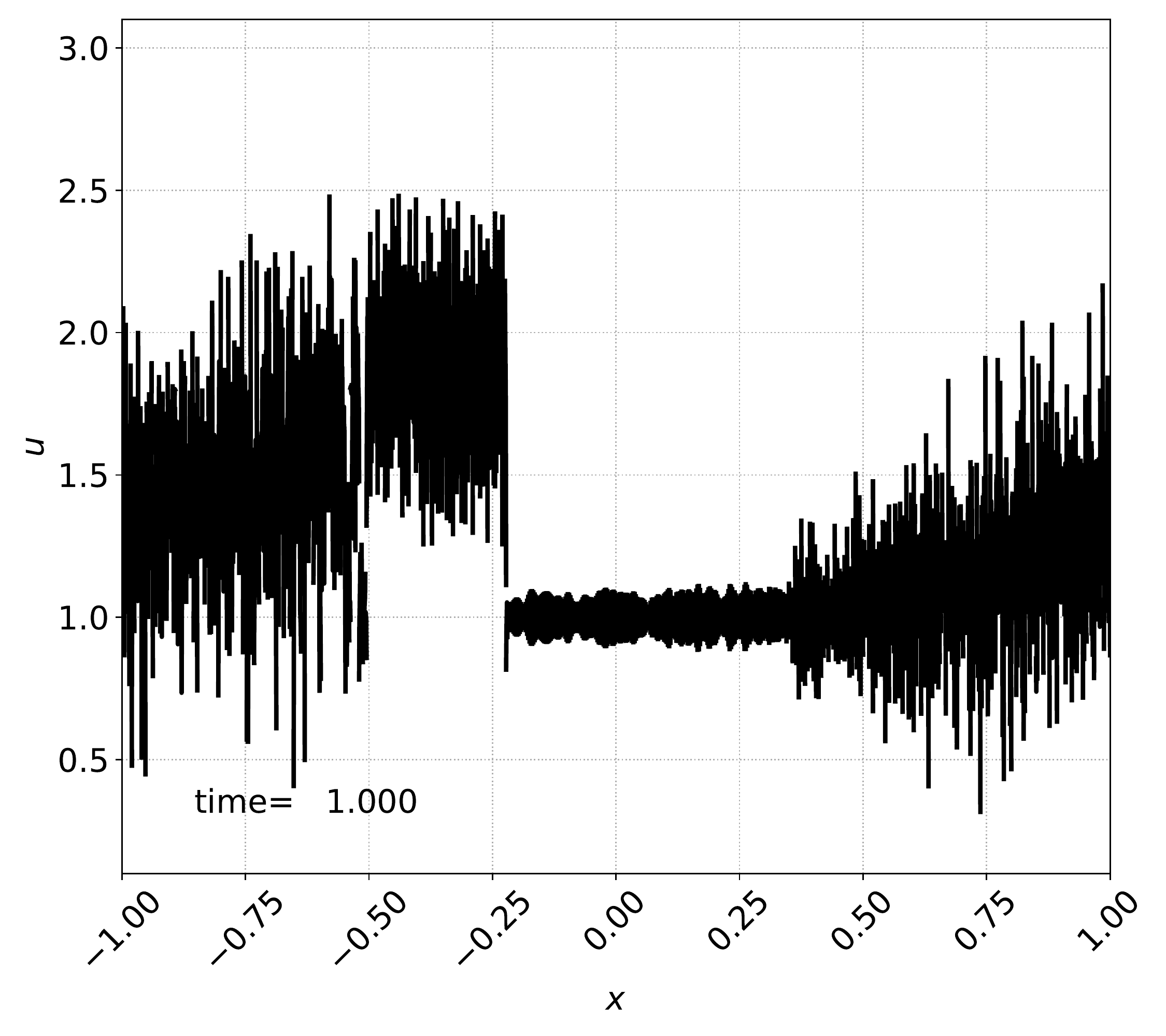}
    \caption{IDT-ARK2}
    \figlab{pde-burgers-ss-ec-ark-t1-idxark2}
  \end{subfigure} 
  \caption{Snapshots of Gaussian profile for the Burgers equation at $t=1$: (a) RK2, (b) ARK2, (c) Relaxation-ARK2, and (d) IDT-ARK2.
  The time step size of RK2 is taken as $\dt_{RK}=3.125\times 10^{-5}$, 
  whereas the time step sizes of ARK2, Relaxation-ARK2, and IDT-ARK2 have $5\times \dt_{RK}$.
  The domain is discretized with a uniform mesh of $N=3$ and $N_E=800$. 
  }
  \figlab{pde-burgers-ec-ss-ark-comparison}
\end{figure}

Next we examine the entropy-stable properties of the ARK methods. 
We perform the simulations for $t\in \LRs{0,2}$ with $N=3$ and $N_E=800$. 
The time step size of RK2 is taken as $\dt_{RK}=2.5 \times 10^{-4}$, 
whereas the time step sizes of the other methods including the ARK2 method have $2.5\times \dt_{RK}$.
\footnote{
  RK2 with $\dt_{RK}=5\times 10^{-4}$ leads to blowup of the numerical solution.
}
Compared with EC flux, ES flux substantially eliminates numerical oscillations but still not enough 
to remove nonphysical oscillations near shocks. 
Thus, we additionally apply the limiter in \eqnref{shock-limiter-burgers-cs17} 
to a marched solution at every time step. 

The snapshots at $t=1$ are reported in Figure \figref{pde-burgers-lf-ss-ark-comparison}.
All the methods with the limiter 
successfully eliminate the spurious oscillations near the shock front. 
The shock front, located near $x=-0.25$, is well captured for all methods with/without the limiter in general.
However, the IDT-ARK2 method with the limiter shows
the shock position error compared with other methods.

\begin{figure} 
  \centering  
  \begin{subfigure}[snapshot]{0.45\linewidth}
    \includegraphics[trim=1.5cm 1.0cm 1.5cm 1.0cm, width=0.9\columnwidth]{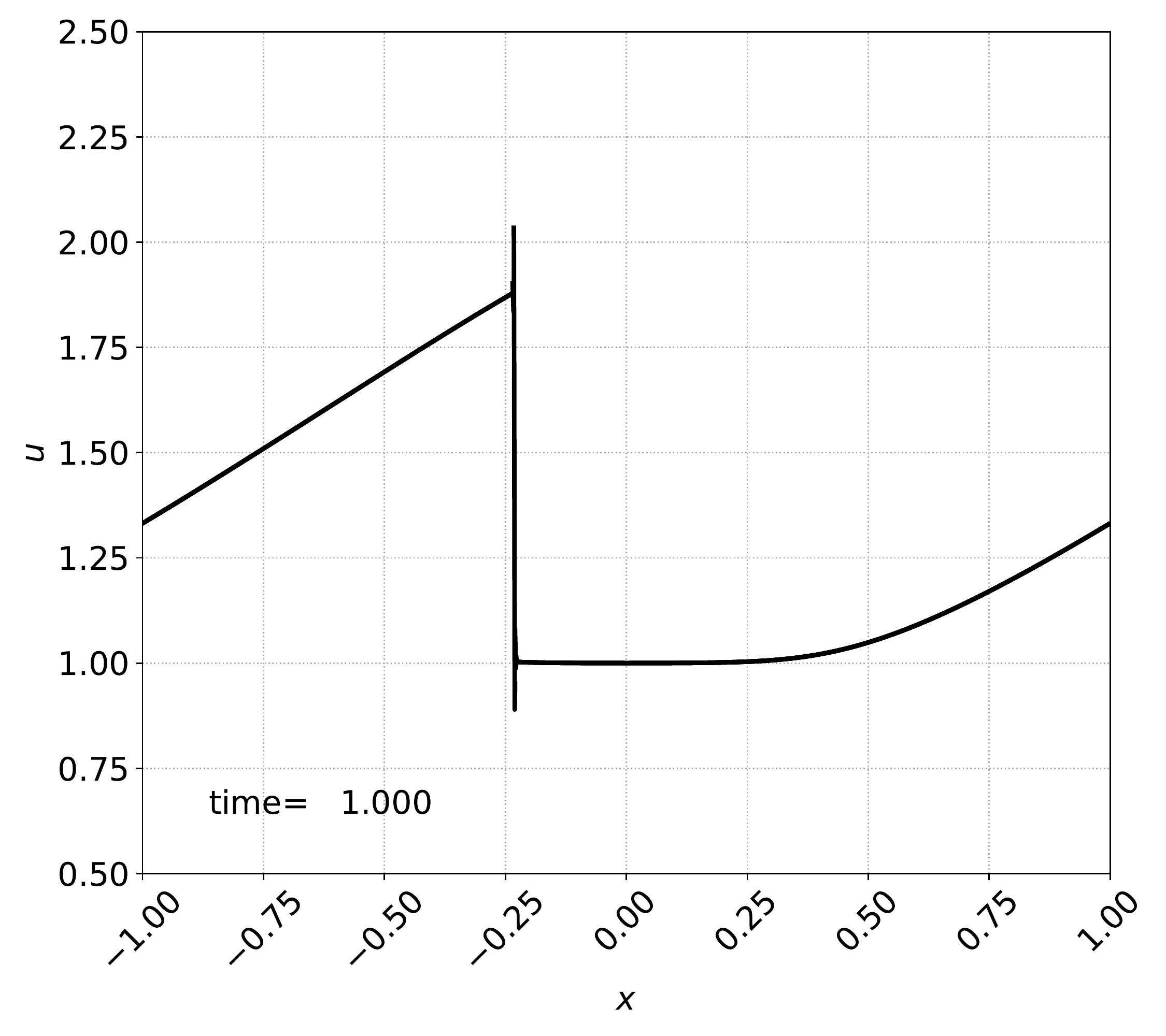}
    \caption{RK2 }
    \figlab{burgers-lf-ark-t1-rk2}
  \end{subfigure} %
  \begin{subfigure}[snapshot]{0.45\linewidth}
    \includegraphics[trim=1.5cm 1.0cm 1.5cm 1.0cm, width=0.9\columnwidth]{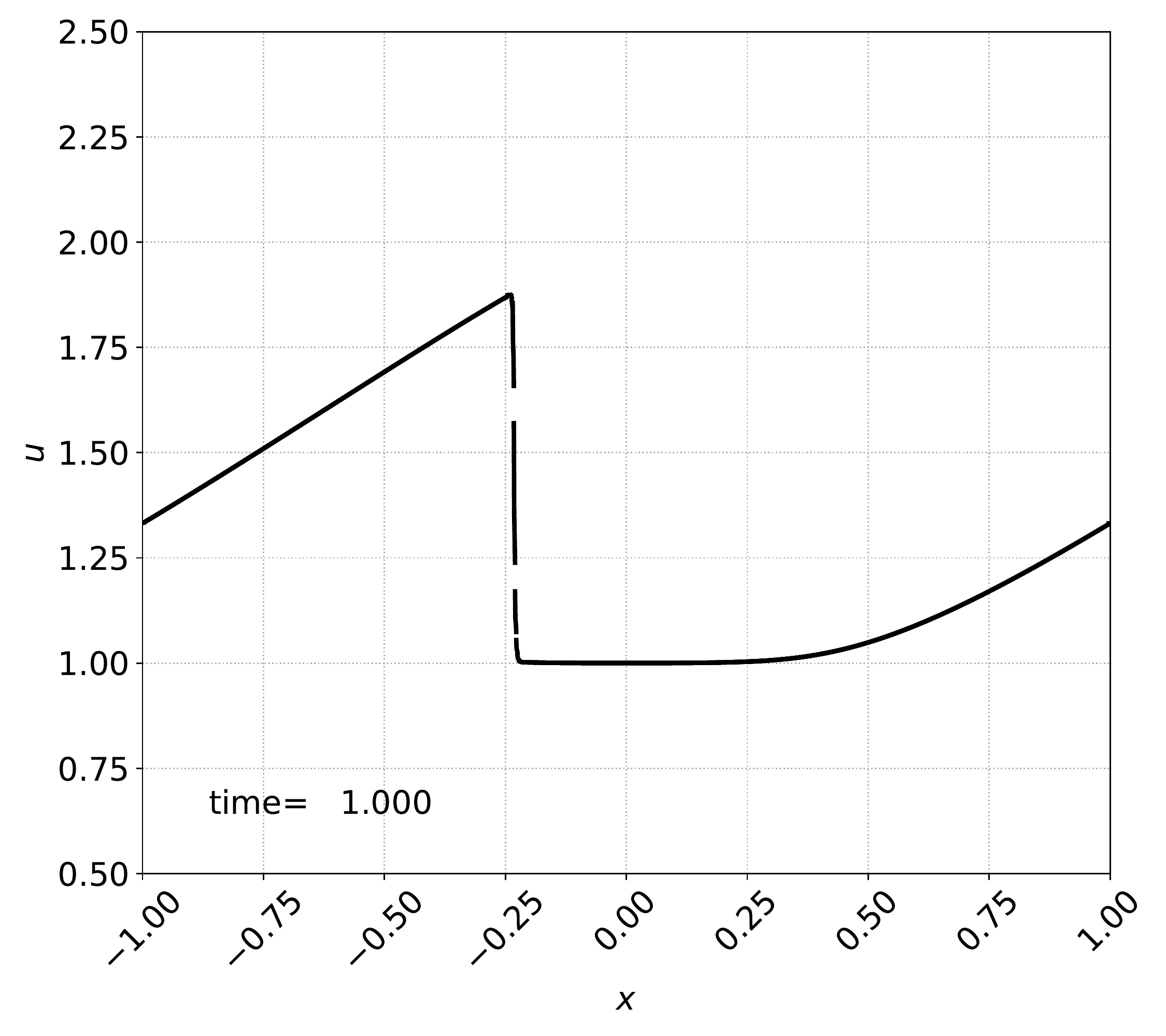}
    \caption{RK2 w/ limiter}
    \figlab{burgers-lf-ark-t1-rk2-cs17}
  \end{subfigure} 
  \begin{subfigure}[snapshot]{0.45\linewidth}
    \includegraphics[trim=1.5cm 1.0cm 1.5cm 0.0cm, width=0.9\columnwidth]{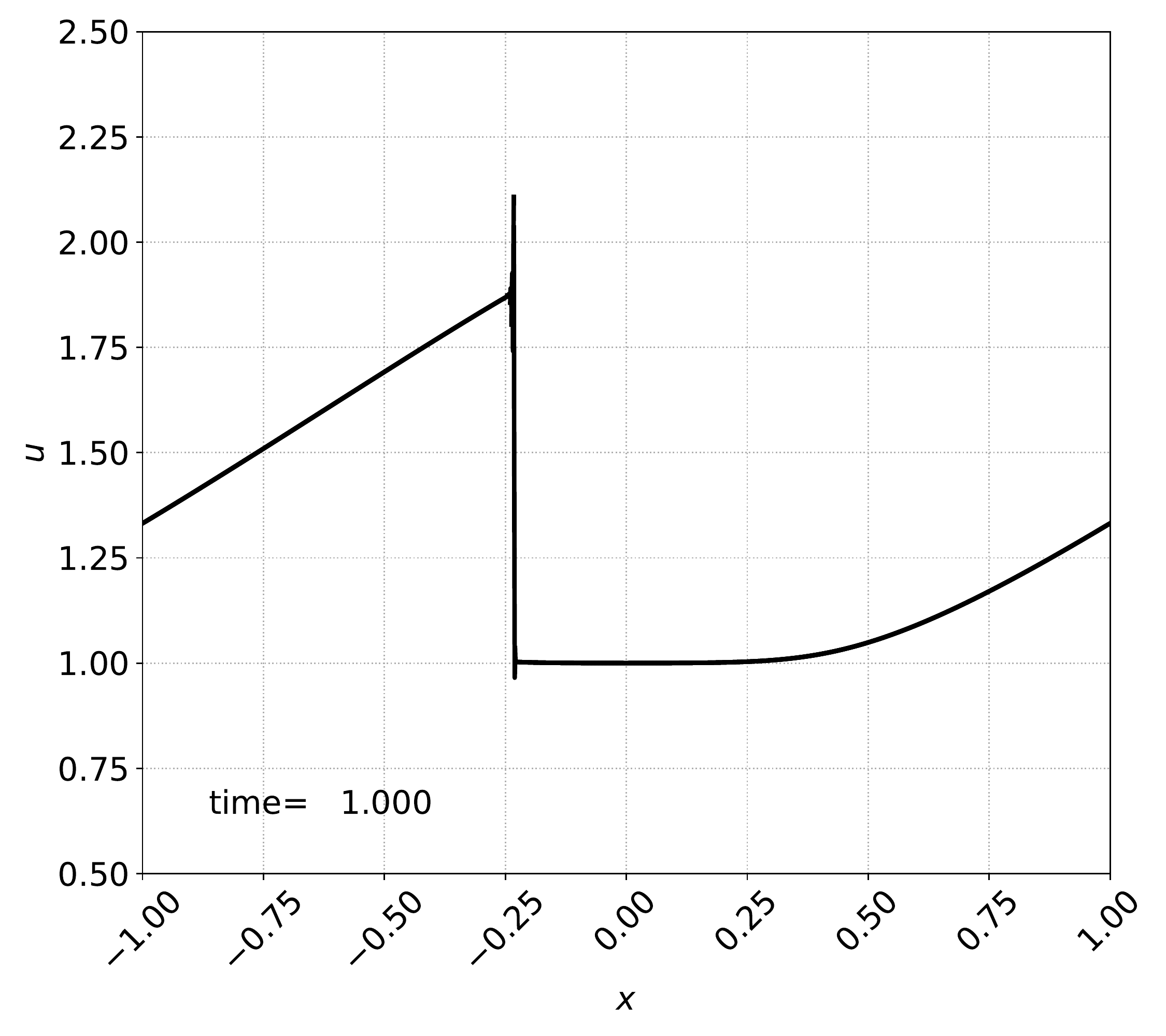}
    \caption{ARK2}
    \figlab{burgers-lf-ark-t1-ark2}
  \end{subfigure} 
  \begin{subfigure}[snapshot]{0.45\linewidth}
    \includegraphics[trim=1.5cm 1.0cm 1.5cm 0.0cm, width=0.9\columnwidth]{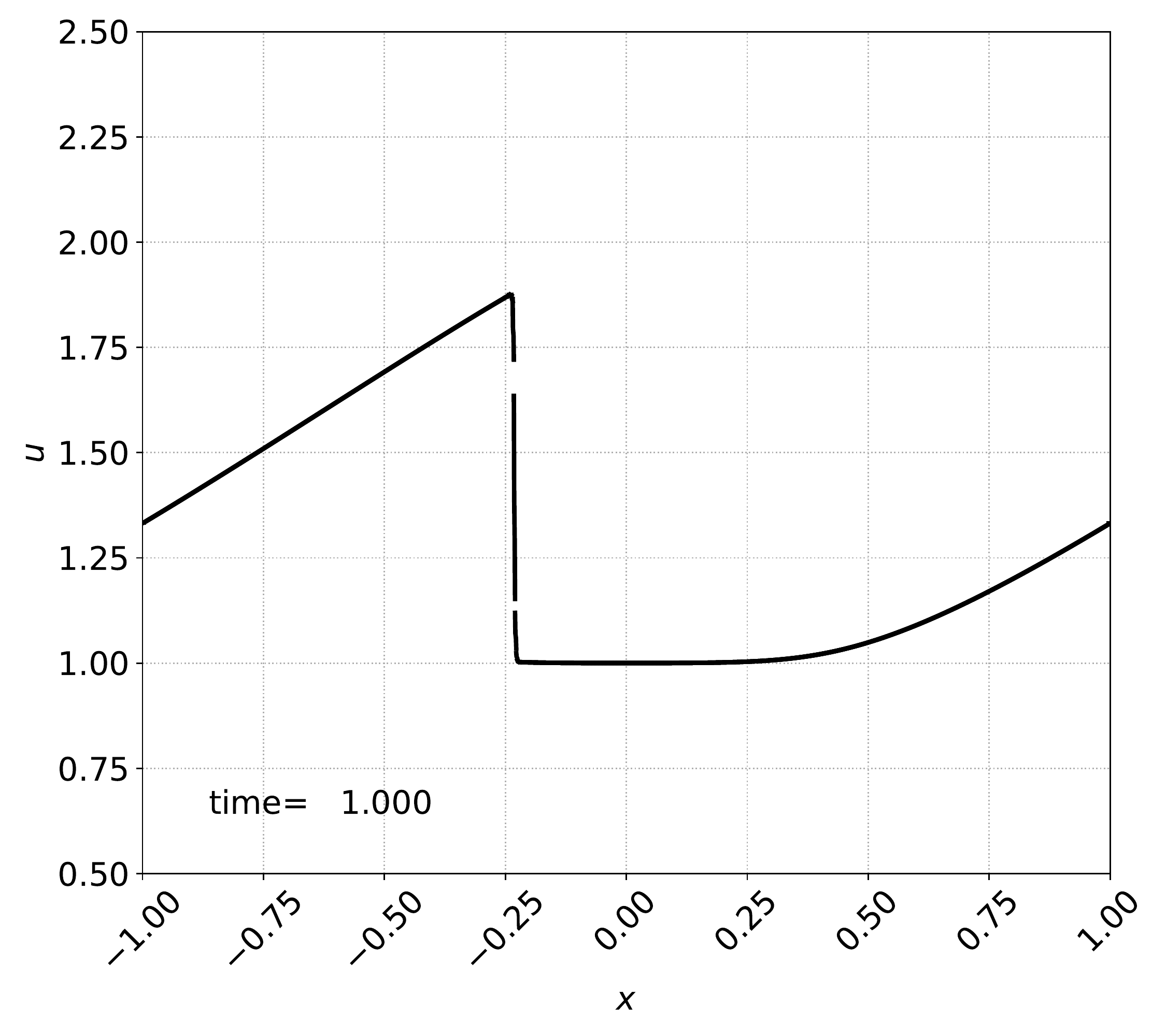}
    \caption{ARK2 w/ limiter}
    \figlab{burgers-lf-ark-t1-ark2-cs17}
  \end{subfigure} 
  \begin{subfigure}[snapshot]{0.45\linewidth}
    \includegraphics[trim=1.5cm 1.0cm 1.5cm 0.0cm, width=0.9\columnwidth]{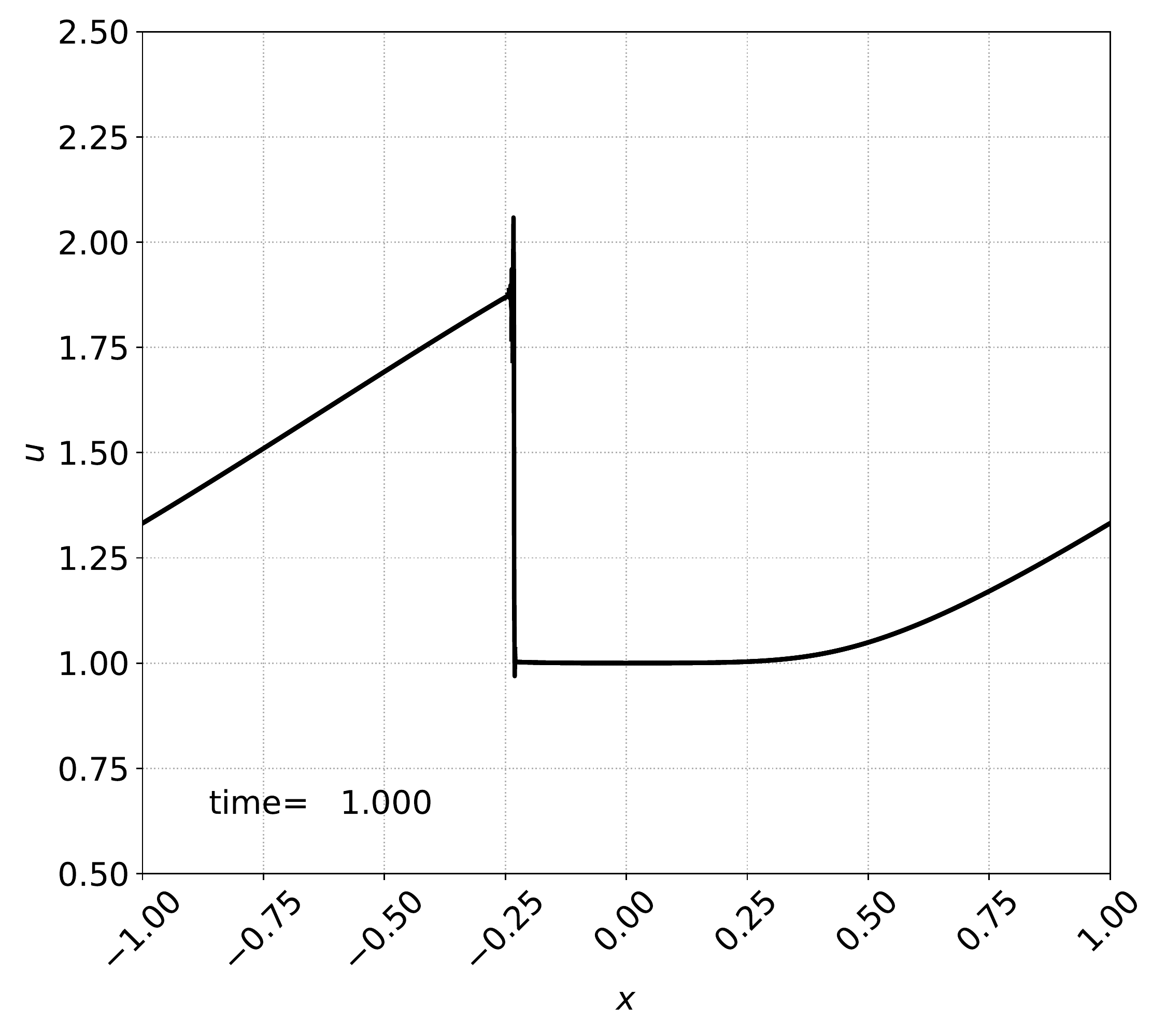}
    \caption{Relaxation-ARK2}
    \figlab{burgers-lf-ark-t1-relaxed-ark2}
  \end{subfigure} 
  \begin{subfigure}[snapshot]{0.45\linewidth}
    \includegraphics[trim=1.5cm 1.0cm 1.5cm 0.0cm, width=0.9\columnwidth]{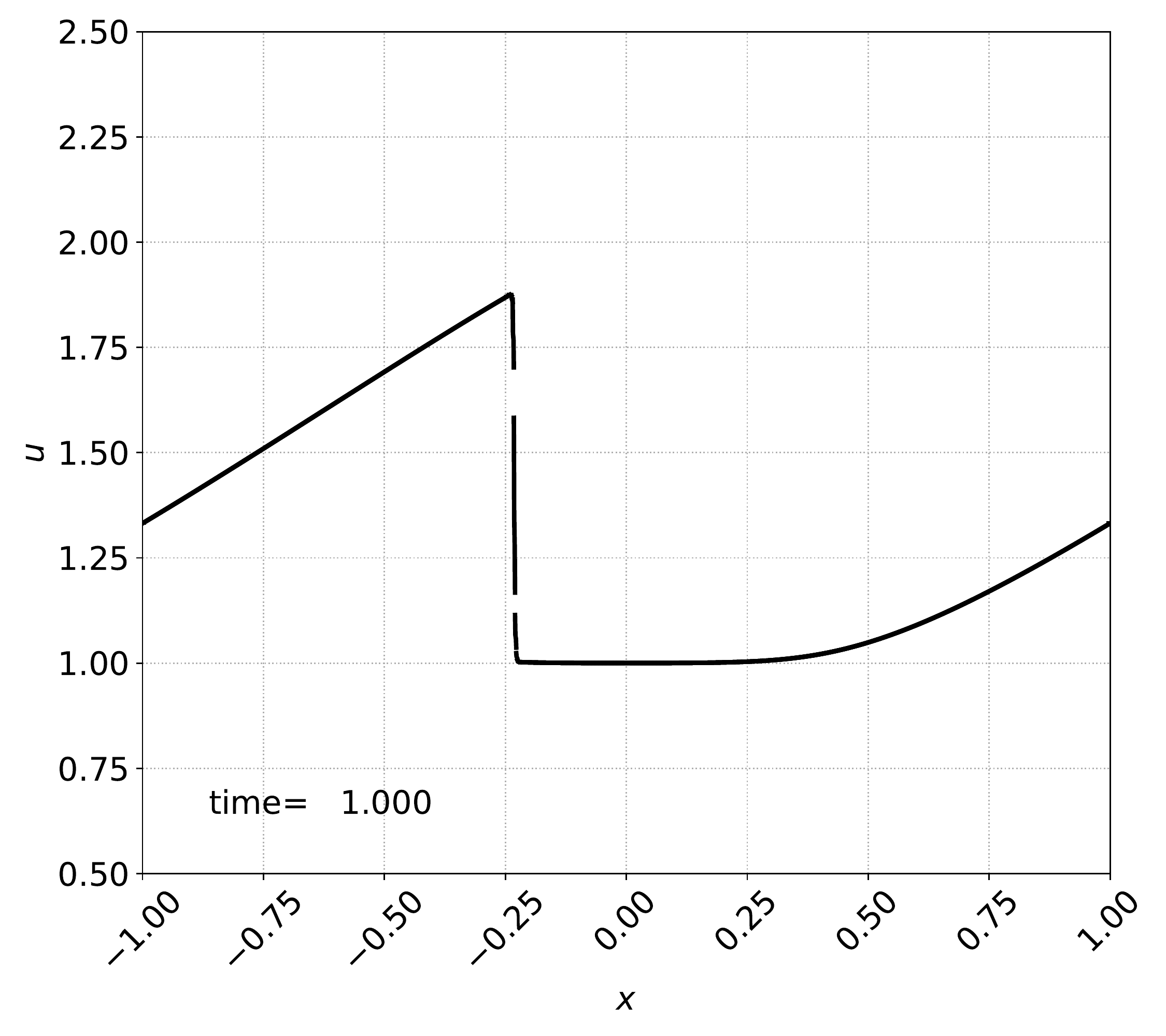}
    \caption{Relaxation-ARK2 w/ limiter}
    \figlab{burgers-lf-ark-t1-relaxed-ark2-cs17}
  \end{subfigure} 
  \begin{subfigure}[snapshot]{0.45\linewidth}
    \includegraphics[trim=1.5cm 1.0cm 1.5cm 0.0cm, width=0.9\columnwidth]{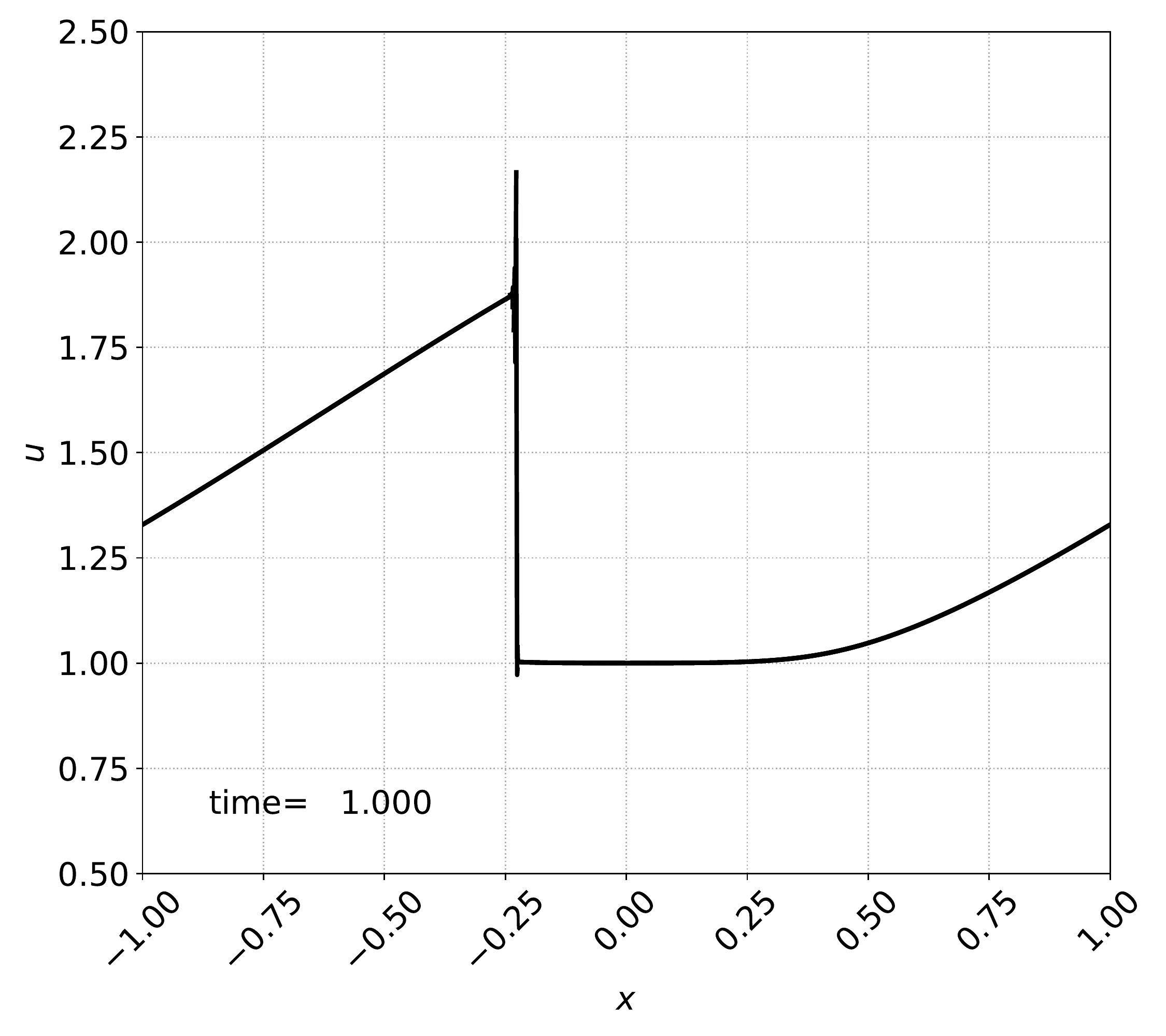}
    \caption{IDT-ARK2}
    \figlab{burgers-lf-ark-t1-relaxed-ark2-idt}
  \end{subfigure} 
  \begin{subfigure}[snapshot]{0.45\linewidth}
    \includegraphics[trim=1.5cm 1.0cm 1.5cm 0.0cm, width=0.9\columnwidth]{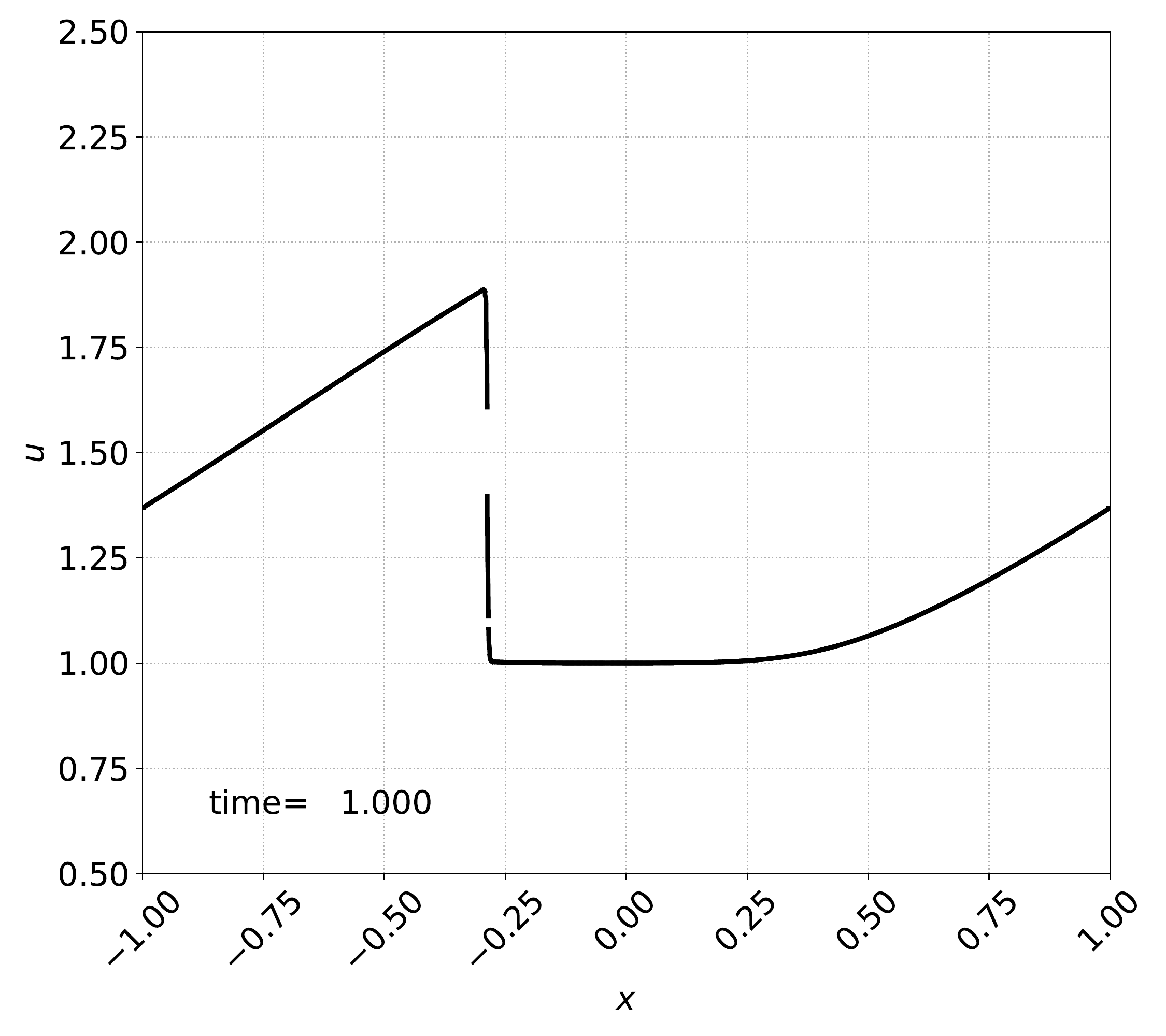}
    \caption{IDT-ARK2 w/ limiter}
    \figlab{burgers-lf-ark-t1-relaxed-ark2-cs17-idt}
  \end{subfigure} 
  \caption{Snapshots of Gaussian profile for the Burgers equation at $t=1$ with entropy-stable (ES) flux 
  for (a) RK2, (b) RK2 with the limiter,
   (c) ARK2, (d) ARK2 with the limiter, 
   (e) Relaxation-ARK2, (f) Relaxation-ARK2 with the limiter, 
   (g) IDT-ARK2, and (h) IDT-ARK2 with the limiter.
  The time step size of RK2 is taken as $\dt_{RK}=2.5 \times 10^{-4}$, 
  whereas the time step sizes of other methods have $2.5\times \dt_{RK}$.
  The domain is discretized with a uniform mesh of $N=3$ and $N_E=800$. 
  }
  \figlab{pde-burgers-lf-ss-ark-comparison}
\end{figure}

In Figure \figref{pde-burgers-lf-totalenergyhistory-ark} the time histories of the total energy and its difference are reported 
for 
ARK2, Relaxation-ARK2, IDT-ARK2,
ARK3, Relaxation-ARK3, and IDT-ARK3 
with/without the limiter.
(The RK2 result is also reported for comparison.)
All the methods with ES flux show entropy-stable behaviors regardless of applying the limiter.
This observation agrees with the work in \cite[Theorem 3.8]{chen2017entropy}. 
   
\begin{figure} 
  \centering
  \begin{subfigure}[Total energy history (ES)]{0.45\linewidth}
      \includegraphics[trim=0.2cm 0.2cm 0.2cm 0.2cm, width=0.95\columnwidth]{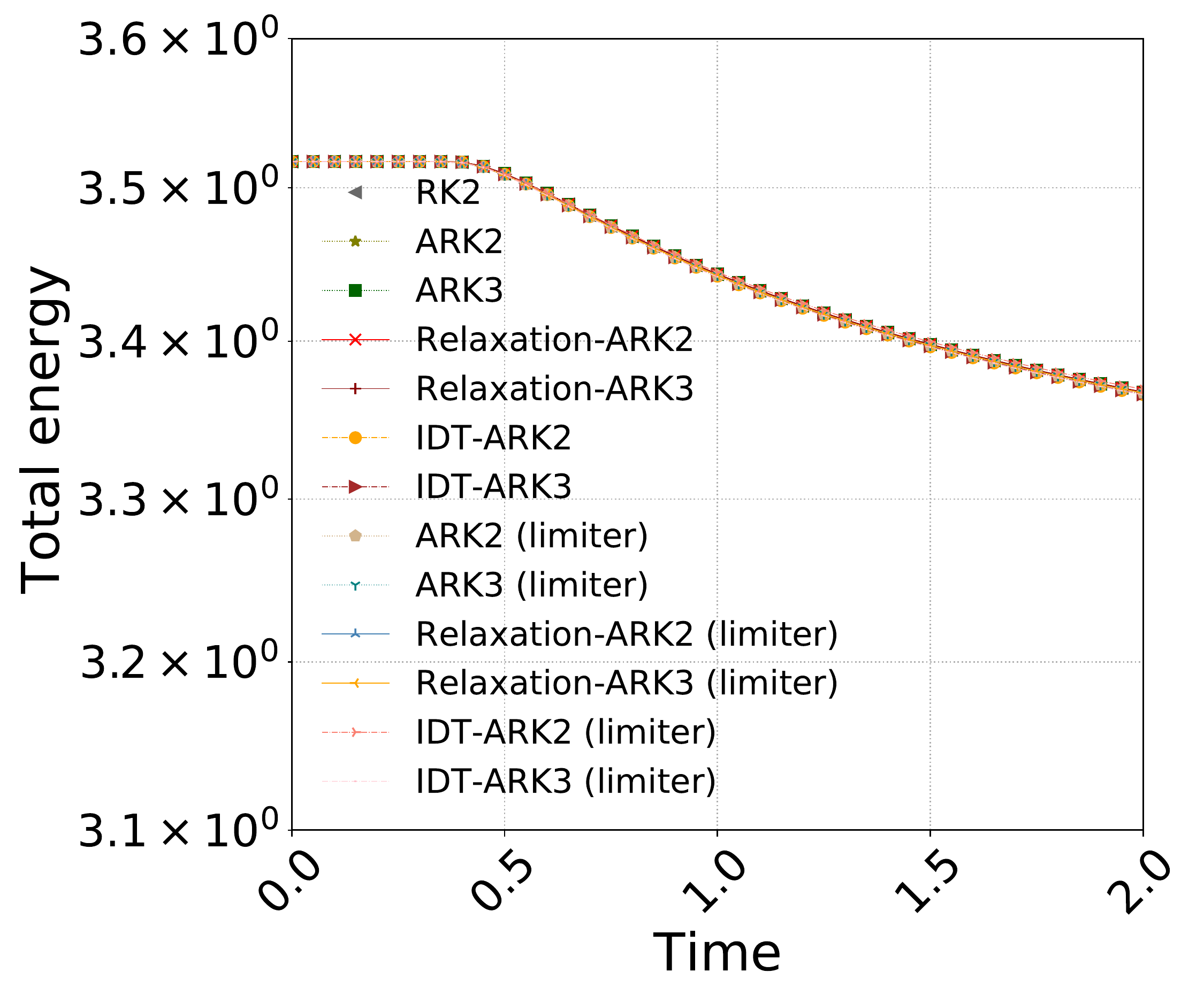}
      \caption{Total energy history (ES)}
      \figlab{pde-burgers-lf-energyloss-ark}
  \end{subfigure} %
  \begin{subfigure}[Total energy difference (ES)]{0.45\linewidth}    
      \includegraphics[trim=0.2cm 0.2cm 0.2cm 0.2cm, width=0.95\columnwidth]{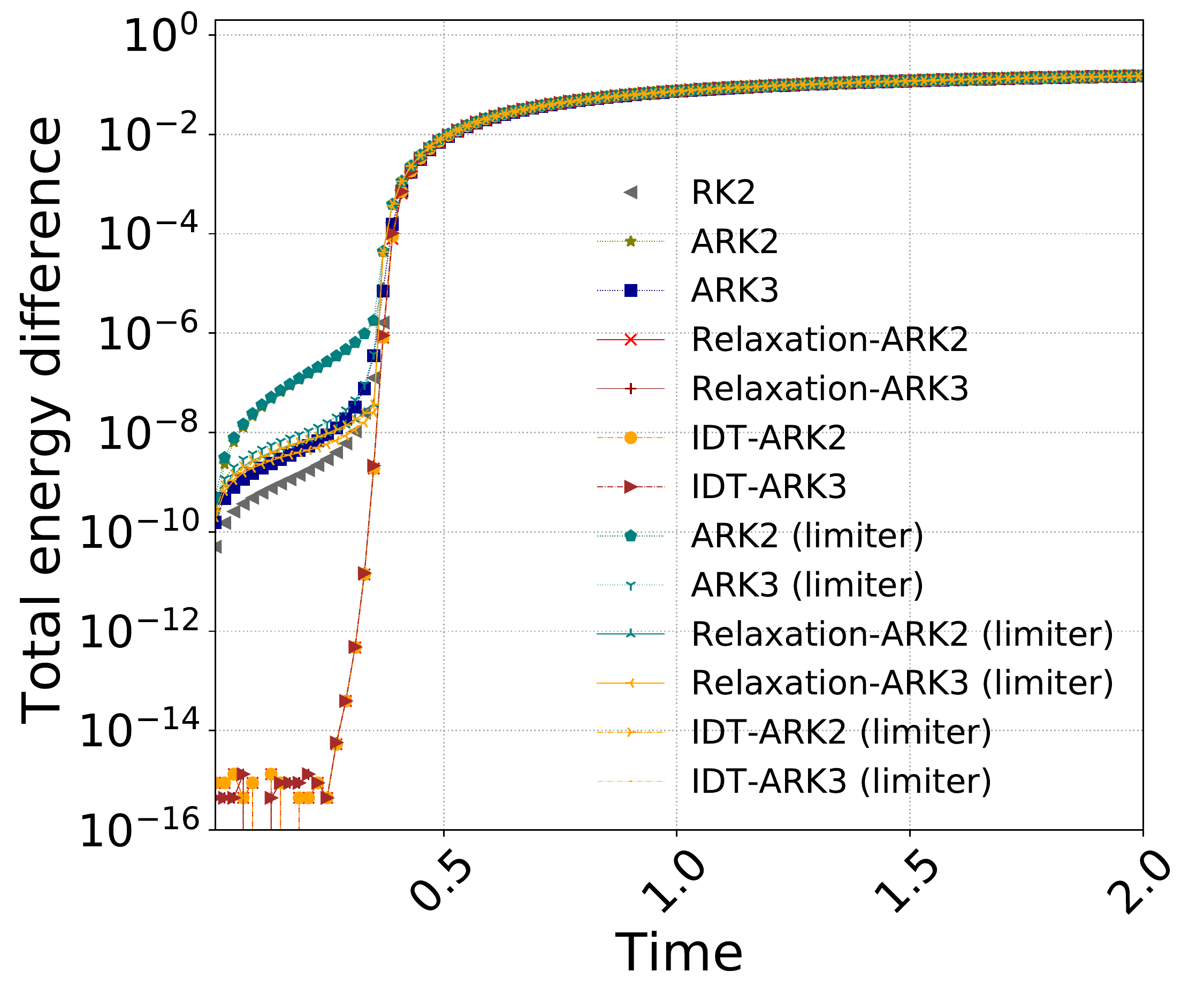}
      \caption{Total energy difference (ES)}
      \figlab{pde-burgers-lf-energyhistory-ark}
  \end{subfigure} 
  \caption{Gaussian example for the Burgers equation: histories of total energy and its difference with ES flux for ARK methods (standard, relaxation, and IDT). 
      All the methods with ES flux show entropy-stable behaviors.
       }
   \figlab{pde-burgers-lf-totalenergyhistory-ark}
\end{figure}

In Figure \figref{pde-burgers-lf-totamassloss-ark} we also plot the time series of the total mass (a linear invariant) 
for ARK2, Relaxation-ARK2, IDT-ARK2,
ARK3, Relaxation-ARK3, and IDT-ARK3 
with/without limiter, as well as RK2.
As expected, all the methods preserve the total mass within $\mc{O}(10^{-14})$ difference.

\begin{figure} 
  \centering
  \includegraphics[trim=0.2cm 0.2cm 0.2cm 0.2cm, width=0.95\columnwidth]{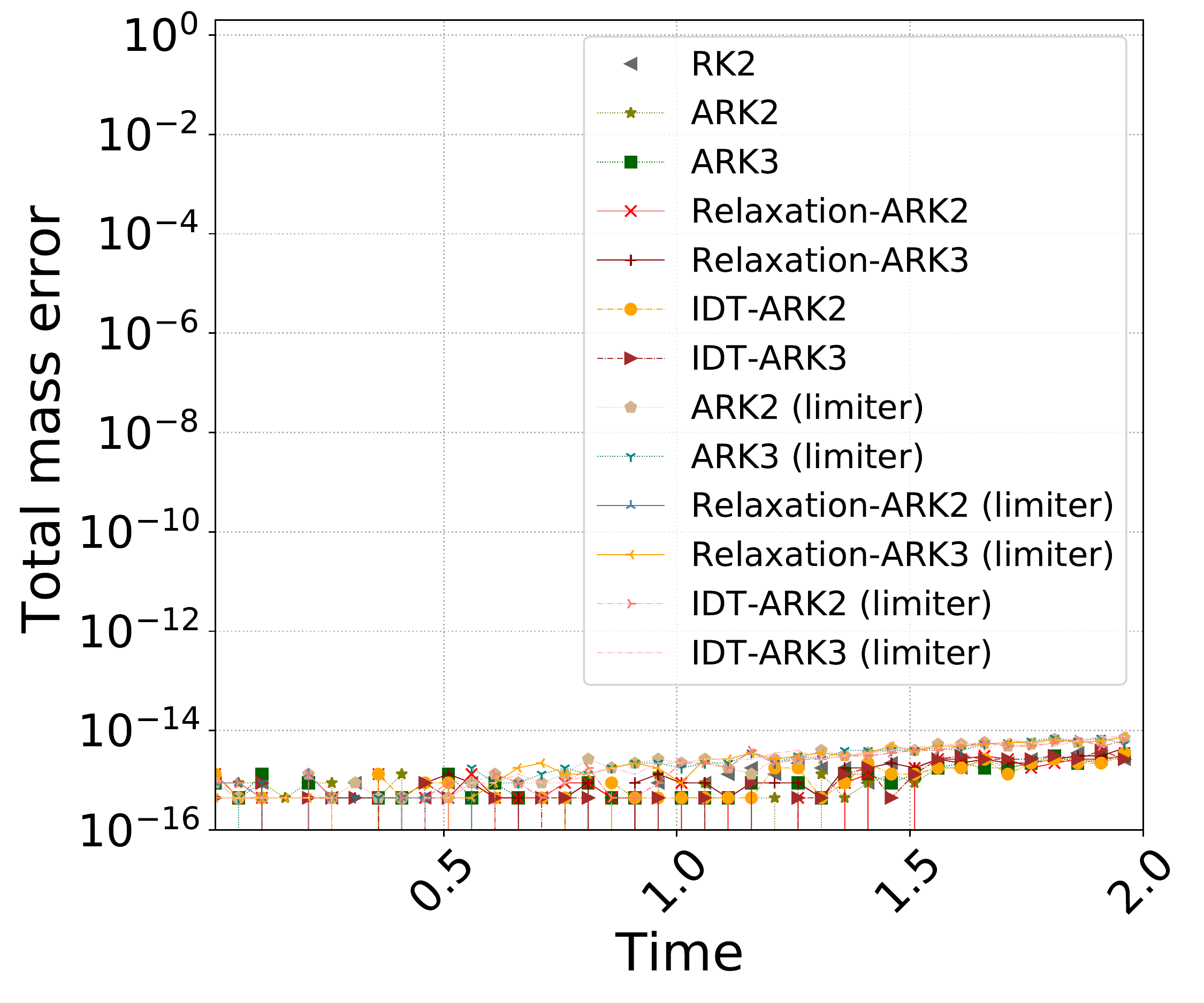}
  \caption{Histories of total mass difference of Gaussian example for the Burgers equation with ES flux: 
  all the methods preserve the total mass within $\mc{O}(10^{-14})$ difference.}
  \figlab{pde-burgers-lf-totamassloss-ark}
\end{figure}

\subsection{Entropy-Stable Multirate Methods for the Burgers Equation on a Nonuniform Mesh}

\begin{figure} 
  \centering  
  \begin{subfigure}[hx]{0.45\linewidth}
    \includegraphics[trim=1.0cm 1.0cm 0.0cm 1.0cm, width=0.95\columnwidth]{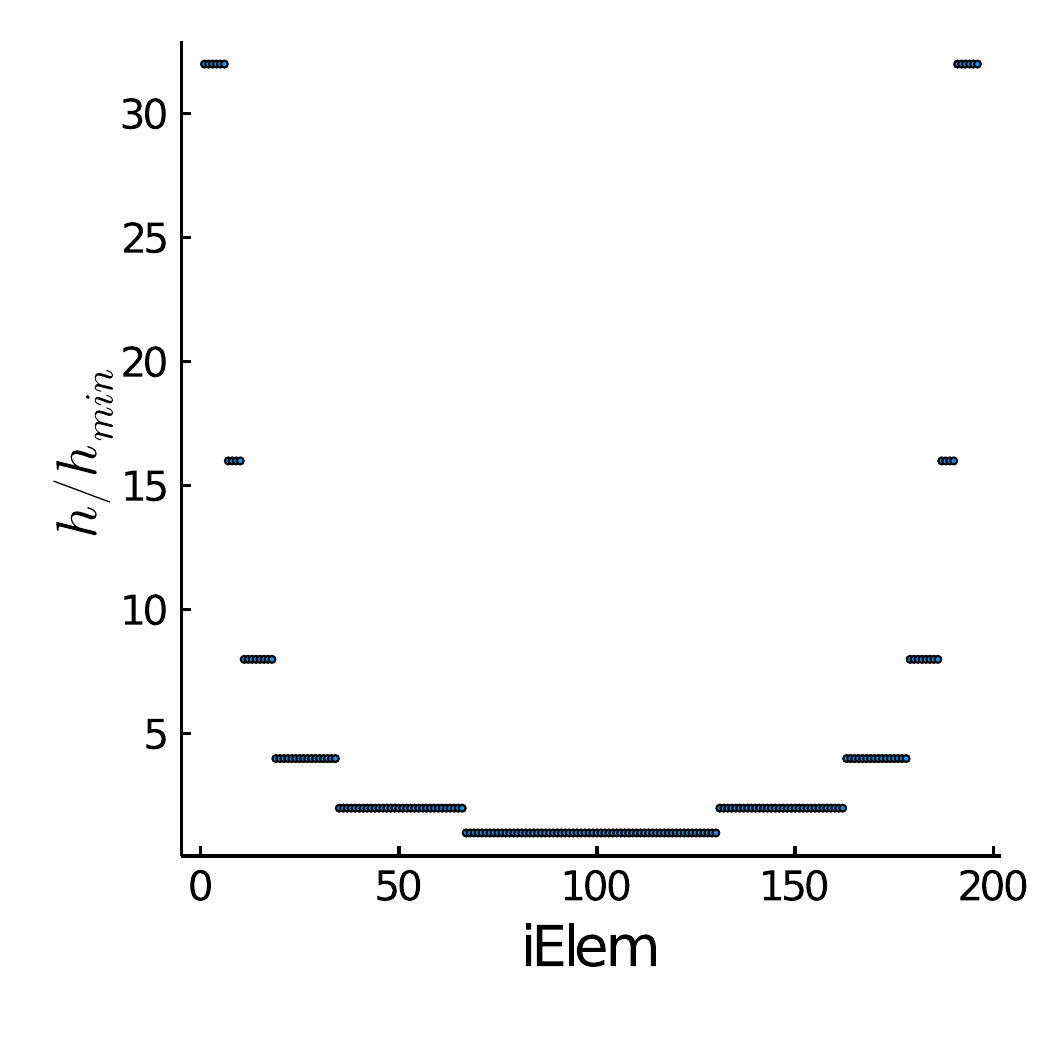}
    \caption{$h/h_{min}$}
    \figlab{pde-burgers-mrk2-ex-hx}
  \end{subfigure} %
  \begin{subfigure}[MR Levels]{0.45\linewidth}
    \includegraphics[trim=0.2cm 1.0cm 0.0cm 1.0cm, width=0.95\columnwidth]{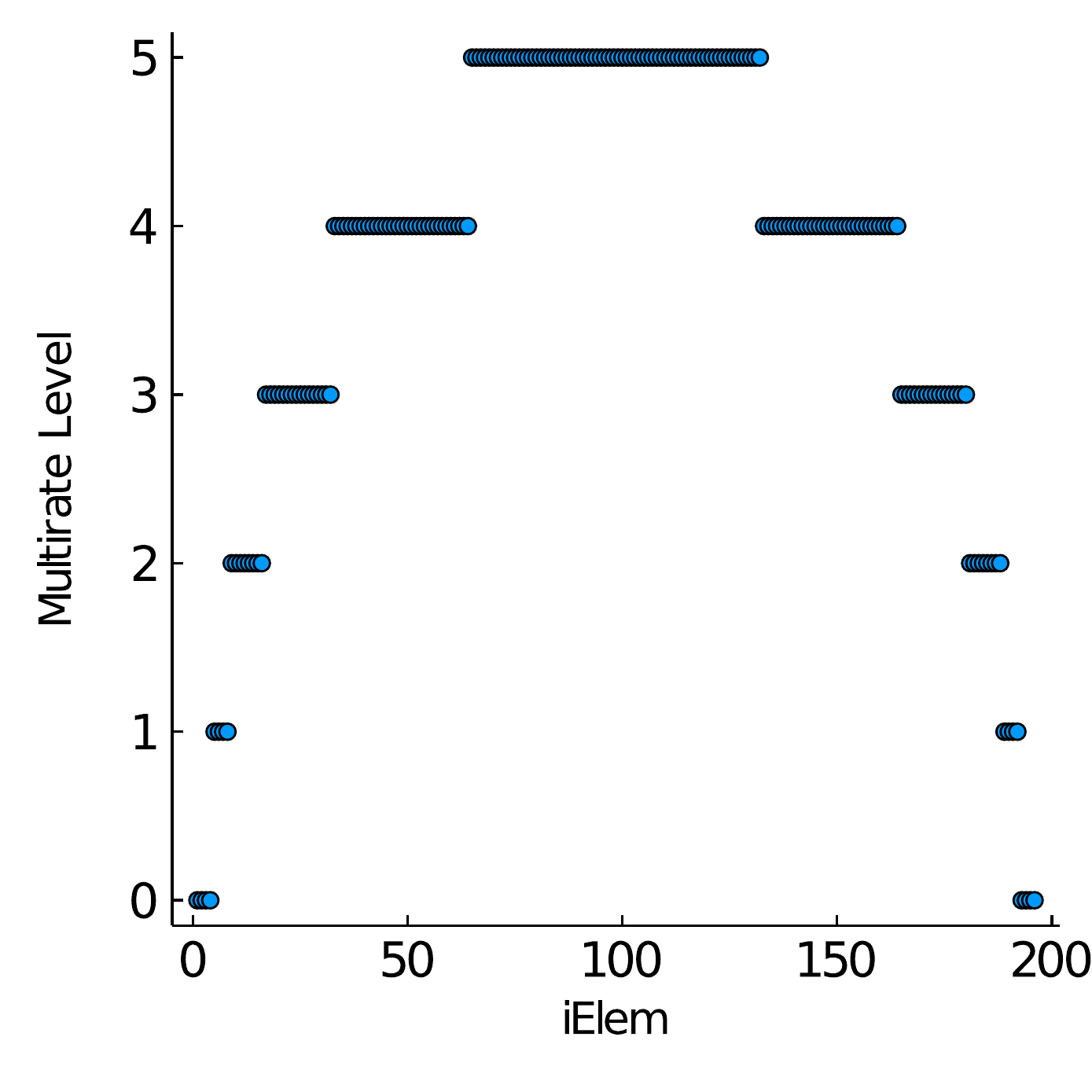}
    \caption{MR levels for $L_{\max}=5$}
    \figlab{pde-burgers-mrk2-ex-mrlev}
  \end{subfigure} %
  \caption{(a) Ratio of element sizes with respect to the minimum size of the elements 
  and (b) multirate levels of elements for $L_{\max}=5$. 
  Here, $L$ is the maximum multirate (MR) level. 
  The center of the domain is five times refined with a $2:1$ ratio so that the largest element is 32 times bigger than the smallest element.
  Multirate level is assigned to each element corresponding to the ratio of element sizes.  }
  \figlab{pde-burgers-mrk2-ex-mrlevels}
\end{figure}

We consider MRK2 methods on a nonuniform mesh for handling geometric-induced stiffness. 
A one-dimensional domain is five times refined at the center of the domain with a 2:1 grid ratio 
so that the biggest element is 32 times larger than the smallest element, as shown in Figure \figref{pde-burgers-mrk2-ex-hx}.
In the  MRK2 algorithm, based on the ratio of the element sizes, multirate levels are assigned to each element 
in Figure \figref{pde-burgers-mrk2-ex-mrlev}, where the highest multirate level is five. 

We first perform temporal convergence studies with the entropy-conserving  and entropy-stable  fluxes 
for MRK2, Relaxation-MRK2, and IDT-MRK2 methods without using the limiter.
We take the RK4 solution (with the fixed step size of $\dt = 5 \times 10^{-6}$, $N=3$, and $N_E=196$) as the ``ground truth" solution 
and measure the relative errors at $t = 0.2$ (before forming a shock) in Table \tabref{burgers-tconv-gaussian-mrk2}. We also report the relative errors at $t=1$ (after forming the shock) for the entropy-stable flux. 

The numerical solutions converge to the reference RK4 solution with second-order accuracy for the MRK2, Relaxation-MRK2, and IDT-MRK2 methods 
regardless of the EC/ES fluxes at $t=0.2$. 
The error differences among MRK2, Relaxation-MRK2, and IDT-MRK2 are within $\mc{O}(10^{-8})$. 
In particular, IDT-MRK2 shows  second-order accuracy in time. 
This is because 
both Relaxation-MRK2 and IDT-MRK2 have tiny relaxation parameters ($\mc{O}(10^{-5})$),
 and the temporal error of IDT-MRK2 is not accumulated enough. 
This agrees with the previous study 
in \cite[Figure 9.]{ketcheson2019relaxation}, where both IDT-RK2 and Relaxation-RK2 show the second-order rate of convergence in time.
However, at $t=1$,
 the error of IDT-MRK2 is at least sixty times larger than that of Relaxation-MRK2. 
 We also observe that the order of temporal accuracy of IDT-MRK2 drops to one with larger time step sizes. The temporal error of IDT-MRK2 has accumulated to the point where the theoretical convergence rate can be seen.
 This agrees with Figure \figref{pde-burgers-lf-ss-ark-comparison} where the location of the shock front for IDT-MRK2 is slightly behind that of Relaxation-MRK2. 

 \begin{table}[t] 
  \caption{Gaussian example: temporal convergence study for MRK2 methods  performed with EC and ES fluxes  
  on a nonuniform mesh of $N=3$ and $K=196$ without using the limiter. 
  We use the time step sizes with $\dt=0.001\LRc{1, 1/2, 1/4, 1/8, 1/16}$ for EC flux 
  and $\dt=0.0025\LRc{1, 1/2, 1/4, 1/8, 1/16}$ for ES flux. 
  By taking the RK4 solution with the fixed step size of $\dt=5.0\times 10^{-6}$ as the ``ground truth" solution, 
  we measure the relative errors of MRK2, Relaxation-MRK2, and IDT-MRK2 methods at $t=0.2$ (before forming a shock). We also report the relative errors at $t=1.0$ (after forming the shock) for ES flux. } 
  \tablab{burgers-tconv-gaussian-mrk2} 
  \begin{center} 
    \begin{tabular}{*{1}{c}|*{1}{c}|*{2}{c}|*{2}{c}|*{2}{c}} 
      \hline 
      \multirow{2}{*}{$flux$}
      & \multirow{2}{*}{$\dt$}
      & \multicolumn{2}{c}{MRK2} 
      & \multicolumn{2}{c}{Relaxation-MRK2 } 
      & \multicolumn{2}{c}{IDT-MRK2} \tabularnewline 
      & & Error & Order &Error & Order &Error & Order \tabularnewline 
      \hline\hline 
        &1.000e-03&       5.67E-06 &     $-$&       5.68E-06 &     $-$&       5.59E-06 & $-$\tabularnewline
        &5.000e-04&       1.43E-06 &    1.98&       1.43E-06 &    1.99&       1.41E-06 &    1.98\tabularnewline
EC      &2.500e-04&       3.61E-07 &    1.99&       3.61E-07 &    1.99&       3.56E-07 &    1.99\tabularnewline
(t=0.2) &1.250e-04&       9.05E-08 &    2.00&       9.06E-08 &    2.00&       8.94E-08 &    1.99\tabularnewline
        &6.250e-05&       2.27E-08 &    2.00&       2.27E-08 &    2.00&       2.24E-08 &    2.00\tabularnewline    
      \tabularnewline 
        &2.500e-03&       7.43E-05 &     $-$&       7.42E-05 &     $-$&       9.22E-05 &      $-$\tabularnewline
        &1.250e-03&       1.68E-05 &    2.14&       1.68E-05 &    2.14&       1.78E-05 &    2.37\tabularnewline
    ES  &6.250e-04&       4.02E-06 &    2.07&       4.02E-06 &    2.07&       4.04E-06 &    2.14\tabularnewline
(t=0.2) &3.125e-04&       9.83E-07 &    2.03&       9.83E-07 &    2.03&       9.77E-07 &    2.05\tabularnewline
        &1.563e-04&       2.43E-07 &    2.02&       2.43E-07 &    2.02&       2.41E-07 &    2.02\tabularnewline

      \tabularnewline 
          &2.500e-03&       1.46E-03 &     $-$&       1.59E-03 &     $-$&       9.68E-02 &     $-$\tabularnewline
          &1.250e-03&       3.44E-04 &    2.08&       3.01E-04 &    2.40&       5.00E-02 &    0.95\tabularnewline
     ES   &6.250e-04&       8.35E-05 &    2.04&       1.29E-04 &    1.22&       1.36E-02 &    1.88\tabularnewline
   (t=1.0)&3.125e-04&       2.06E-05 &    2.02&       2.25E-05 &    2.53&       3.31E-03 &    2.04\tabularnewline
          &1.563e-04&       5.11E-06 &    2.01&       5.18E-06 &    2.12&       8.22E-04 &    2.01\tabularnewline
            
      \hline\hline 
      \end{tabular} 
  \end{center}     
\end{table}

Next, we examine the entropy conservation of MRK2 methods. 
We perform the simulations for $t\in \LRs{0,2}$ with $N=3$ and $K=784$ ($L_{\max}=5$). 
The time step size of RK2 is taken as $\dt_{RK}=6.25 \times 10^{-6}$, 
whereas those of MRK2, Relaxation-MRK2, and IDT-MRK2 have $\dt=20\times \dt_{RK}$.
\footnote{
  RK2 with $\dt_{RK}=1.25\times 10^{-6}$ leads to blowup of its numerical solution.
}
Figure \figref{pde-burgers-ec-totalenergyhistory-mrk2} shows the time histories of the total energy and its difference 
for the RK2, MRK2, Relaxation-RK2, Relaxation-MRK2, IDT-RK2, and IDT-MRK2 methods. 
We see that both the relaxation and the IDT methods preserve the total energy during the simulation. 
The difference between the total energy for the relaxation and the IDT methods is around $\mc{O}(10^{-13})$, 
whereas the standard RK2 and MRK2 counterparts increase to $\mc{O}(10^{-1})$. 
\begin{figure} 
  \centering
  \begin{subfigure}[Total Energy History (EC)]{0.45\linewidth}
      \includegraphics[trim=0.2cm 0.2cm 0.2cm 0.2cm, width=0.95\columnwidth]{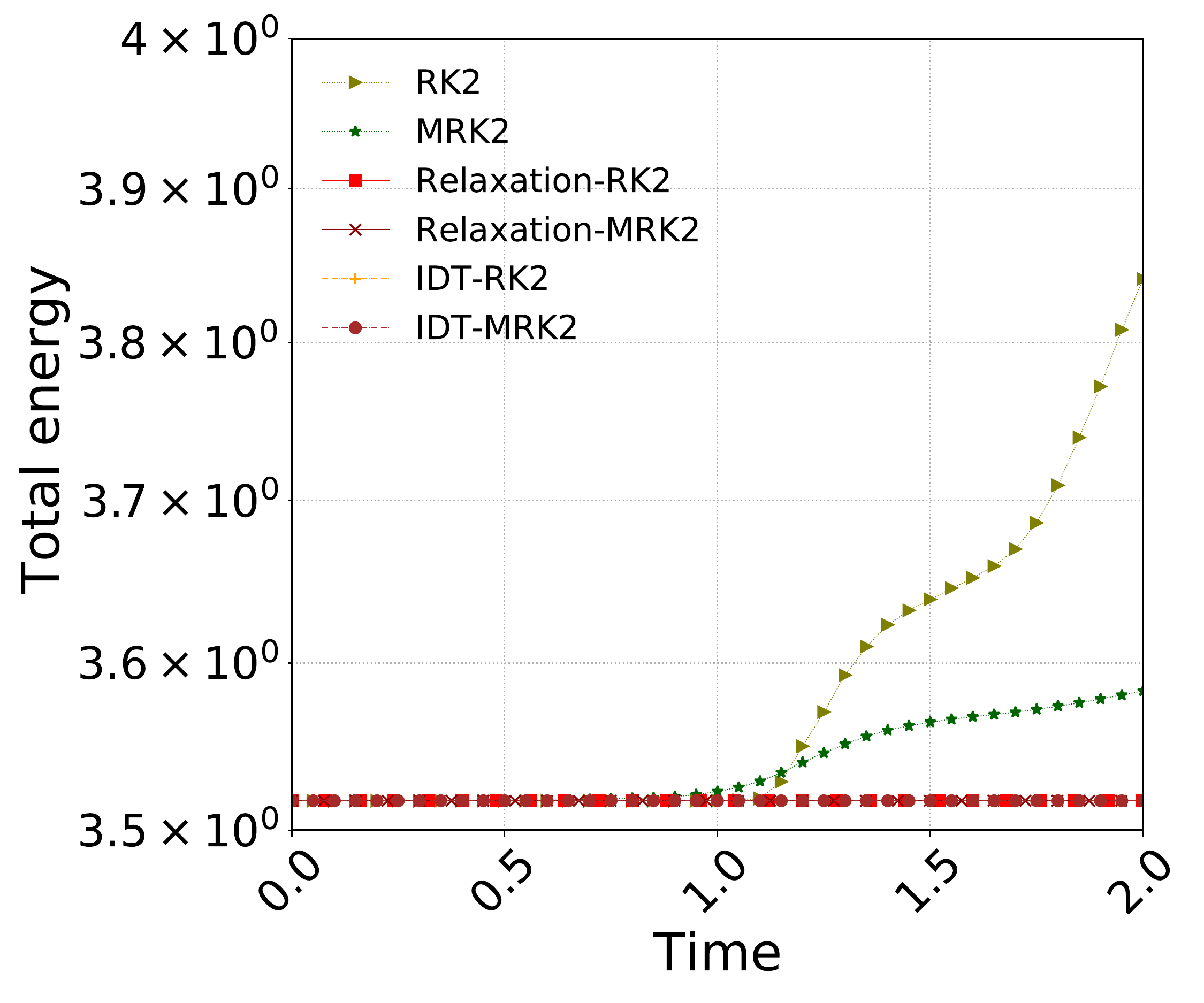}
      \caption{Total Energy History (EC)}
      \figlab{pde-burgers-ec-energyloss-mrk2}
  \end{subfigure} %
  \begin{subfigure}[Total energy difference (EC)]{0.45\linewidth}    
      \includegraphics[trim=0.2cm 0.2cm 0.2cm 0.2cm, width=0.95\columnwidth]{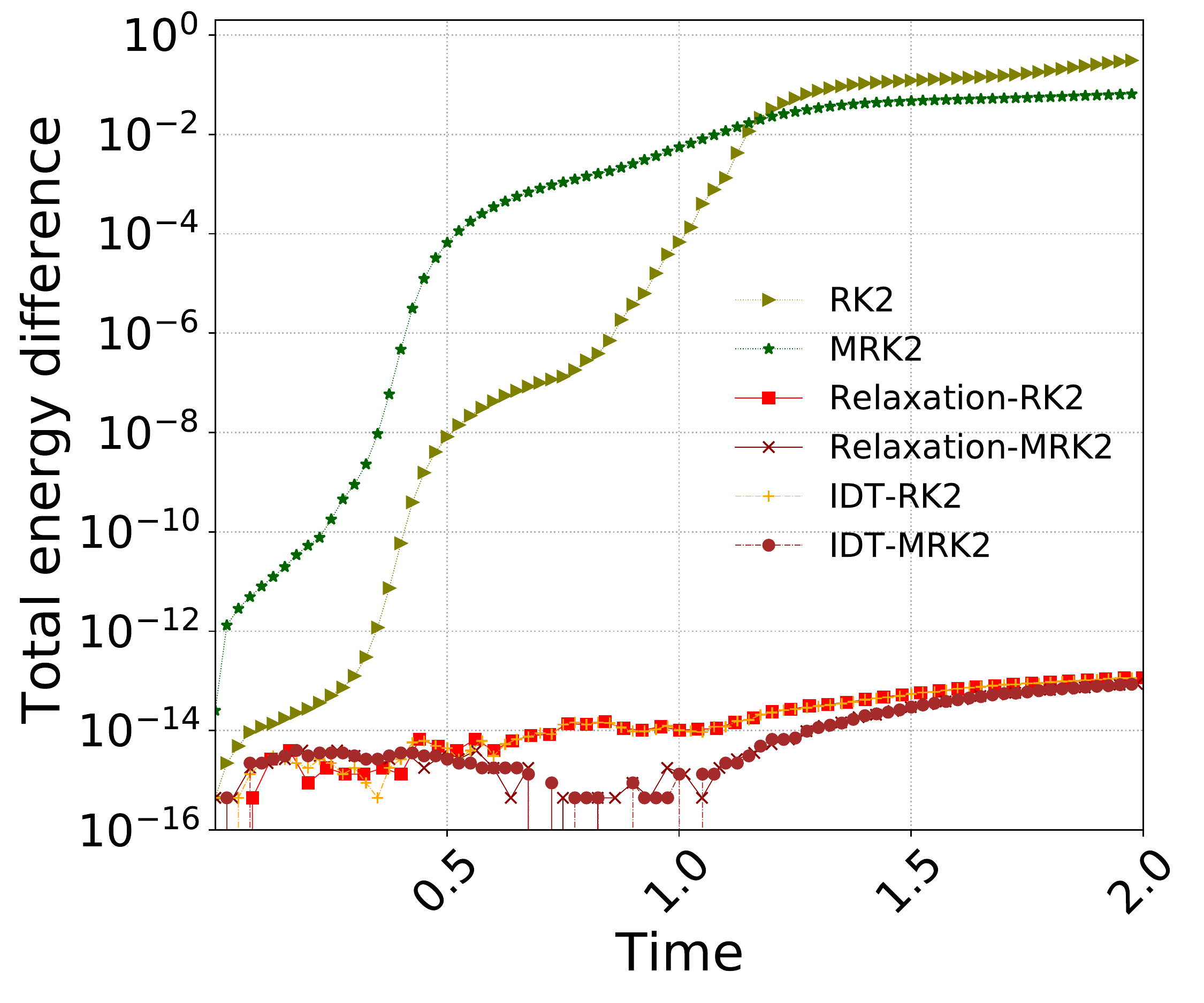}
      \caption{Total energy difference (EC)}
      \figlab{pde-burgers-ec-energyhistory-mrk2}
  \end{subfigure} 
  \caption{Histories of total energy and its difference of Gaussian example for the Burgers equation with EC flux: 
      the relaxation methods with EC flux conserve the total energy within $\mc{O}(10^{-13})$.  }
   \figlab{pde-burgers-ec-totalenergyhistory-mrk2}
\end{figure}
We also show the snapshots at $t=1$ in Figure \figref{pde-burgers-ec-ss-mrk2-comparison}.
As expected, high oscillatory noises are observed, 
but numerical solutions are still stable.

\begin{figure} 
  \centering  
  \begin{subfigure}[snapshot]{0.45\linewidth}
    \includegraphics[trim=1.5cm 1.0cm 1.5cm 1.0cm, width=0.9\columnwidth]{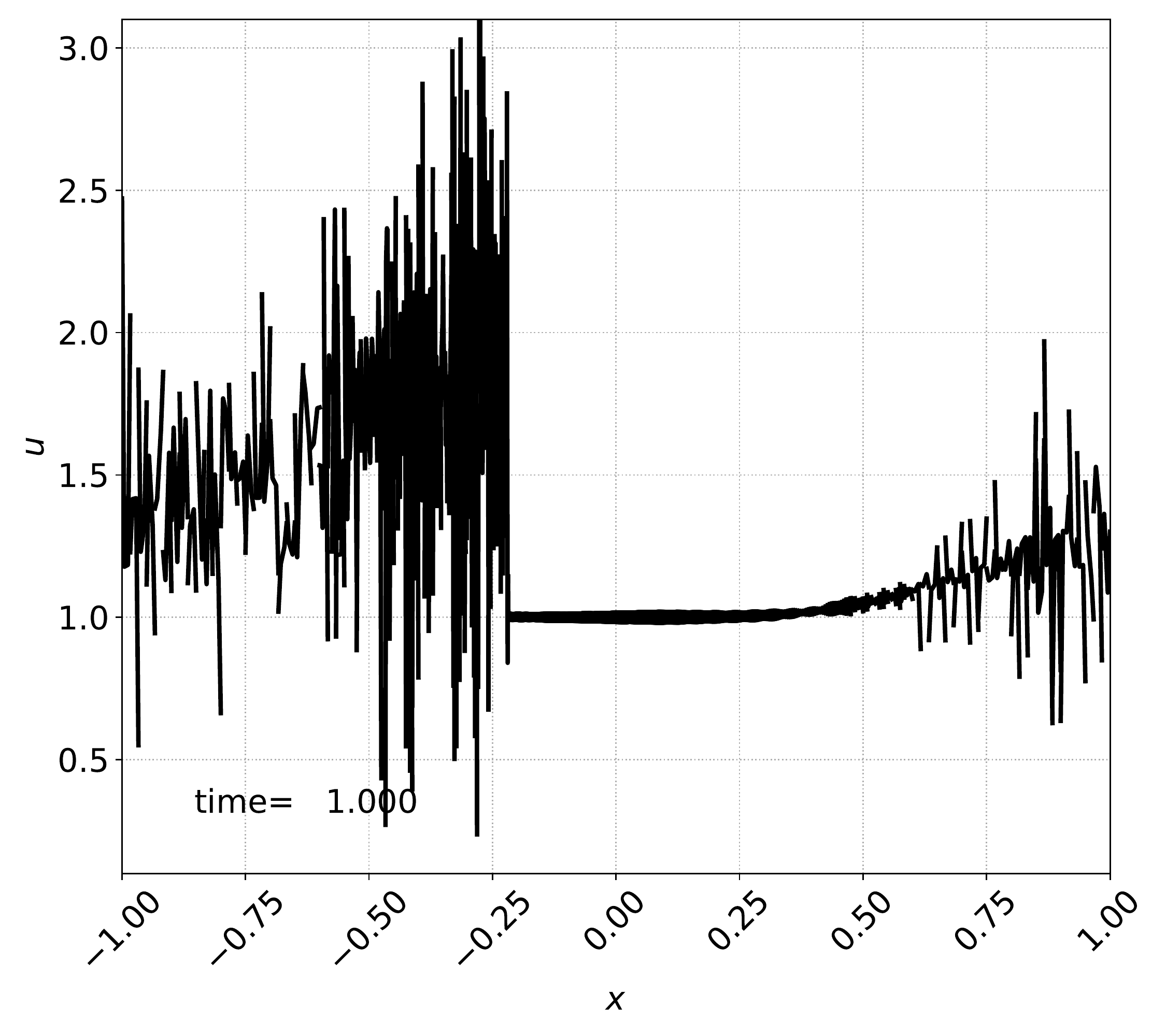}
    \caption{RK2 }
    \figlab{burgers-ec-t1-rk2}
  \end{subfigure} %
  \begin{subfigure}[snapshot]{0.45\linewidth}
    \includegraphics[trim=1.5cm 1.0cm 1.5cm 1.0cm, width=0.9\columnwidth]{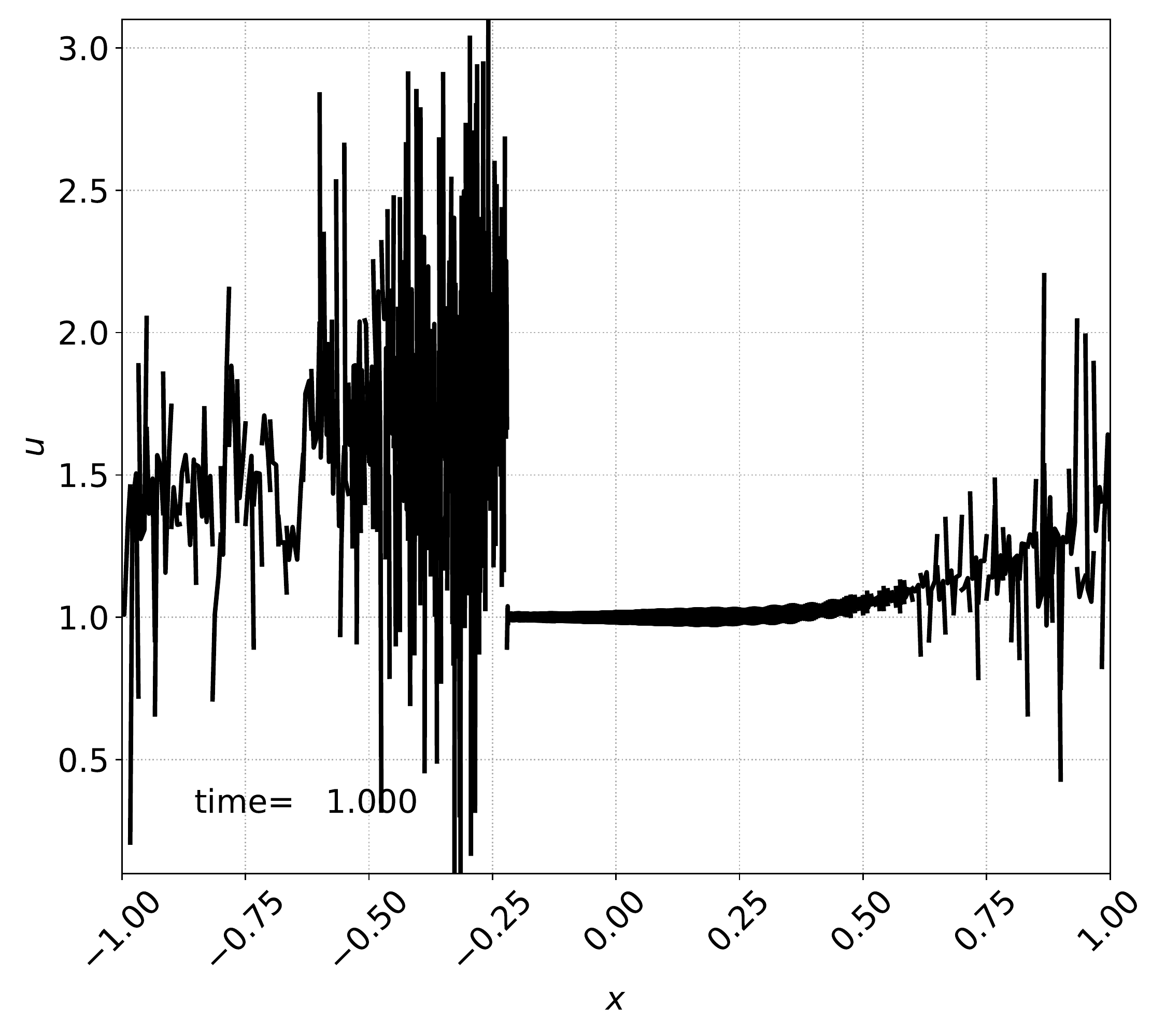}
    \caption{MRK2}
    \figlab{burgers-ec-t1-mrk2}
  \end{subfigure} 
  \begin{subfigure}[snapshot]{0.45\linewidth}
    \includegraphics[trim=1.5cm 1.0cm 1.5cm 0.0cm, width=0.9\columnwidth]{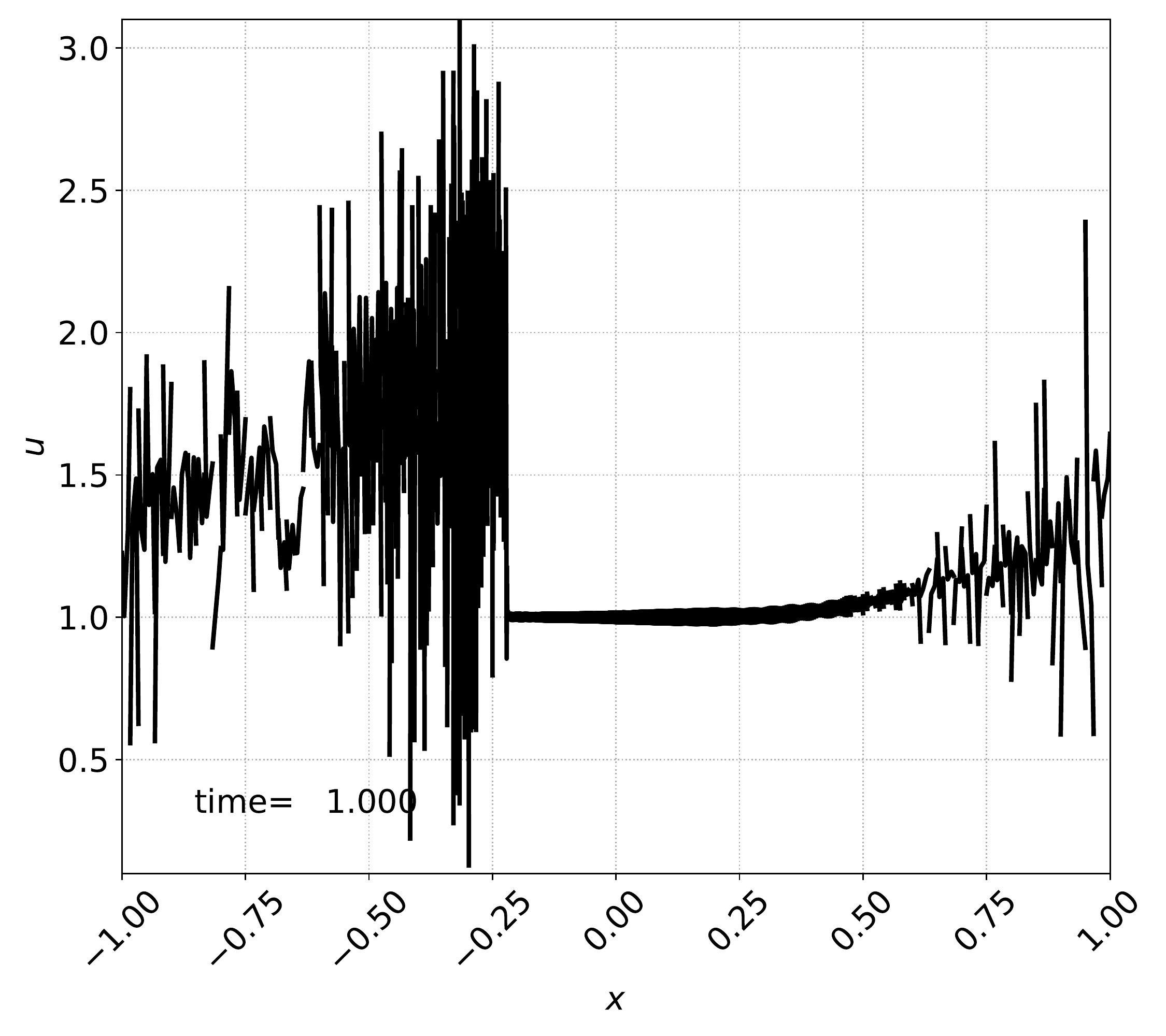}
    \caption{Relaxation-MRK2}
    \figlab{burgers-ec-t1-relaxmrk2}
  \end{subfigure} 
  \begin{subfigure}[snapshot]{0.45\linewidth}
    \includegraphics[trim=1.5cm 1.0cm 1.5cm 0.0cm, width=0.9\columnwidth]{figures/burgers-ec-pde-r-010-relaxmrk2}
    \caption{IDT-MRK2}
    \figlab{burgers-ec-t1-relaxmrk2-idt}
  \end{subfigure} 
  \caption{Snapshots of Gaussian profile for the Burgers equation at $t=1$ on a nonuniform mesh with energy-conserving (EC) flux for (a) RK2, (b) MRK2, (c) Relaxation-MRK2, and (d) IDT-MRK2.
  The time step size of RK2 is taken as $\dt_{RK}=6.25 \times 10^{-6}$, 
  whereas those of MRK2, Relaxation-MRK2, and IDT-MRK2 have $\dt=20\times \dt_{RK}$.
  The domain is discretized with a nonuniform mesh of $N=3$ and $K=784$ ($L_{\max}=5$). 
  }
  \figlab{pde-burgers-ec-ss-mrk2-comparison}
\end{figure}

Now we examine the entropy stability of MRK2 methods with ES flux. 
We perform the simulations for $t\in \LRs{0,2}$ with $N=3$ and $K=784$ ($L_{\max}=5$). 
The time step size of RK2 is taken as $\dt_{RK}=5 \times 10^{-5}$, 
whereas the time step sizes of MRK2, Relaxation-MRK2, and IDT-MRK2 have $\dt=25\times \dt_{RK}$.
\footnote{
  RK2 with $\dt_{RK}=6.25\times 10^{-5}$ yields a blowup solution.
}
The snapshots at $t=1$ are reported in Figure \figref{pde-burgers-lf-ss-mrk2-comparison}.
Similar to Figure \figref{pde-burgers-lf-ss-ark-comparison}, 
the IDT method suffers from phase errors. 
The shock front of IDT-MRK2 is slightly lagged behind, and 
the error becomes severe when the limiter is applied.  
This example demonstrates that the relaxation approach 
is better than the IDT approach in terms of accuracy, especially when the limiter is applied. 
 

\begin{figure} 
  \centering  
  \begin{subfigure}[snapshot]{0.45\linewidth}
    \includegraphics[trim=1.5cm 1.0cm 1.5cm 1.0cm, width=0.9\columnwidth]{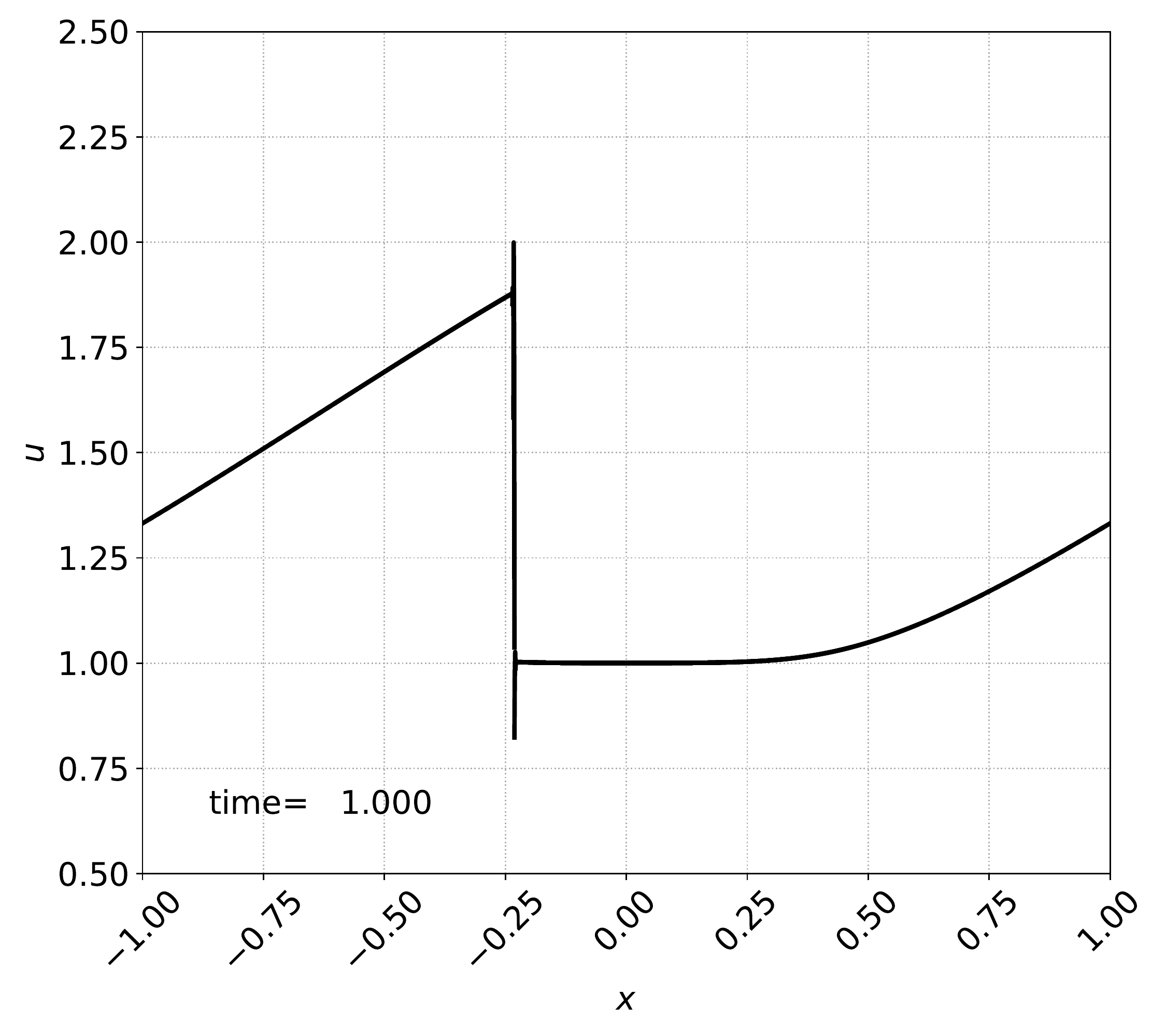}
    \caption{RK2 }
    \figlab{burgers-lf-t1-rk2}
  \end{subfigure} %
  \begin{subfigure}[snapshot]{0.45\linewidth}
    \includegraphics[trim=1.5cm 1.0cm 1.5cm 1.0cm, width=0.9\columnwidth]{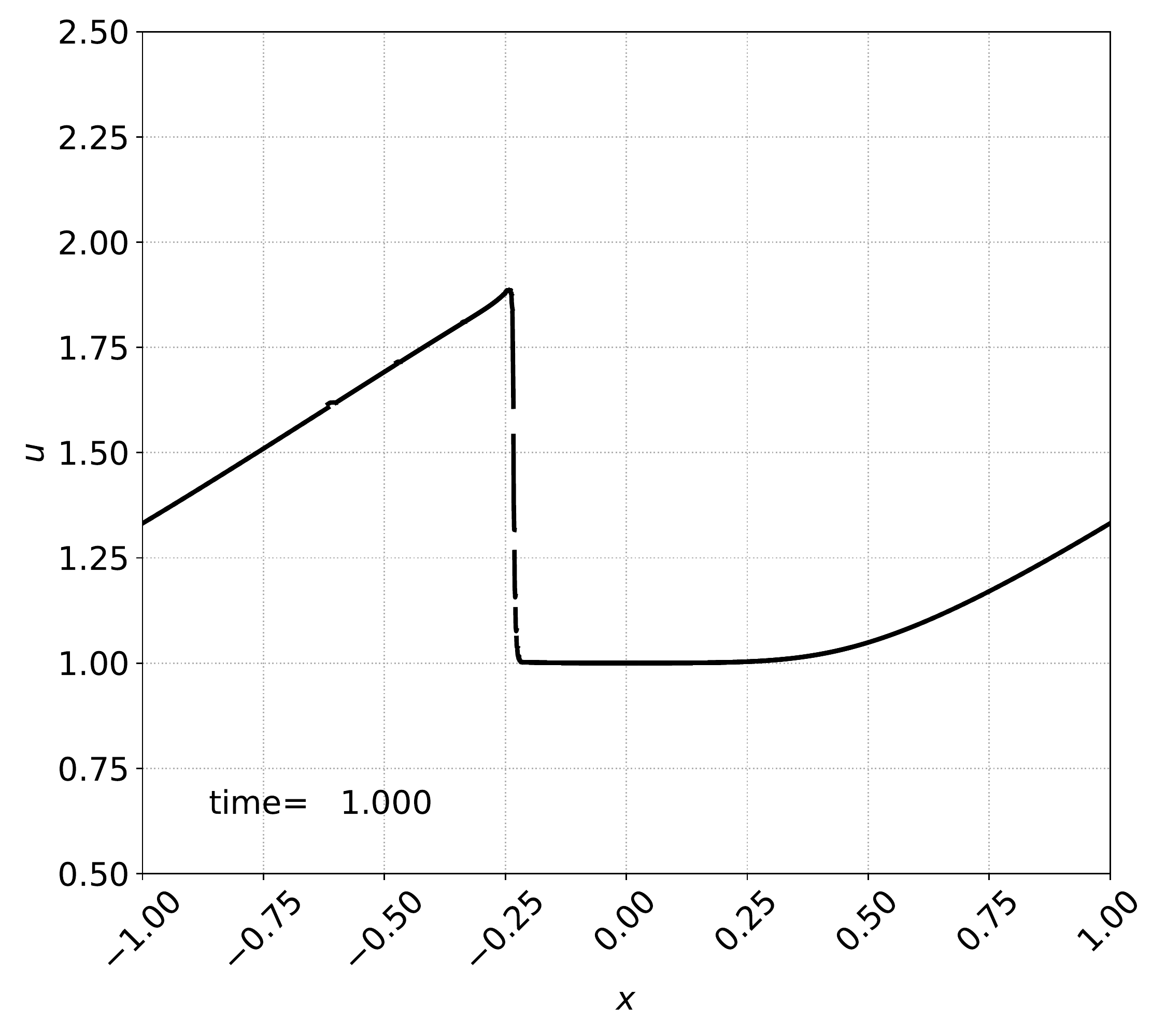}
    \caption{RK2 w/ limiter}
    \figlab{burgers-lf-t1-rk2-cs17}
  \end{subfigure} 
  \begin{subfigure}[snapshot]{0.45\linewidth}
    \includegraphics[trim=1.5cm 1.0cm 1.5cm 0.0cm, width=0.9\columnwidth]{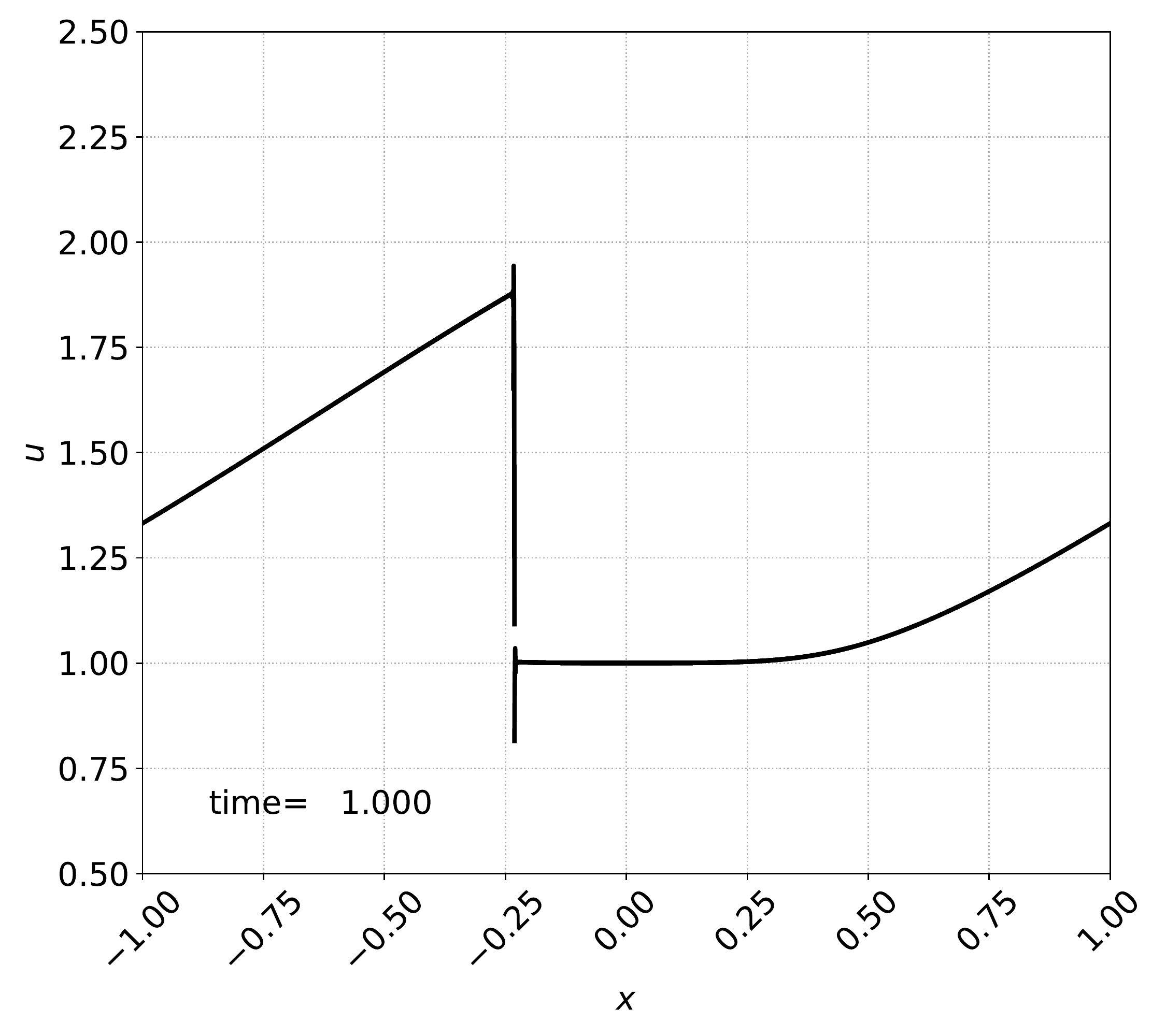}
    \caption{MRK2}
    \figlab{burgers-lf-t1-mrk2}
  \end{subfigure} 
  \begin{subfigure}[snapshot]{0.45\linewidth}
    \includegraphics[trim=1.5cm 1.0cm 1.5cm 0.0cm, width=0.9\columnwidth]{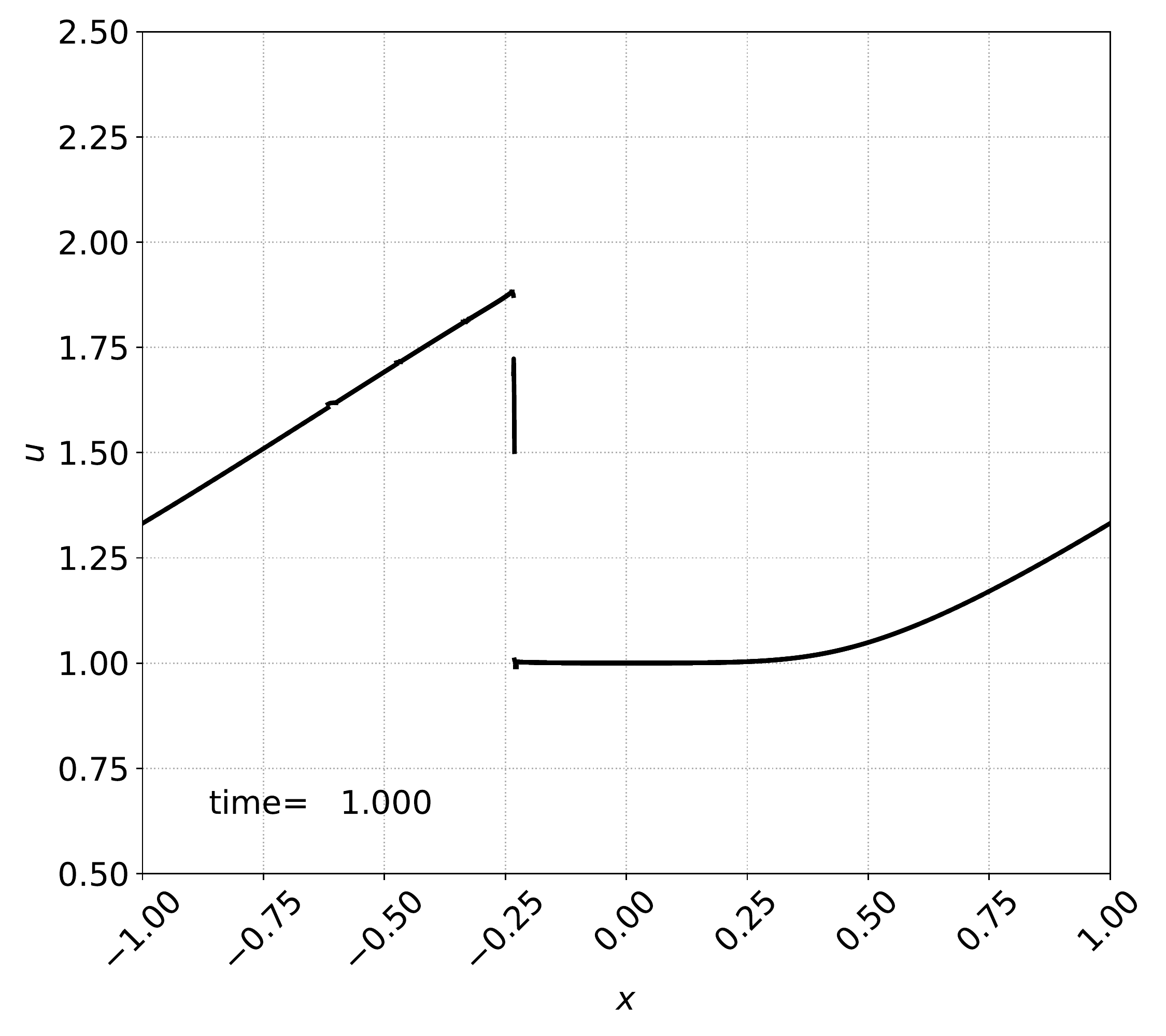}
    \caption{MRK2 w/ limiter}
    \figlab{burgers-lf-t1-mrk2-cs17}
  \end{subfigure} 
  \begin{subfigure}[snapshot]{0.45\linewidth}
    \includegraphics[trim=1.5cm 1.0cm 1.5cm 0.0cm, width=0.9\columnwidth]{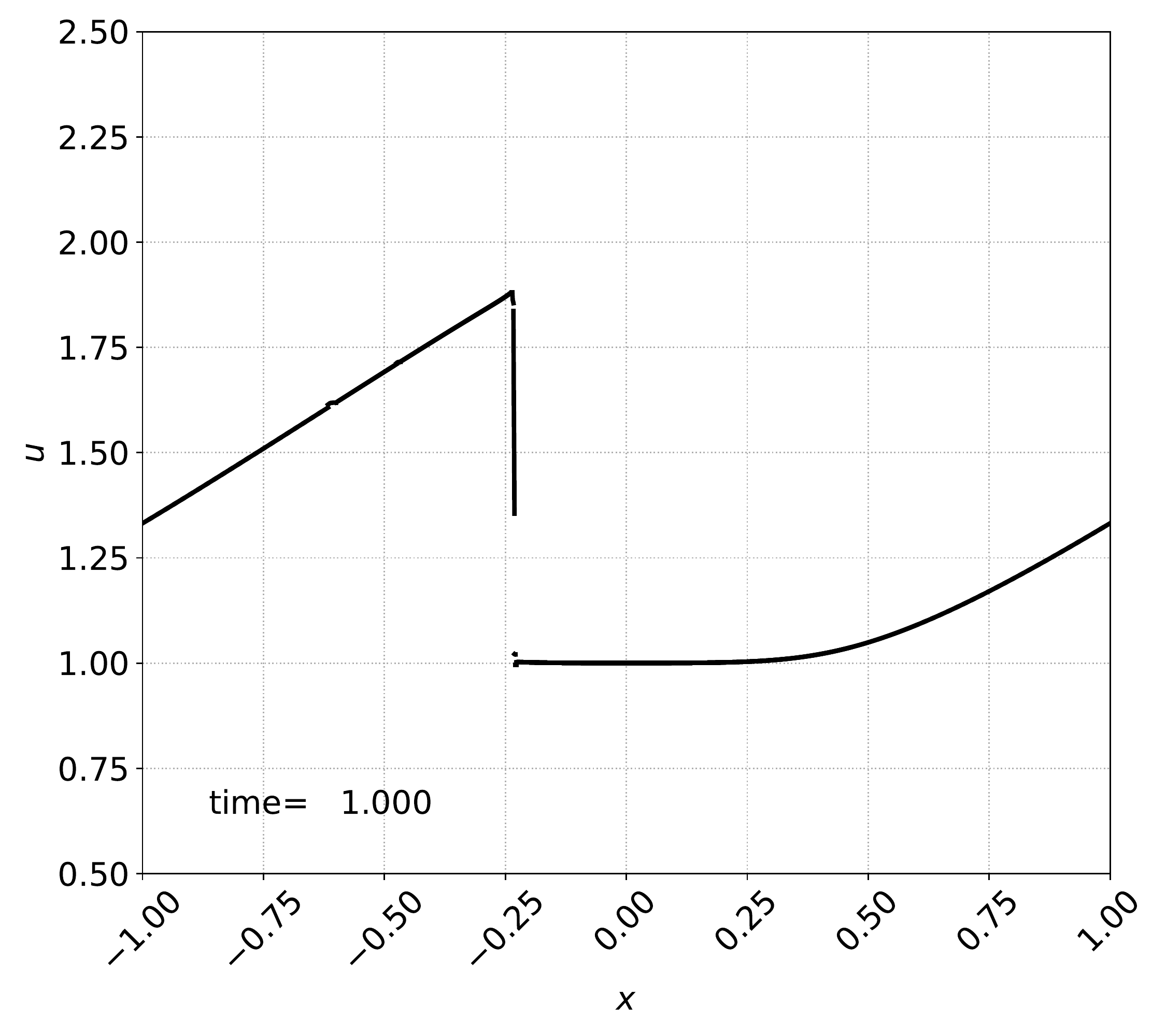}
    \caption{Relaxation-MRK2}
    \figlab{burgers-lf-t1-relaxed-mrk2}
  \end{subfigure} 
  \begin{subfigure}[snapshot]{0.45\linewidth}
    \includegraphics[trim=1.5cm 1.0cm 1.5cm 0.0cm, width=0.9\columnwidth]{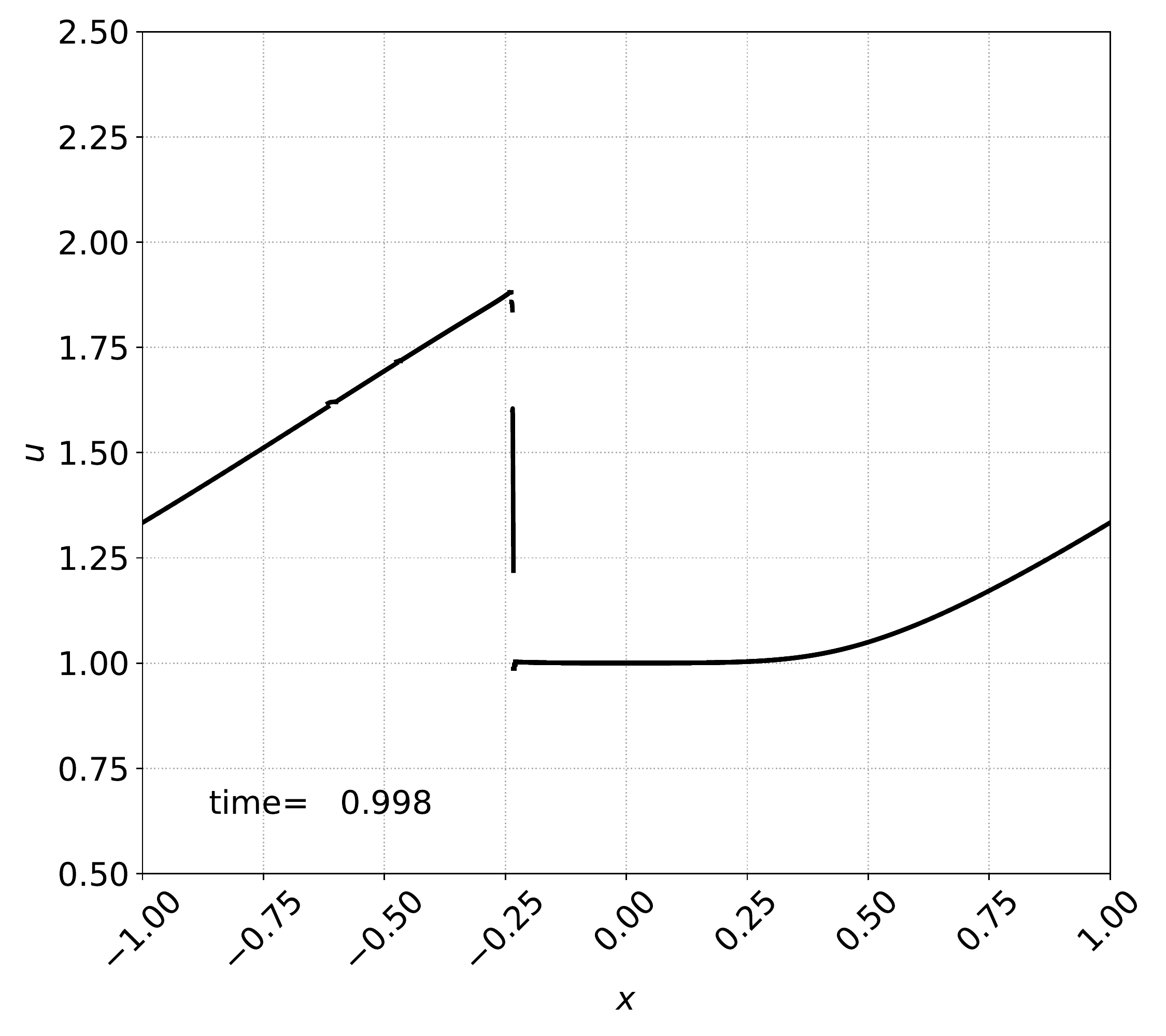}
    \caption{Relaxation-MRK2 w/ limiter}
    \figlab{burgers-lf-t1-relaxed-mrk2-cs17}
  \end{subfigure} 
  \begin{subfigure}[snapshot]{0.45\linewidth}
    \includegraphics[trim=1.5cm 1.0cm 1.5cm 0.0cm, width=0.9\columnwidth]{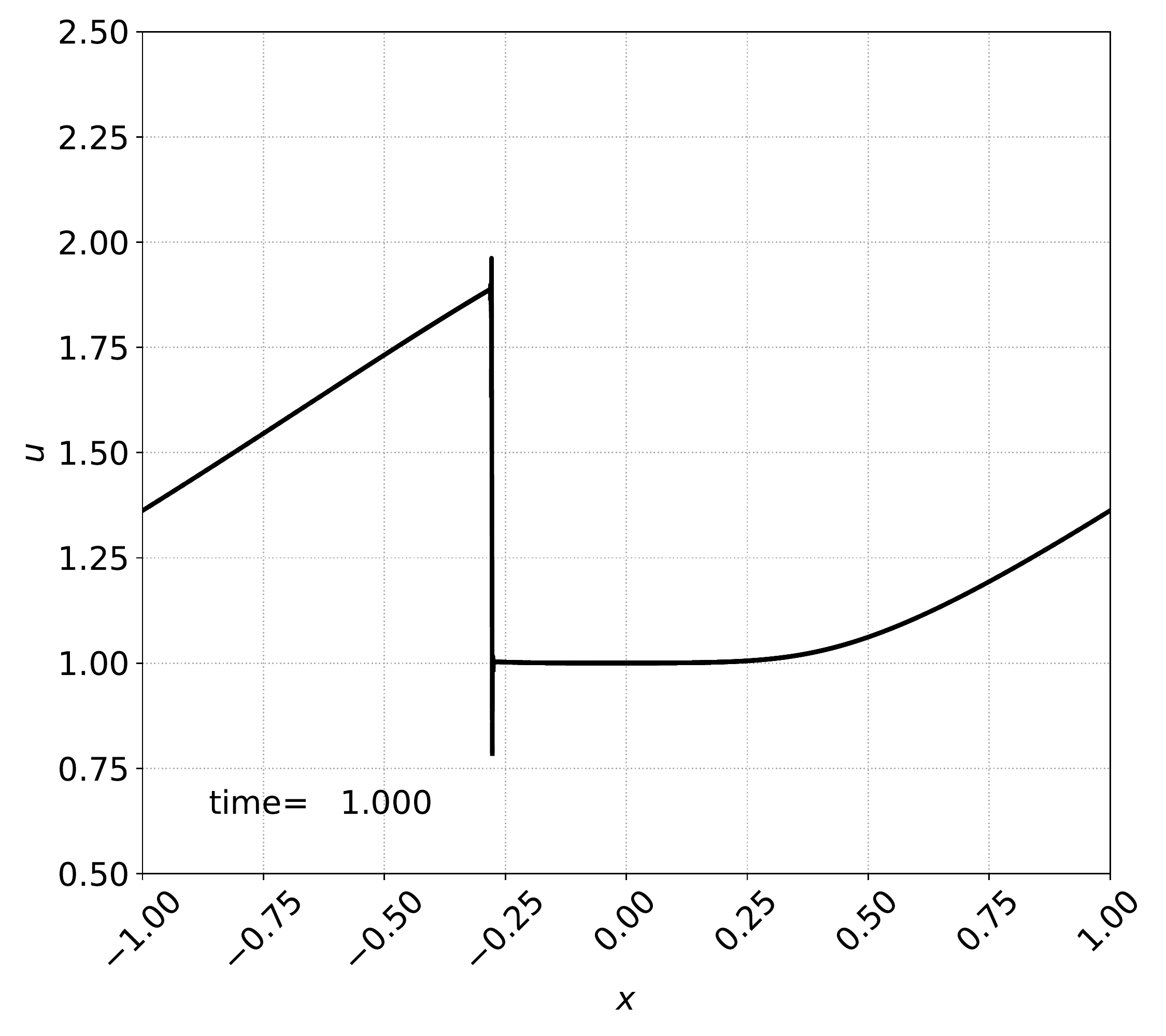}
    \caption{IDT-MRK2}
    \figlab{burgers-lf-t1-relaxed-mrk2-idt}
  \end{subfigure} 
  \begin{subfigure}[snapshot]{0.45\linewidth}
    \includegraphics[trim=1.5cm 1.0cm 1.5cm 0.0cm, width=0.9\columnwidth]{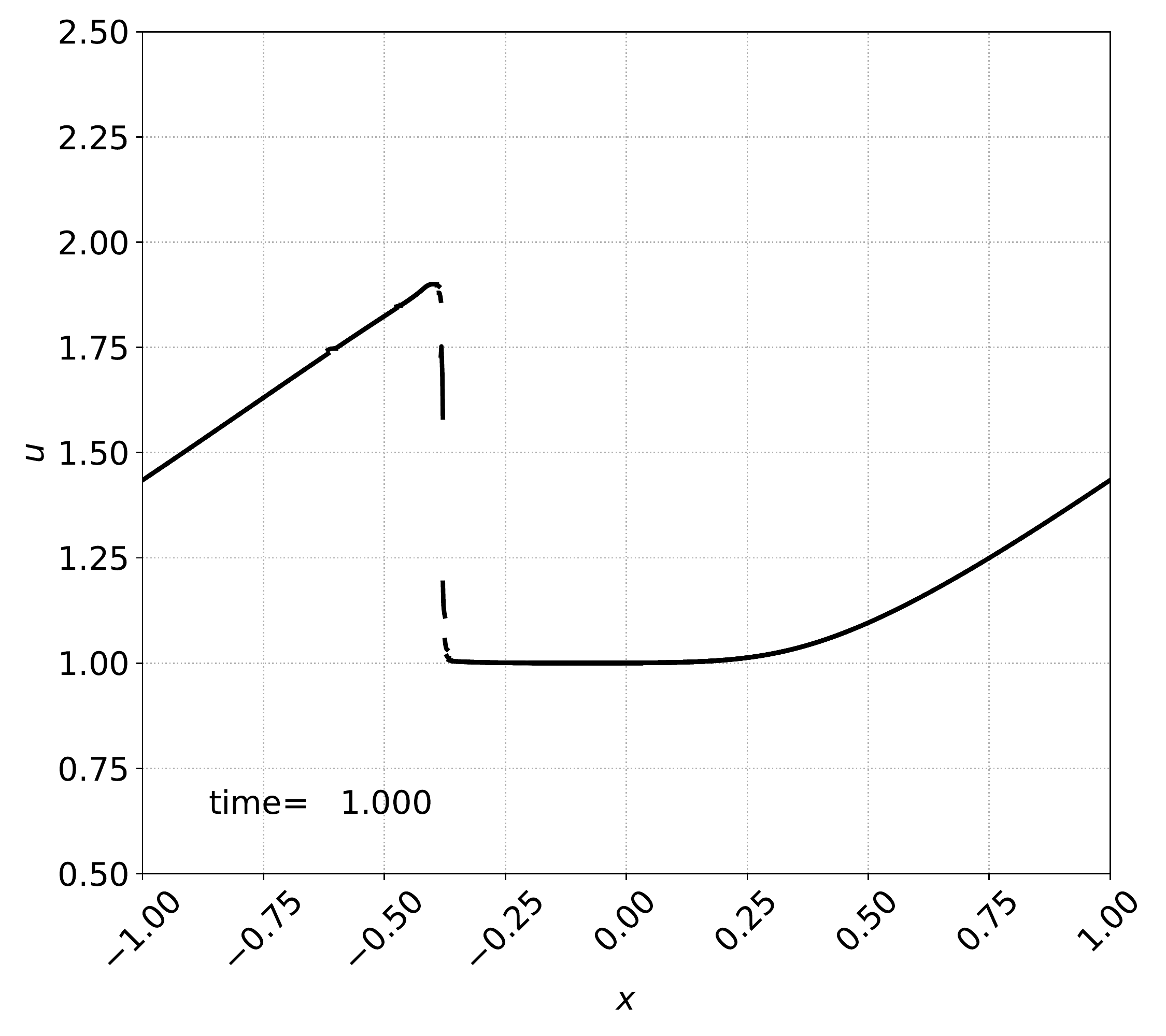}
    \caption{IDT-MRK2 w/ limiter}
    \figlab{burgers-lf-t1-relaxed-mrk2-cs17-idt}
  \end{subfigure} 
  \caption{Snapshots of Gaussian profile for the Burgers equation at $t=1$ on a nonuniform mesh with entropy-stable (ES) flux 
  for (a) RK2, (b) RK2 with limiter,
   (c) MRK2, (d) MRK2 with limiter, 
   (e) Relaxation-MRK2, (f) Relaxation-MRK2 with limiter, 
   (g) IDT-MRK2, and (h) IDT-MRK2 with limiter.
  The time step size of RK2 is taken as $\dt_{RK}=5 \times 10^{-5}$, 
  whereas those of MRK2, Relaxation-MRK2, and IDT-MRK2 have $\dt=25\times \dt_{RK}$.
  The domain is discretized with a nonuniform mesh of $N=3$ and $K=784$ ($L_{\max}=5$). 
  }
  \figlab{pde-burgers-lf-ss-mrk2-comparison}
\end{figure}
Figure \figref{pde-burgers-lf-totalenergyhistory-mrk2} shows the time histories of the total entropy and its difference
for the RK2, MRK2, Relaxation-RK2, Relaxation-MRK2, IDT-RK2, and IDT-MRK2 methods 
with/without the limiter. 
All the methods show entropy-stable behaviors.
The entropy differences of all methods reach $\mc{O}(10^{-1})$ as time passes. 
\begin{figure} 
  \centering
  \begin{subfigure}[Total energy history (ES)]{0.45\linewidth}
      \includegraphics[trim=0.2cm 0.2cm 0.2cm 0.2cm, width=0.95\columnwidth]{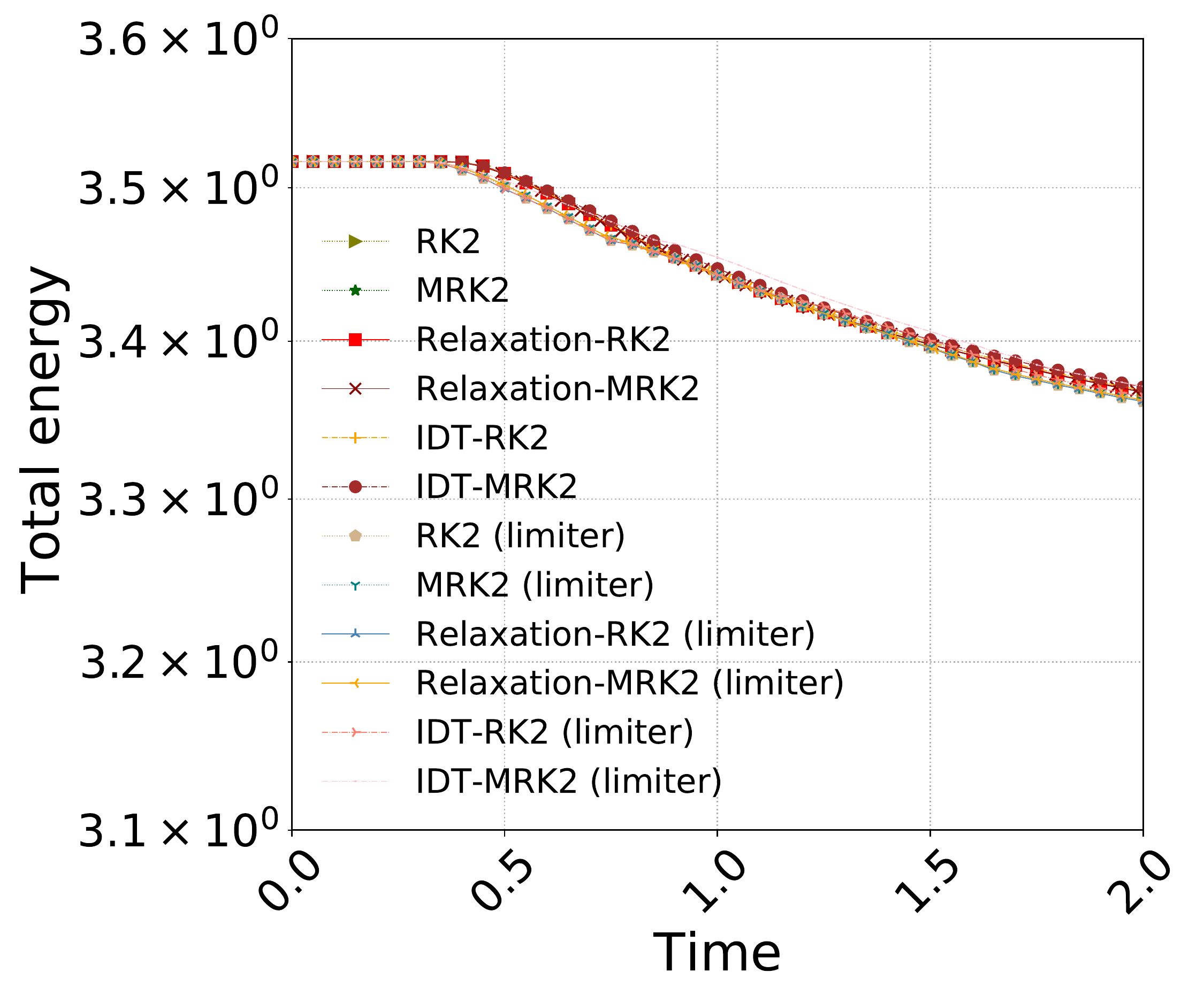}
      \caption{Total energy history (ES)}
      \figlab{pde-burgers-lf-energyloss-mrk2}
  \end{subfigure} %
  \begin{subfigure}[Total energy difference (ES)]{0.45\linewidth}    
      \includegraphics[trim=0.2cm 0.2cm 0.2cm 0.2cm, width=0.95\columnwidth]{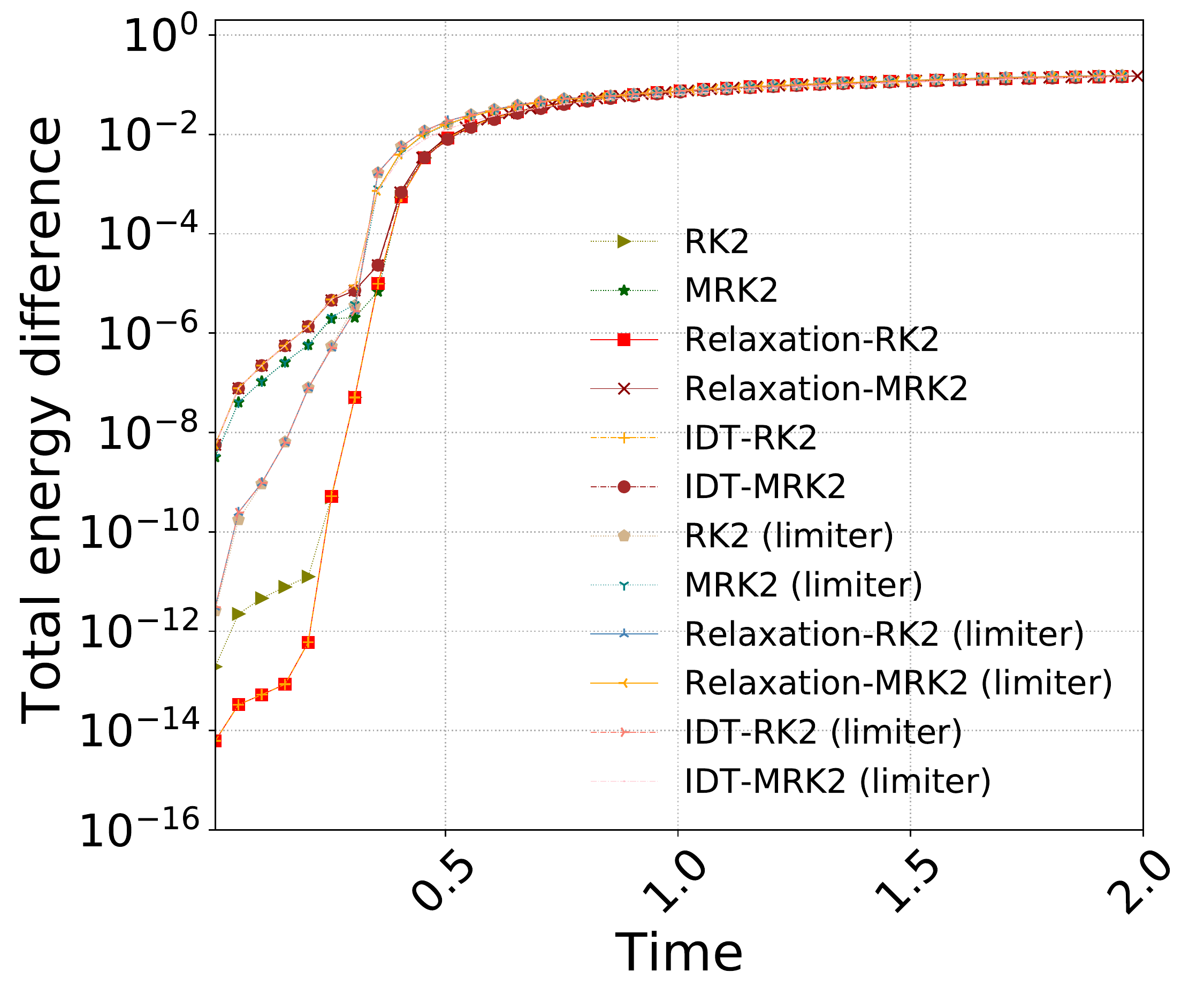}
      \caption{Total energy difference (ES)}
      \figlab{pde-burgers-lf-energyhistory-mrk2}
  \end{subfigure} 
  \caption{Histories of total entropy and its difference of Gaussian example for the Burgers equation with ES flux: 
      all the methods with ES flux show entropy-stable behaviors.
       }
   \figlab{pde-burgers-lf-totalenergyhistory-mrk2}
\end{figure}
In Figure \figref{pde-burgers-lf-totamassloss-mrk2} the time history of the total mass difference is shown 
for the RK2, MRK2, Relaxation-RK2, Relaxation-MRK2, IDT-RK2, and IDT-MRK2 methods with/without the limiter. In general, all the methods demonstrate good total mass conservation. 
In particular, without the limiter, all the methods preserve the total mass within $\mc{O}(10^{-14})$ error.
With the limiter, however, the total mass difference is bounded 
by $\mc{O}(10^{-12})$ for the RK2, Relaxation-RK2, and IDT-RK2 methods 
and by $\mc{O}(10^{-13})$ for the MRK2, Relaxation-MRK2, and IDT-MRK2 methods. 

\begin{figure} 
  \centering
  \includegraphics[trim=0.2cm 0.2cm 0.2cm 0.2cm, width=0.95\columnwidth]{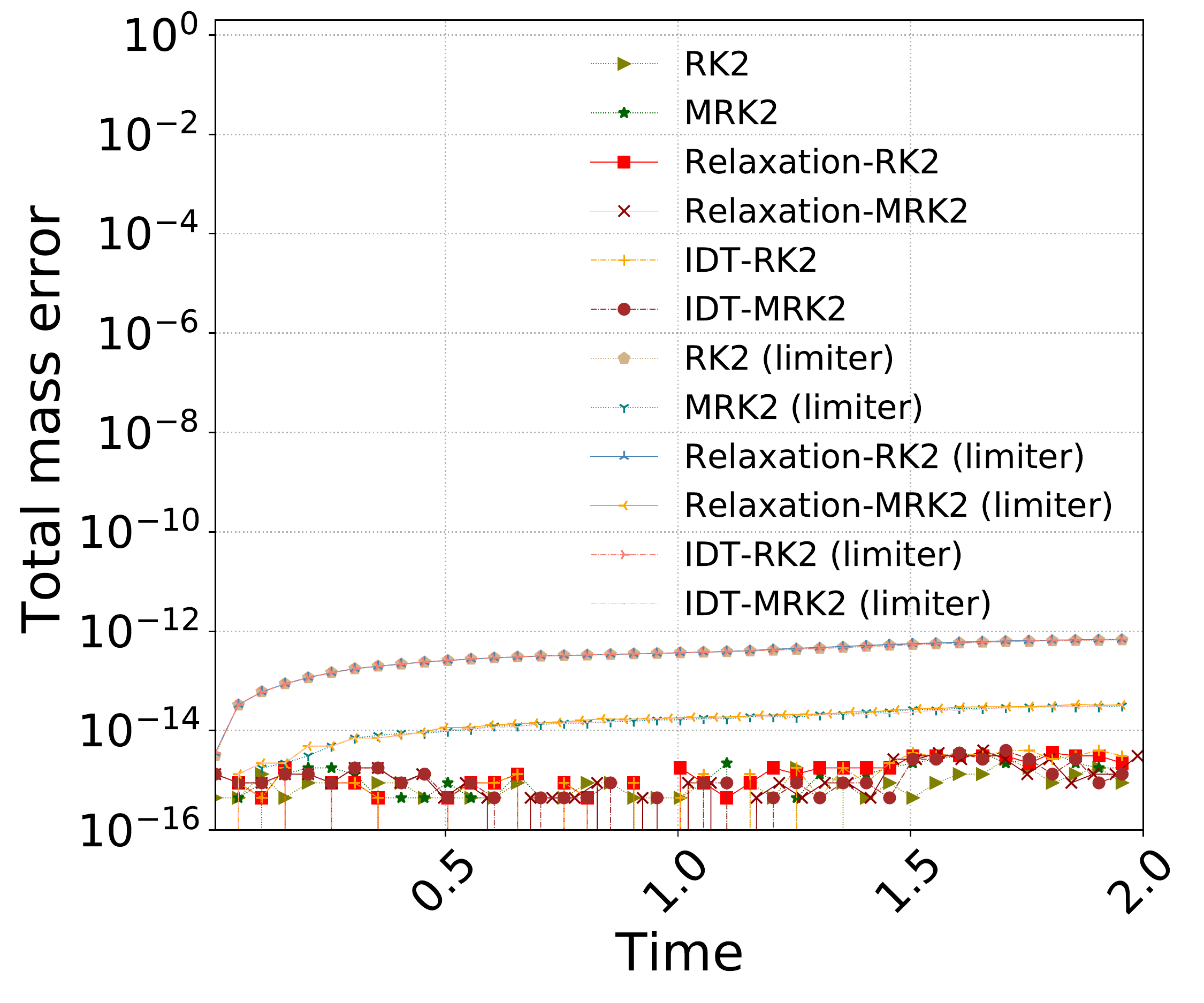}
  \caption{Histories of total mass difference of Gaussian example for the Burgers equation with ES flux: 
  without the limiter, the methods preserve the total mass within $\mc{O}(10^{-14})$.
However, the limiting procedure somehow affects the mass difference quantity, but the methods are bounded 
by $\mc{O}(10^{-12})$}
  \figlab{pde-burgers-lf-totamassloss-mrk2}
\end{figure}

  We note that the relaxation approach in \eqnref{mrk2-multilevel} is ``global.'' The entropy conservation/stability in time is imposed only at the synchronization time for all steps, which corresponds to the coarsest time level. If numerical instability occurs during the stage integration of the multirate method,
then the instability can lead to unstable numerical solutions. 
For this reason, when a shock occurs, we recommend 
using entropy-stable flux rather than entropy-conserving flux  
because the diffusive penalty term in entropy-stable flux 
helps  mitigate the numerical instability. 
Indeed, we numerically observed that the relaxation approach is stable with entropy-stable flux on deeply  nested mesh refinement.
We perform a numerical simulation for $t\in [0,1]$ with $\dt=0.002$.
The computational domain is non-uniformly refined with $L_{\max}=10$, $N=3$, and $N_E=3099$. 
Figure \figref{pde-burgers-lev10-lf-ss-mrk2} shows the snapshot of Gaussian example at $t=1$ without the limiter.
The shock front is highly resolved thanks to the fine resolution, and hence sharp spikes at the shock front are reduced, 
compared with Figure \figref{burgers-lf-t1-relaxed-mrk2}. 
%

\begin{figure} 
  \centering  
  \begin{subfigure}[snapshot]{0.45\linewidth}
    \includegraphics[trim=0cm 1.0cm 0cm 0.5cm, width=0.9\columnwidth]{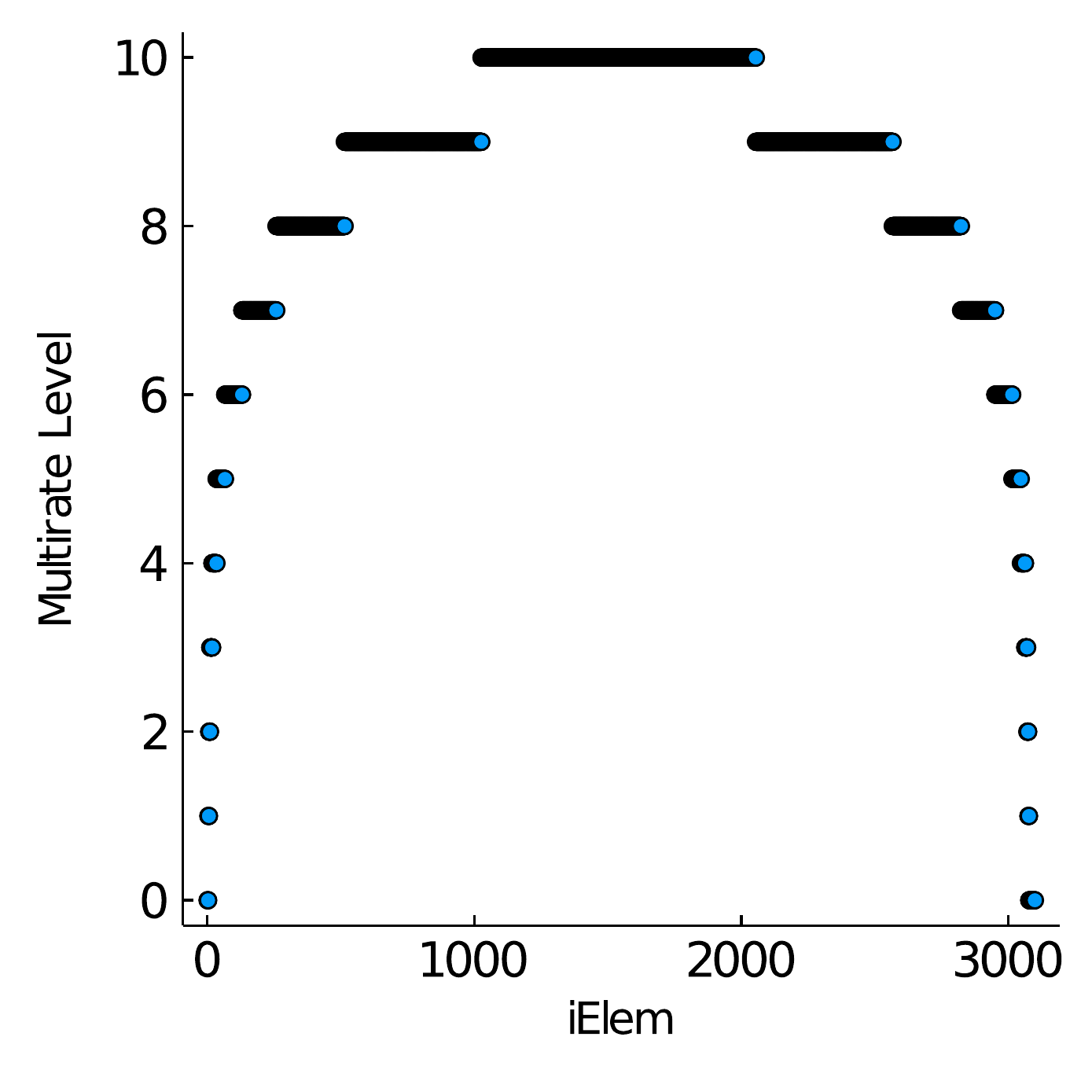}
    \caption{MR Levels}
    \figlab{burgers-lev10-mr}
  \end{subfigure} %
  \begin{subfigure}[snapshot]{0.45\linewidth}
    \includegraphics[trim=1.5cm 1.0cm 1.5cm 1.0cm, width=0.9\columnwidth]{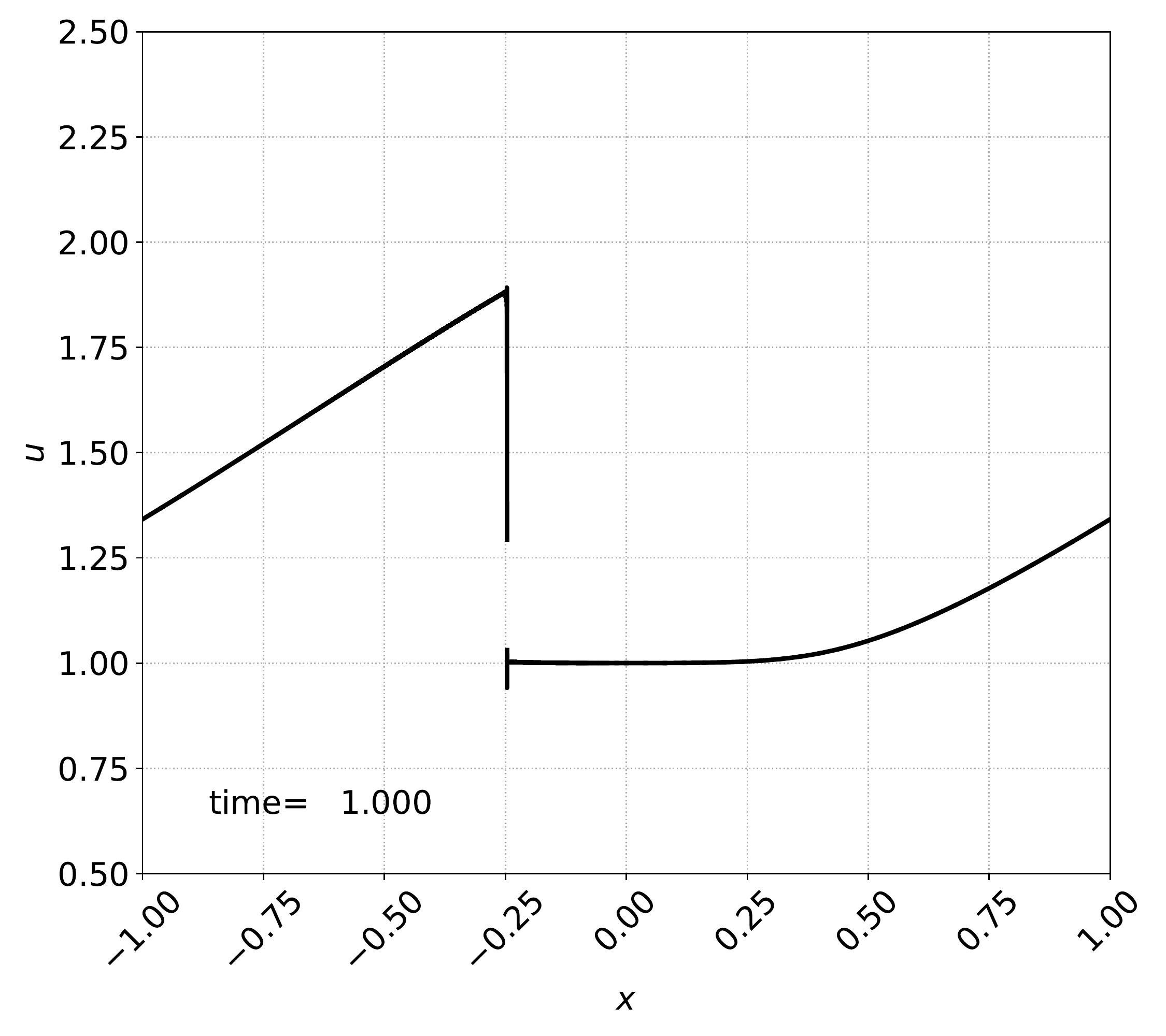}
    \caption{Relaxation-MRK2}
    \figlab{burgers-lev10-relaxation-mrk2}
  \end{subfigure} 
  \caption{(a) Multirate levels of elements for $L_{\max}=10$ 
  and (b) snapshot of Gaussian profile for the Burgers equation at $t=1$ on a nonuniform mesh with entropy-stable (ES) flux 
  for Relaxation-MRK2.
  The domain is discretized with a nonuniform mesh of $N=3$ and $K=3099$. 
  }
  \figlab{pde-burgers-lev10-lf-ss-mrk2}
\end{figure}

\section{Conclusions}
\seclab{Conclusion}

In this paper we present entropy-preserving/stable time discretization methods for partitioned Runge--Kutta schemes. 
Our work is an extension of the explicit relaxation Runge--Kutta methods \cite{ketcheson2019relaxation,ranocha2020general} 
to partitioned Runge--Kutta methods.  
In particular, we use the relaxation method to IMEX--RK methods and to a class of explicit second-order multirate methods.
IMEX-RK methods allow for a longer time step size than that restricted by explicit methods by defining the linearized flux containing the fast wave in the system
 with the stiffness being implicitly treated. Multirate methods decompose the original problem into subproblems,
 where different time step sizes can be used locally on each subproblem. 
 Unlike IMEX-RK methods, multirate methods do not require any linear/nonlinear solve and, hence, are attractive for parallel computing if proper preconditioning is not available. 
In combination with entropy conservation/stable spatial discretization, the proposed method successfully demonstrates
the entropy conservation and stability properties 
for a few ODEs and the Burgers equation. 

We numerically found that Relaxation-ARK approaches provide high-order accuracy in time, 
whereas the Relaxation-MRK2 method has a second-order rate of convergence, as expected. 
We also observed that the relaxation approach is one degree more accurate
 than the incremental direction technique when enough temporal errors have accumulated. 
 The location error of the incremental direction technique is larger than the relaxation strategy, especially in the presence of shocks. 
 When the limiter is used, the inaccuracy becomes substantially worse.
However, regardless of whether or not the limiter is applied, 
all the Relaxation-ARK, Relaxation-MRK2, IDT-ARK, IDT-MRK2 methods
show entropy-conserving/stable behavior for the Burgers equation.



The key idea of the relaxation method is to adjust the step completion with the relaxation parameter so that 
the time-adjusted solution satisfies entropy conservation and stability properties. 
The relaxation parameter is computed by solving a scalar nonlinear equation in general at each timestep; 
but, as for energy entropy, the relaxation parameter can be determined explicitly.
We theoretically provided the explicit forms of the relaxation parameters for IMEX-RK methods and the multirate methods 
and numerically verified that the explicit relaxation parameters work for the Burgers equation. 

We note that entropy conservation/stability in time is guaranteed only at the coarsest time level.
Numerical solutions may become unstable if numerical instability arises during the stage integration of the IMEX or multirate methods.
Because of the implicit correction step at each stage,
Relaxation-ARK approaches can reduce numerical instability. 
Relaxation-MRK2, on the other hand, lacks the ability to manage instability during stage integration. 
Therefore, with Relaxation-MRK2, 
entropy-stable flux is preferred above entropy-conserving flux, 
especially on deep-nested mesh refinement.
We showed that Relaxation-MRK2 with entropy-stable flux performs well on the deep-nested mesh refinement (with 10 levels).

To exploit more sophisticated problems, 
we will focus our future work  on extension to multidimensions 
as well as additional partial differential equations, such as Euler equations.
Working on entropy-conserving/entropy-stable coupling techniques for multiphysics problems is also interesting.

\appendix

\section*{Acknowledgments}
This material is based upon  work  supported by the U.S. Department of Energy, Office of Science, Office of Advanced Scientific Computing Research (ASCR) and Office of Biological and Environmental Research (BER), Scientific Discovery through Advanced Computing (SciDAC) program under Contract DE-AC02-06CH11357 
through the Coupling Approaches for Next-Generation Architectures (CANGA) Project and ASCR Base Program.

\section*{Declaration}


\subsection*{Availability of data and material}





 
The datasets generated during and/or analyzed during the current study are available from the corresponding author on reasonable request.

\subsection*{Code availability}
The code used to generate the results is available from the corresponding author on reasonable request.

\bibliography{main}

 \begin{center}
	\scriptsize \framebox{\parbox{4in}{Government License (will be removed at publication):
			The submitted manuscript has been created by UChicago Argonne, LLC,
			Operator of Argonne National Laboratory (``Argonne").  Argonne, a
			U.S. Department of Energy Office of Science laboratory, is operated
			under Contract No. DE-AC02-06CH11357.  The U.S. Government retains for
			itself, and others acting on its behalf, a paid-up nonexclusive,
			irrevocable worldwide license in said article to reproduce, prepare
			derivative works, distribute copies to the public, and perform
			publicly and display publicly, by or on behalf of the Government.
			The Department of Energy will provide public access to these results of federally sponsored research in accordance with the DOE Public Access Plan.
			http://energy.gov/downloads/doe-public-access-plan.
}}
	\normalsize
\end{center}

\end{document}